\documentclass[10pt]{amsart}

\usepackage{amsfonts}
\usepackage{amsthm}
\usepackage{amsmath}
\usepackage{amssymb}
\usepackage{graphicx}
\usepackage{subfigure}
\usepackage{xcolor,ulem}
\usepackage{bbm}
\usepackage{hyperref}
\usepackage{tikz}
\usepackage{mathrsfs}

\newtheorem{prop}{Proposition}[section]
\newtheorem{rules}[prop]{Rules List}
\newtheorem{lemm}[prop]{Lemma}
\newtheorem{coro}[prop]{Corollary}
\newtheorem{thm}[prop]{Theorem}

\newtheorem*{ques}{Question}
\theoremstyle{definition}
\newtheorem{nota}[prop]{Notation}
\newtheorem{defn}[prop]{Definition}
\newtheorem{rema}[prop]{Remark}
\newtheorem{exam}[prop]{Example}

\newcommand{\AAA}{\mathcal{A}}
\newcommand{\BBB}{\mathcal{B}}
\newcommand{\CCC}{\mathcal{C}}
\newcommand{\DDD}{\mathcal{D}}
\newcommand{\EEE}{\mathcal{E}}
\newcommand{\III}{\mathcal{I}}

\newcommand{\KKK}{\mathcal{K}}
\newcommand{\LLL}{\mathcal{L}}
\newcommand{\MMM}{\mathcal{M}}

\newcommand{\PPP}{\mathcal{P}}

\newcommand{\SSS}{\mathcal{S}}
\newcommand{\TTT}{\mathcal{T}}

\newcommand{\VVV}{\mathcal{V}}
\newcommand{\WWW}{\mathcal{W}}
\newcommand{\XXX}{\mathcal{X}}

\newcommand{\NN}{\mathbb{N}}
\newcommand{\ZZ}{\mathbb{Z}}

\newcommand{\GAP}{\vspace{.05in}}

\newcommand{\eps}{\varepsilon}

\definecolor{jcolor}{rgb}{0.5,0,0.6}
\definecolor{jcolor2}{rgb}{1.0,0,0.6}
\definecolor{jcolor3}{rgb}{0.5,0.2,1.0}

\title[Ergodic measures for subshifts]{The number of ergodic measures for transitive subshifts under the regular bispecial condition}
\date{\today}

\author[M.\ Damron]{Michael Damron}
\email{mdamron6@math.gatech.edu}
\author[J.\ Fickenscher]{Jon Fickenscher}
\email{jonfick@princeton.edu}

\begin{document}

\begin{abstract}
If $\mathcal{A}$ is a finite set (alphabet), the shift dynamical system consists of the space $\mathcal{A}^{\mathbb{N}}$ of sequences with entries in $\mathcal{A}$, along with the left shift operator $S$. Closed $S$-invariant subsets are called subshifts and arise naturally as encodings of other systems. In this paper, we study the number of ergodic measures for transitive subshifts under a condition (``regular bispecial condition'') on the possible extensions of words in the associated language. Our main result shows that under this condition, the subshift can support at most $\frac{K+1}{2}$ ergodic measures, where $K$ is the limiting value of $p(n+1)-p(n)$, and $p$ is the complexity function of the language. As a consequence, we answer a question of Boshernitzan from `84, providing a combinatorial proof for the bound on the number of ergodic measures for interval exchange transformations.
\end{abstract}

\maketitle

\tableofcontents

\section{Introduction}

If $\NN = \{1,2,3,\dots\}$ is the set of positive integers and
	$\AAA$ is a finite alphabet of symbols, we let
	$$
		\AAA^\NN = \{x = x_1x_2x_3\cdots:~x_i\in \AAA~\forall i\in \NN\}
	$$
denote the set of all one-sided infinite sequences with letters in $\AAA$.
When endowed with the natural product topology, this
	set becomes a compact metric space with its Borel $\sigma$-algebra.
The left shift map $S : \AAA^\NN \to \AAA^\NN$ given by
	$$
		(Sx)_i = x_{i+1} \mbox{ for all }i\in \NN
	$$
is a continuous function and creates the natural dynamical system, the \emph{shift} $(\AAA^\NN,S)$.
We will consider \emph{subshifts} $X\subseteq \AAA^\NN$, meaning
	closed $S$-invariant subsets of $\AAA^\NN$,
and assume that the restricted system $(X,S)$ is \emph{transitive},
	meaning that for each pair of non-empty open $U,V\subseteq X$
	there is an iterate $n\in \NN$ such that $S^n(U) \cap V$
	is non-empty.
	
Let $\AAA^* = \bigcup_{n\in \NN} \AAA^n$ denote the
	set of all finite \emph{words} with letters in $\AAA$.
We may naturally associate to a subshift $X$ its \emph{language}
	$$
		\LLL(X) = \{w\in \AAA^*:~w=x_ix_{i+1} \dots x_{i+|w|-1}\mbox{ for some } i\in \NN~\&~x\in X\},
	$$
	meaning the collection of all finite words
	that occur as a subword within a sequence $x\in X$.
Here $|w|$ is the length of $w$, meaning $|w| =n$ if $w = w_1w_2\dots w_n$.

There exists a broad collection of literature that relates combinatorial properties of
	$\LLL(X)$ with dynamical properties of $(X,S)$.
For example, under the stronger assumption that $(X,S)$ is minimal\footnote{For every $x\in X$, its orbit $\{x,Sx,S^2x,\dots\}$
	is dense, or equivalently any closed $S$-invariant subset of $X$ is either empty or $X$.}
	Boshernitzan \cite{cBosh} bounded the size of $\EEE(X,S)$, the set of $S$-ergodic probability measures
		on $X$,
	by considering the growth rate of the complexity function
		$$
			p_X(n) = \big|\{w\in \LLL:~|w| = n\}\big|.
		$$
In particular, he showed that if $\liminf_{n\to\infty} \frac{p_X(n)}{n} = \alpha \geq 1$
	then $|\EEE(X,S)| \leq \alpha$.
Furthermore, if $\limsup_{n\to\infty} \frac{p_X(n)}{n} = \alpha \geq 2$ then $|\EEE(X,S)| \leq \alpha - 1$.
We note that $|\EEE(X,S)| \geq 1$ due to the Krylov-Bogolyubov Theorem \cite{cKrylov} and if $\alpha < 1$ in either limit then 
	$(X,S)$ is periodic as a consequence of the Morse-Hedlund Theorem \cite{cMorseHedlund}.
More recently in \cite{cCyrKra}, Cyr and Kra extended these results to a broader class of dynamical systems
		and proved the same bound for the larger class of \emph{generic measures}.
Furthermore, they showed that the original bounds by Boshernitzan are sharp.

A class of systems discussed as motivation in \cite{cBosh} were subshifts that
	arise naturally from \emph{interval exchange transformations}.
These are piece-wise isometries from $I = [0,1)$ to itself defined
	by dividing the interval into $d$ subintervals of prescribed lengths
		$\lambda_1,\dots,\lambda_d$
	and then rearranging them by translation according to a permutation
		$\pi$ over $\{1,\dots,d\}$.
If we label the initial subintervals as $I_1,\dots, I_d$,
	the subshift associated to such an interval exchange $f$ is created by
	the closure of the image of the map
	\begin{equation}\label{eq: encoding}
		[0,1) \ni z \mapsto x\in \{1,\dots,d\}^\NN,
		\mbox{ where }
		 x_i = j \iff f^{i-1} z\in I_j.
	\end{equation}
It was shown independently by Katok \cite{cKatokMeasures}
	and Veech \cite{cVeechMeas} that for a minimal interval exchange $f$
	over $d$ subintervals we have
	$$
		|\EEE(I,f)| \leq \frac{d}{2}.
	$$
We note here that the bound is actually more precise and based on
	data\footnote{The \emph{genus} of a suspension of $f$. Such a surface may be realized by the
		zippered rectangle construction defined by Veech \cite{cVeechZip}.} obtained from $\pi$ that defines $f$,
but in many cases this bound agrees with $\lfloor d/2\rfloor$.

Under the natural encoding in \eqref{eq: encoding}, a minimal intevral exchange becomes a minimal subshift $(X,S)$
	and the complexity function $p_X(n)$ satisfies
	$$
		p_X(n+1) - p_X(n) \leq d-1 \mbox{ for all } n\in \NN,
	$$
and in fact under a general assumption\footnote{The infinite distinct orbit condition, also called the Keane condition. See \cite{cKeaneIDOC}.},
	equality holds above for all $n$.
This observation led to the natural question:

\begin{ques}[Boshernitzan \cite{cBosh}]
Can the bound $d/2$ for $|\EEE(X,S)|$ be shown combinatorially 
 for $(X,S)$ obtained from a minimal interval exchange transformation?
\end{ques}

In part, this question was presented as motivation for the work
	in \cite{cBosh} as the main results prove that for a general interval exchange
	the bound $|\EEE(X,S)| \leq d-2$ holds because $\lim_{n\to \infty} \frac{p_X(n)}{n} = d-1$.
So for $d = 2,3,4$ this bound agrees with the previously established bound.
This problem has been considered and some work has been done to
	extend Boshernistzan's result.
Before Cyr and Kra's work \cite{cCyrKra},
	Ferenczi and Monteil \cite[Theorem 7.3.7]{cFerenMont2010}
	expanded the bounds to ``$K$-deconnectable''
		systems.
		
In \cite{cDamFick2016}, the authors showed that $|\EEE(X,S)| \leq K - 2$
	for minimal
	$(X,S)$ satisfying
	\begin{equation}\label{eq: ECG_intro}
	p_X(n+1) - p_X(n ) = K \mbox{ for all large }n\in \NN
	\end{equation}
	where $K \geq 4$.
Because $K = d-1$ for interval exchanges, this answered Boshernitzan's
	question for $d = 5,6$.
	
To provide a proof for all $d$, we will in this paper make a further assumption on
	the language $\LLL$ if $(X,S)$.
A word $w\in \LLL$ is \emph{left special} if
	there are distinct letters $a,a' \in \AAA$
	so that $aw$ and $a'w$ exist in the language.
Likewise, $w$ is \emph{right special} if $wb$ and $wb'$ exist in the language
	for distinct letters $b,b'$.
A word is \emph{bispecial} if it is both left and right special.
A bispecial word is \emph{regular bispecial} if only one left extension of $w$ is
	right special and only one right extension of $w$ is left special.
	$\LLL$ (or equivalently $(X,S)$)
	satisfies the \emph{regular bispecial condition (RBC)}
	if all large enough bispecial words are regular.
We note here that RBC implies the constant growth condition
	from \eqref{eq: ECG_intro} above for some $K$.
Furthermore, all subshifts that arise from interval exchanges
	satisfy this property by work of Ferenczi and Zamboni \cite{cRBCforIET}.
Finally, this property is stable in the following sense:
	using results from \cite{cDolcePerrin3} we have that
	RBC is closed under topological conjugacy,
	meaning if $(X,S)$ is a subshift satisfying RBC and $(Y,S)$ is a subshift
		with $\varphi:X \to Y$ is a homemorphism satisfying
		$\varphi \circ S = S \circ \varphi$ then $(Y,S)$ also satisfies RBC.

Our main result is the following:

\begin{thm}[Main Theorem]\label{THM_MAIN}
	Let $(X,S)$ be a transitive subshift satisfying RBC
		with growth constant $K$.
	Then $|\EEE(X,S)| \leq \frac{K+1}{2}$.
\end{thm}

Because the proof will be combinatorial in nature, this provides
	a complete answer to Boshernitzan's question for all $d$. The result is also strictly more general than one that would apply only to interval exchanges, since interval exchanges form a proper subset of those systems with RBC.
Indeed, linear involutions as discussed in \cite{cBoissyLanneau}
		arise as first returns of flows on half-translation surfaces
	and a generic involution satisfies RBC.
	However they are not equivalent to
		interval exchanges\footnote{There is a natural way to define an IET that is a two-fold covering of a given linear involution (see \cite{cAvilaResende} for example).
		However, the set of invariant measures for the IET and the linear involution do not necessarily agree.}	
	
In terms of Rauzy graphs, which are directed graphs derived from the language $\LLL$
	of a subshift,
	the works in \cite{cBosh} and \cite{cFerenMont2010} bounded
		$|\EEE(X,S)|$
	by identifying ``important''\footnote{In \cite{cBosh} these were the vertices related to right special words
			in the language while in \cite{cFerenMont2010} these vertices formed a ``$K$-deconnectable'' set.} vertices in these
	graphs to the measures in $\EEE(X,S)$.
In \cite{cDamFick2016}, the improvement essentially developed from adding a notion of dynamics
	to these graphs and tracing these identifications
	over time via a ``coloring function.''
By incorporating the RBC condition with the dynamics from \cite{cDamFick2016}, we will construct an auxiliary graph and relate the number of ergodic measures of $(X,S)$ to the connectedness of this graph.


\subsection{Outline of paper}

Section \ref{SEC_DEF} will provide definitions, including
	more conventional objects as well as
	some of the machinery from \cite{cDamFick2016}.
We will first introduce languages $\LLL$ over a finite alphabet
	as well as the subshift dynamical systems $(X,S)$ and
	recall standard results concerning their relationships.
We also relate RBC to a property defined by Dolce and Perrin in \cite{cDolcePerrin3}
	to show that transitive subshifts with RBC must in fact be minimal (Corollary \ref{COR_REC_IFF_UREC})
	and that RBC is closed under topological conjugation (Corollary \ref{COR_RBC_CONJ}).
Then, following the work from \cite{cDamFick2016},
	the subshift/language will be used to define a graph $\Lambda$
	with coloring rule $\CCC$.
The coloring rule is based on a type of density function
	that will play a key role in translating
	between combinatorics in $\LLL$ and
	ergodic measures for $(X,S)$.
	
In Section \ref{SEC_EXIT_WORDS},
	a notion of an exit word of a given word $w$ will be developed
	along with relevant combinatorial results.
The main problem will be described using
	graphs $\Lambda$ and coloring rules in $\CCC$
	in Section \ref{SEC_GRAPHS}.
These graphs
	and coloring rules will be given restrictions
	via Rules Lists.
To aid in understanding, these restrictions will
	be stated as given and the proofs will be given in Section \ref{SEC_ERG_TO_GRAPH}.	
As a main result, Proposition \ref{PROP_CONNECT} will be proven in the appendix and
	via Corollary \ref{cor: main_corollary}
	will relate the bound on $|\EEE(X,S)|$
	to the connected-ness of an auxiliary graph $\Xi$.

The discussion concerning density from Section~\ref{SEC_DEF} and exit words
	from Section \ref{SEC_EXIT_WORDS}
	will continue in Section \ref{SEC_ERG}.
Having already considered exit words from the word combinatorial perspective,
	we will in this section derive ergodic theoretic properties.
Lemma \ref{LemLoopWords} will help in bounding densities of
	of recurring words, while
Lemma \ref{LemExitDens} will establish mutual bounds between the density of
	$w$ and its exit words.
	
The Rules Lists from Section \ref{SEC_GRAPHS} will be justified in
	Section \ref{SEC_ERG_TO_GRAPH}.
Here will will use the results from previous sections to
	explicitly construct our graphs as
	well as prove the aforementioned graph and coloring rule
	restrictions.
	
The key result, Proposition \ref{PROP_CONNECT},
	concerning the auxiliary graph will be proven in the appendix.
The arguments used rely only on definitions and Rules Lists from Section \ref{SEC_GRAPHS}
	and so this appendix may be read independently from Sections \ref{SEC_ERG}
	and \ref{SEC_ERG_TO_GRAPH}.

	\subsection{Acknowledgments}
	The research of M. D. is supported by an NSF CAREER grant. 
	J.\ F.\ thanks M.\ Boshernitzan for posing his question and F.\ Dolce for his helpful conversations that lead to, among many other things,
		Corollaries \ref{COR_REC_IFF_UREC} and \ref{COR_RBC_CONJ}.
	J.\ F.\ is also thankful every day for Laine, Charlie and Amy.

\section{Definitions}\label{SEC_DEF}

\subsection{Languages on a Finite Alphabet}

Let $\AAA$ be a finite alphabet of symbols, let
	$$\AAA^n = \{w = w_1w_2\dots w_n: w_i\in \AAA \mbox{ for } 1\leq i \leq n\}$$
	be the set of words of length $n$ with letters in $\AAA$
	and let $\AAA^* = \bigcup_{n\in \NN} \AAA^n$ be the collection of all finite words with letters in $\AAA$.
	Note that in this paper the empty word is not considered part of $\AAA^*$ or
		any language (as will be defined).
	We denote by $|w|$ the length of a word $w\in \AAA^*$.
If $u,v\in \AAA^*$ then $uv$ represents the concatenation of $u$ followed by $v$ and for $k\in \NN$
	the word
	$$u^k = \underbrace{u \cdots u}_k$$
represents $k$ copies of $u$ concatenated together.
For $w\in \AAA^n$ and $1\leq i \leq j \leq n$ we call $w_{[i,j]} = w_i w_{i+1} \cdots w_{j}$
	the \emph{subword} of $w$ starting at position $i$ of length $j - i +1$.
It follows then that for $w\in \AAA^n$, $1\leq i \leq j \leq n$ and $1\leq k \leq m \leq j-i+1$
	\begin{equation}
		(w_{[i,j]})_{[k,m]} = w_{[i',j']} \mbox{ where }
			\left\{\begin{array}{l} i' = i + k - 1, \\ j' = i + m -1.\end{array}\right.
	\end{equation}
For any two words $w,u\in \AAA^*$ we let
	$$
		|w|_u = |\{1\leq j \leq |w| - |u| + 1:~ w_{[j,j+|u| - 1]} = u\}|
	$$
denote the \emph{number of occurrences} of $u$ as a subword in $w$.
	
\begin{defn}
	A \emph{language} $\LLL \subseteq \AAA^*$ is a collection of finite words so that
			\begin{enumerate}
				\item $\AAA\subseteq \LLL$.
				\item For all $w\in \LLL$ and $1\leq i \leq j \leq |w|$, $w_{[i,j]} \in \LLL$.
				\item For each $w\in \LLL$ there exist $a,b\in \AAA$ so that $awb\in \LLL$.
			\end{enumerate}
	For each $n\in \NN$, $\LLL_n = \AAA^n \cap \LLL$ is the collection of all words in $\LLL$ of length $n$.
\end{defn}

A language $\LLL\subseteq \AAA^*$ is \emph{recurrent} if for any $u,v\in \LLL$ there exists $w\in \LLL$
	so that $uwv\in \LLL$
	and $\LLL$ is \emph{uniformly recurrent} if for each $u$ there exists $N\in \NN$ so that $u$
	is a subword of each $w\in \LLL_N$.
A language $\LLL$ is \emph{periodic} if there exists $p\in \NN$ so that for each $w\in \LLL$
	we have $w_i = w_{i + p}$ for all $1\leq i \leq |w|-p$ and $\LLL$ is \emph{aperiodic} otherwise.

By definition, each $w\in \LLL$ may be extended to the left (resp. right) by at least one symbol in $\AAA$.
We say that $w\in \LLL$ is \emph{left special} if the set
	$$Ex^\ell(w;\LLL) := \{a\in \AAA: aw\in \LLL\},$$
	of left extensions contains at least two elements.
Likewise, $w$ is \emph{right special} if $|Ex^r(w;\LLL)|\geq 2$,
	where
	$$
		Ex^r(w;\LLL) := \{b\in \AAA:~wb\in \LLL\},
	$$
	denotes the set of right extensions.
	We will typically exclude $\LLL$ in the notation of left (resp. right)
		extensions when the language is understood.
Let $\LLL^{\ell}$ (resp. $\LLL^r$) denote the left (resp. right) special words in the language
	and let $\LLL^{\mathfrak{s}}_n = \LLL^{\mathfrak{s}} \cap \LLL_n$ for $(n,\mathfrak{s}) \in \NN \times\{\ell,r\}$.
As a convention, we will let $\ell$-special (resp. $r$-special) denote left special (resp. right special).

A word $w\in \LLL$ is \emph{bispecial} if it is both left and right special.
We call $w$ \emph{regular bispecial} if there are unique $\hat{a}\in Ex^\ell(w)$ and $\hat{b}\in Ex^r(w)$ such that
	$w\hat{b}$ is left special and $\hat{a}w$ is right special.
In other words, for each $a\in Ex^{\ell}(w)\setminus\{\hat{a}\}$ the word $aw$ is not right special and likewise for
	each $b\in Ex^r(w)\setminus\{\hat{b}\}$ and word $wb$.
	\begin{rema}
		This definition of regular bispecial is a generalization of \emph{ordinary bispecial} words
			defined in \cite{cCass} but for alphabets of any size.
	\end{rema}
	
\begin{defn}\label{def: RBC}
	A language $\LLL \subseteq \AAA^*$ satisfies the \emph{regular bispecial condition (RBC)}
		if for some $n_0 \in \NN$, all bispecial $w\in \LLL$ of length at least $n_0$
		are regular bispecial.
\end{defn}

The next result, Lemma \ref{LEM_RBC_UNIQUENESS}, provides three statements essential to the proofs that follow.
In the case of left-special words: the first part states that a prefix of a left-special word is left-special,
	the second part states that any large enough left-special word may be uniquely extended to the right
	within the language to form another left-special word
	and the third part states that the left extensions of large left-special words having a large enough common prefix
		must agree.
The parts also analogously apply to right-special words and suffixes.
The first part is well-known and applies to any language $\LLL$ while the last two parts
	require the RBC condition.
\begin{lemm}\label{LEM_RBC_UNIQUENESS}
	Let $\LLL\subseteq \AAA^*$ be a language.
	\begin{enumerate}
		\item For each $n_2 \geq n_1$, $\mathfrak{s}\in \{\ell,r\}$ and $w'\in \LLL_{n_2}^\mathfrak{s}$
			the word
				$$w = \begin{cases}
						w'_{[1,n_1]}, & \text{if }\mathfrak{s} = \ell,\\
						w'_{[n_2 - n_1 + 1,n_2]}, & \text{if }\mathfrak{s} = r,
					\end{cases}$$
				is $\mathfrak{s}$-special.
	\end{enumerate}
	Now assume also that $\LLL$ satisfies RBC and $n_0$ is from Definition \ref{def: RBC}.
	\begin{enumerate}
		\setcounter{enumi}{1}
		\item For each $n_2\geq n_1 \geq n_0$, $\mathfrak{s}\in \{\ell,r\}$ and $w\in \LLL_{n_1}^\mathfrak{s}$ there exists
			a unique $w'\in \LLL_{n_2}^\mathfrak{s}$ such that 
						$$w = \begin{cases}
						w'_{[1,n_1]}, & \text{if }\mathfrak{s} = \ell,\\
						w'_{[n_2 - n_1 + 1,n_2]}, & \text{if }\mathfrak{s} = r.
					\end{cases}$$

		\item If $n_2 \geq n_1$ are large enough and $\mathfrak{s}\in\{\ell,r\}$ then any words $w'\in \LLL_{n_1}^\mathfrak{s}$
			and $w''\in \LLL_{n_2}^\mathfrak{s}$ satisfying
				$$
					\begin{cases}
						w'_{[1,n_0]} = w''_{[1,n_0]}, & \text{if }\mathfrak{s} = \ell,\\
						w'_{[n_1 - n_0 + 1, n_1]} = w''_{[n_2-n_0+1,n_2]}, & \text{if }\mathfrak{s} = r,
					\end{cases}
				$$
			have the same $\mathfrak{s}$-extensions.
	\end{enumerate}
\end{lemm}

\begin{proof}
	We will prove the lemma for the case $\mathfrak{s} = \ell$, as the other case is handled in a nearly identical way.
	For part 1, if $aw' \in \LLL$ then $aw = (aw')_{[1,n_1+1]} \in \LLL$ as well.
	In particular, $Ex^{\ell}(w') \subseteq Ex^{\ell}(w)$, or in other words the set of left extensions of $w'$ is a subset of the left extensions for $w$.
	Therefore $w'\in \LLL_{n_2}^\ell$ implies that $w\in \LLL_{n_1}^\ell$.
	
	For part 2, note that by part 1 we may assume $n_1 = n_0$.
	Indeed, if $w$ is in, for example, $\mathcal{L}_{n_1}^\ell$,
		then $w_{[1,n_0]}$ is also left-special,
		and if it has an unique right extension to a left-special word in $\mathcal{L}_{n_2}$, then so does $w$.
	We then proceed by induction on $n_2$.
	The result holds for $n_2= n_0$ so assume the claim is true for values at most $n_2$ and 
		consider the value $n_2+1$.
	Let $\tilde{w}\in \LLL_{n_2}^\ell$ be the unique left special word of length $n_2$
		with prefix $w$.
	If $\tilde{w}$ is not bispecial there is a unique $b\in \AAA$ so that $w' = \tilde{w} b\in \LLL_{n_2+1}$.
	It follows that $w'_{[1,n_0]} = \tilde{w}_{[1,n_0]} = w$ and each left extension of $\tilde{w}$ must also
		be a left extension of $w'$.
	Therefore $w'$ is left special and is unique as $\tilde{w}$ is unique.	
	If instead $\tilde{w}$ is bispecial,
		we use RBC to let $\hat{b} \in Ex^{r}(w)$ be the unique element such that
		$w' = \tilde{w}\hat{b}$ be left special.
	Again $w=w'_{[1,n_0]}$ and must be unique by the uniqueness of $\tilde{w}$.
	
	We will now show part 3.
	By part 2, for any $w\in \LLL_{n_0}^\ell$ and $n\geq n_0$ there is a unique
		$u^{(n)} \in \LLL_n^\ell$ so that $(u^{(n)})_{[1,n_0]} = w$.
	As in the proof of part 1
		$$
			Ex^\ell(u^{(n)}) \supseteq Ex^\ell(u^{(n+1)}) \text{ for all } n \geq n_0.
		$$
	Therefore, there exists $N_w$ so that equality holds above for all $n \geq N_w$.
	The claim then holds for all $n$ at least $N = \max_{w\in \LLL_{n_0}^\ell}\{N_w\}$,
		as by parts 1 and 2 for any $\tilde{n} \geq N$ and $\tilde{w} \in \LLL_{\tilde{n}}^\ell$
		there is a unique $w\in \LLL_{n_0}^\ell$ so that $\tilde{w} = u^{(\tilde{n})}$
		as defined above for $w$.
\end{proof}

For a language we define the \emph{complexity function} $p(n) = p_\LLL(n) : \NN \to \NN$ by
	\begin{equation}
		p(n) = |\LLL_n|.
	\end{equation}
In other words $p(n)$ is the number of distinct words in $\LLL$ of length $n$.
If $\LLL$ is recurrent, the Morse-Hedlund Theorem \cite{cMorseHedlund} states that $\LLL$ is periodic if and only if
	$p(n_0) \leq n_0$ for some $n_0\in \NN$.
Equivalently, there exists $n_0\in \NN$ so that $p(n_0+ 1) - p(n_0) = 0$
	and in fact $p(n+1) - p(n) = 0$ for all $n\geq n_0$.
It follows that the aperiodic recurrent languages $\LLL$ must have complexity functions
	that satisfy $p(n) \geq n+1$.
The recurrent languages such that $p(n) = n +1$ for all $n\in\NN$ are the well studied
	Sturmian languages \cite{cSturmian}.
Such languages have constant (complexity) growth $1 = p(n+1) - p(n)$ for all $n\in\NN$.
The following definition from \cite{cDamFick2016} generalizes these observations.

\begin{defn}
	A language $\LLL\subseteq \AAA^*$ has \emph{eventually constant growth (ECG)}
		if there are constants $N_0\in \NN$ and $K,C\in \NN\cup\{0\}$ so that
		\begin{equation}
			p(n) = Kn + C \text{ for all } n \geq N_0,
		\end{equation}
	or equivalently for some constants $K,N_0\in \NN$,
		\begin{equation}
			p(n+1) - p(n) = K \text{ for all } n \geq N_0.
		\end{equation}
\end{defn}

For any $n\in \NN$, each $w'\in \LLL_{n+1}$ is uniquely defined by
	its length $n$ suffix $w = w'_{[2,n+1]}$ and its $\ell$-extension $a = w'_1$,
	and the case for $r$-extensions is analogous.
Therefore
	$$
		p(n + 1) = \sum_{w\in \LLL_n} |Ex^\mathfrak{s}(w)|
	$$
and so
	\begin{equation}\label{EQ_GROWTH_AND_EXTENSIONS}
		p(n+1) - p(n) = \sum_{w\in \LLL_n}( |Ex^\mathfrak{s}(w)| -1 ) = \sum_{w\in \LLL_n^\mathfrak{s}} (|Ex^\mathfrak{s}(w)| - 1).
	\end{equation}
	
\begin{lemm}\label{LEM_RBC_then_ECG}
	If language $\LLL\subseteq \AAA^*$ satisfies RBC then it has ECG.
\end{lemm}

\begin{proof}
	Let $n_0$ be from the definition of RBC and
		$\mathfrak{s}\in \{\ell,r\}$.
	Letting $\phi_n$ be the bijection from $\LLL_{n_0}^\mathfrak{s}$ to $\LLL_{n}^\mathfrak{s}$
		for any $n\geq n_0$ implicitly defined by parts 1 and 2 of Lemma \ref{LEM_RBC_UNIQUENESS},
		then
		by part 3 for some $N_0$ and any $n_2 \geq  n_1 \geq N_0$ we have that
			$$
				|Ex^\mathfrak{s}(\phi_{n_2}(w))|
					= 
				|Ex^\mathfrak{s}(\phi_{n_1}(w))|
				\text{ for all }w\in \LLL_{n_0}^\mathfrak{s}.
			$$
	For any such $n_1$ and $n_2$, it must be that
			$$
				\sum_{w'\in \LLL_{n_1}^\mathfrak{s}} (|Ex^\mathfrak{s}(w')| - 1)
					=
				\sum_{w''\in \LLL_{n_2}^\mathfrak{s}} (|Ex^\mathfrak{s}(w'')| - 1)
			$$
	and by \eqref{EQ_GROWTH_AND_EXTENSIONS}, $p(n+1) - p(n)$ is a constant $K$
		for all $n \geq N_0$.
	Therefore $\LLL$ has ECG as desired.
\end{proof}

Aside from the sets of (one-sided) extensions $Ex^\mathfrak{s}(w)$ for $w\in \LLL$ and $\mathfrak{s}\in \{\ell,r\}$,
	let
	$$
		Ex^{\ell r}(w;\LLL) = \{(a,b)\in \AAA^2:~awb\in \LLL\}
	$$
	be the set of two-sided extensions.
Adapting Dolce and Perrin's notation from \cite{cDolcePerrin2}, let
	$$
		m_\LLL(w) = |Ex^{\ell r}(w;\LLL)| - |Ex^\ell(w;\LLL)| - |Ex^r(w;\LLL)| + 1.
	$$
In that paper, the authors prove the next result.
\begin{lemm}[Theorem \cite{cDolcePerrin2}]
	If a language $\LLL$ is recurrent and for some $n_0$ we have $m_\LLL(w) = 0$ for all $w\in \LLL$ such that $|w| \geq n_0$, then 
		$\LLL$ is uniformly recurrent.
\end{lemm}

Dolce and Perrin called a language $\LLL$ \emph{eventually neutral} if $m_\LLL(w) =0$ for all large enough $w\in \LLL$.
Furthermore, in \cite{cDolcePerrin3} they define a stronger property.
To define this property, we first define the graph $\mathscr{E}(w)$ for $w\in \LLL$ as follows:
$\mathscr{E}(w)$ is a bipartite graph with vertices $Ex^\ell(w) \cup Ex^r(w)$
	and an edge $(a,b)$ if and only if $(a,b)\in Ex^{\ell r}(w)$.
They say that $\LLL$ is \emph{eventually dendric} if there exists $n_0$ such that
	for all $w\in \LLL$ satisfying $|w|\geq n_0$ the graph $\mathscr{E}(w)$ is a tree.

We note that a connected graph with $m$ vertices is a tree if and only if it has $m-1$ edges.
It then follows that if $\mathscr(E)(w)$ is a tree then
	$m_\LLL(w) = 0$ and so
	if $\LLL$ is eventually dendric then $\LLL$ is also eventually neutral.
	
In fact, we have the following relationship between this property and RBC.

\begin{lemm}
	Let $\LLL$ be a recurrent language.
	Then $\LLL$ is eventually dendric if and only if $\LLL$ satisfies RBC.
\end{lemm}

\begin{proof}
	Let us first assume $\LLL$ is eventually dendric.
	By \cite[Lemma 1]{cDolcePerrin3} there exists $n_0$ such that
		$\mathscr{E}(w)$ for $w\in \LLL$ satisfying $|w| \geq n_0$ is a \emph{simple tree}, meaning
		that the diameter (maximum path distance between any two vertices) is at most $3$.
		
	Fix $w\in \LLL$ satisfying $|w|\geq n_0$.
	If $w$ is bispecial, then we will show that there exists a unique $a_0\in Ex^{\ell}(w)$
		and a unique $b_0\in Ex^r(w)$ such that $Ex^{\ell}(w) = Ex^{\ell}(wb_0)$ and $Ex^{r}(w) = Ex^{r}(a_0 w)$.
	Because $\mathscr{E}(w)$ is a tree,
	Therefore $w$ is regular bispecial and therefore $\LLL$ satisfies RBC.

	By the bispeciality of $w$, $Ex^{\ell}(w)$ and $Ex^{r}(w)$ each contain at least two elements.
	Let $a,a'\in Ex^{\ell}(w)$ be distinct elements.
	Because $\mathscr{E}(w)$ is connected with diameter at most $3$, there must exist a path of length $2$ connecting $a$ and $a'$
		(all paths between elements of $Ex^{\ell}(w)$ must be of even length).
	Let $b_0\in Ex^{r}(w)$ be such that $(a,b_0)$ and $(a',b_0)$ are in $\mathscr{E}(w)$.
	Because $\mathscr{E}(w)$ is a tree, if there exists $b\in Ex^r(w)$ such that $(a,b)$ and $(a',b)$ are in $\mathscr{E}(w)$,
		then $b = b_0$.
	Furthermore, $b_0$ is the unique vertex that is visited by any path of length two connecting distinct elements of $Ex^{\ell}(w)$.
	It follows that $Ex^{\ell}(w b_0) = Ex^{\ell}(w)$ and for any $b\neq b_0$ in $Ex^{r}(w)$ we have
		$|Ex^{\ell}(w b)| = 1$.
	We may show by a similar argument that there exists $a_0\in Ex^{\ell}(w)$ so that
		$Ex^{r}(a_0 w) = Ex^r(w)$ and for $a\neq a_0$ in $Ex^{\ell}(w)$ we have $|Ex^{r}(a w)| =1$.
		
	Now we assume that $\LLL$ satisfies RBC and let $n_3$ be such that all $n_2 \geq n_1 \geq n_3$
		satisfy all parts of Lemma \ref{LEM_RBC_UNIQUENESS}.
	We first observe that for any $w\in \LLL$ the graph $\mathscr{E}(w)$
		must be a (simple) tree if $w$ is not bispecial,
		so we now assume that $|w|\geq n_3$ and $w$ is (regular) bispecial.
	Let $a_0$ be the unique left extension such that $a_0 w$ is right special and
		$b_0$ be the unique right extension such that $w b_0$ is left special.
	The connected component of $\mathscr{E}(w)$ containing $a_0$ and $b_0$
		is a (simple) tree,
		and suppose by contradiction that $\mathscr{E}(w)$ is not connected.
	Then there must exist an edge $(a_1,b_1)$ in $\mathscr{E}(w)$
		where $a_1 \neq a_0$ is in $Ex^{\ell}(w)$ and $b_1\neq b_0$ is in $Ex^{r}(w)$.
	Note that $(a_1,b_0)$ and $(a_0,b_1)$ cannot be edges in $\mathscr{E}(w)$.
	However, $w b_0$ is left special, has $w$ as its prefix but $a_1 \not\in Ex^{\ell}(wb_0)$.
	This is a contradiction of the third part of Lemma \ref{LEM_RBC_UNIQUENESS}
		as $Ex^{\ell}(w b_0) \neq Ex^{\ell}(w)$.
		
	Therefore, $\mathscr{E}(w)$ is a tree for all large enough $n$ and this concludes the proof.
\end{proof}

\begin{coro}\label{COR_REC_IFF_UREC}
	If $\LLL$ satisfies RBC, then $\LLL$ is recurrent if and only if it is uniformly recurrent.
\end{coro}

In \cite[Theorem 2]{cDolcePerrin3}, it was shown that the set of subshifts $X$ with eventually
	dendric $\LLL$ is closed under topological conjugation.
We therefore also arrive at the same conclusion.

\begin{coro}\label{COR_RBC_CONJ}
	The subset $X$ of subshifts whose languages satisfy RBC is closed under conjugation.
\end{coro}

\subsection{Subshifts of a Finite Alphabet}\label{SSEC_INTRO_SUBSHIFTS}
Let
	$$
		\AAA^\NN = \{x = x_1x_2\dots: x_i \in \AAA \text{ for } i\in \NN\}
	$$
be the set of \emph{sequences} on alphabet $\AAA$ and for each $x\in \AAA^\NN$ and $1\leq i \leq j$ let
	$$
		x_{[i,j]} = x_i x_{i+1} \dots x_{j}
	$$
be the subword of $x$ of length $j-i +1$ starting at position $i$.
On the set $\AAA^\NN$ of infinite sequences on finite alphabet $\AAA$
	the \emph{cylinder sets}
		$$
			[w] = \{x\in \AAA^\NN: x_{[1,|w|]} = w\},\mbox{ for }w\in \AAA^*
		$$
form a basis for the natural topology $\TTT$.
Moreover, $(\AAA^\NN,\TTT)$ is compact metric space \cite{cFogg} given by metric
		$$
			d(x,y) = 2^{-\min\{i\in \NN: x_i \neq y_i\}}
		$$
for $x,y\in \AAA^\NN$, $x\neq y$.
The \emph{(left) shift} $S : \AAA^\NN \to \AAA^\NN$ is the continuous map defined by
	$S(x_1x_2x_3 \dots) = x_2 x_3 \dots$, or 
	$$
		(Sx)_i = x_{i+1} \mbox{ for all }i\in \NN, 
	$$
and the topological dynamical system $(\AAA^\NN,\TTT,S)$ is the \emph{shift} on $\AAA$.

A \emph{subshift} $X \subseteq \AAA^\ZZ$ is any non-empty closed set $X$ such that $SX = X$, considered as a subsystem
	$(X,S)$, where we omit mention of the induced topology of $\TTT$ on $X$.
$(X,S)$ is \emph{transitive} if for each non-empty open $U,V\subseteq X$ there
	exists $n\in \NN$ so that $S^n( U) \cap V \neq \emptyset$.
$(X,S)$ is \emph{minimal} if no proper closed subset $A\subseteq X$ is $S$-invariant.
	
The \emph{language} $\LLL(X)$ associated to $X$,
	$$
		\LLL(X) = \{w\in \AAA^*: x_{[1,|w|]} = w \mbox{ for some }x\in X\}
	$$
is the collection of words that occur as subwords of some sequence in $X$.
Likewise, given a language $\LLL$, we may define the associated subshift as
	$$
		X(\LLL) = \{x\in \AAA^\NN: x_{[i,j]} \in \LLL \mbox{ for all } i,j\in \NN, i\leq j\}.
	$$
	
	The following statements are well-known, and, given the definitions
		above, the proof is left as an exercise.
\begin{lemm}\label{LEM_UREC_IFF_MIN}
	Let $X \subseteq \AAA^\NN$ be a subshift and $\LLL \subseteq \AAA^*$ its associated language.
		\begin{enumerate}
			\item $(X,S)$ is transitive if and only if $\LLL$ is recurrent.
			\item $(X,S)$ is minimal if and only if $\LLL$ is uniformly recurrent.
		\end{enumerate}
\end{lemm}

We may consider $(X,S)$ as a measurable dynamical system by endowing it
	with the Borel $\sigma$-algebra $\Sigma$ generated by $\TTT$.
Let $\MMM(X,S)$ be the space of $S$-invariant probability measures on $X$,
	noting that $\MMM(X,S) \subseteq \MMM(\AAA^\NN,S)$
	by assigning to any measurable $A\subset \AAA^\NN$ the measure $\mu(A\cap X)$.
Because a subshift is a compact metric space, each $\mu\in \MMM(X,S)$ is regular.
In particular, by the Reisz Representation Theorem each $\mu$ is uniquely defined by
	$$
		\mu([w]) = \int_{X} \mathbbm{1}_{[w]} \,d\mu\mbox{ for all } w\in \AAA^*,
	$$
where $\mathbbm{1}_{[w]}$ is the indicator function on $[w]$,
	as $\{\mathbbm{1}_{[w]}: w\in \AAA^*\}$ forms a basis for the space of complex valued continuous functions on $\AAA^\NN$.
	
Let
	$
		\EEE(X,S)
	$
be the set of \emph{ergodic} probability measures of $(X,S)$, meaning
	all measures $\nu\in \MMM(X,S)$ such that for any measurable $A \subset X$,
	$\nu(S^{-1}(A) \Delta A) = 0$ implies that $\nu(A) \in \{0,1\}$.
By the pointwise ergodic theorem, for $\nu\in \EEE(X,S)$ and $\nu$-almost every $x\in X$,
	\begin{equation}
		\lim_{N\to \infty} \frac{1}{N} \sum_{i=0}^{N-1} \mathbbm{1}_{[w]}(S^i(x)) = \nu([w])\mbox{ for each } w\in \AAA^*.
	\end{equation}
Equivalently, for $u,w\in \AAA^*$, let
	$$
		|u|_w = \left|\{1 \leq i \leq |w| - |u|: w = u_{[i,i+|w|-1]} \}\right|.
	$$
Then for ergodic $\nu$ and $\nu$-almost every $x\in X$ we have
	\begin{equation}\label{EQ_ERGODIC_THM}
		\lim_{N\to \infty} \frac{\left|x_{[1,N]}\right|_{w}}{N} = \nu([w])\mbox{ for each } w\in \AAA^*.
	\end{equation}
	
Consider a sequence of words $\{u_n\}_{n\in \NN}$ such that $|u_n| \to \infty$ as $n\to\infty$.
For any $w\in \AAA^*$ we may refine to a subsequence $\{u_n\}_{n\in \WWW_w}$ for an infinite $\WWW_w\subset \NN$
	so that
		\begin{equation}\label{EQ_WORD_TO_FUNCTION}
			\lim_{\WWW_w \ni n \to \infty} \frac{|u_n|_w}{|u_n|} =: \phi(w)
		\end{equation}
	exists.
Because $\AAA^*$ is countable, we may diagonalize to find infinite $\WWW \subseteq \NN$ so that
	\eqref{EQ_WORD_TO_FUNCTION} exists for all $w\in \AAA^*$.
It follows that there is a unique $\mu\in \MMM(\AAA^\NN,S)$ defined by $\mu([w]) = \phi(w)$ for all $w\in \AAA^*$.
Furthermore, if $u_n\in \LLL(X)$ for all $n\in \WWW$ for some subshift $X$, then $\mu\in \MMM(X,S)$.
We denote this relationship by
	\begin{equation}\label{EQ_WORDS_TO_MEASURE}
		\{u_n\}_{n\in \WWW} \to \mu.
	\end{equation}
If we have multiple sequences $\{u^{(k)}_n\}_{n\in \NN}$, for finite list of $k$'s, we may find an infinite
	$\WWW\subset\NN$ so that $\{u_n^{(k)}\}_{n\in \WWW} \to \mu^{(k)}$ for each $k$.
Moreover, we may initially restrict this construction to any infinite $\VVV\subset \NN$ and find an infinite
	$\WWW\subset \VVV$.

\subsection{Natural Coding of a System}

Given a language $\LLL$ (or equivalently a subshift $X$ with language $\LLL = \LLL(X)$) and $n \in \mathbb{N}$,
	the Rauzy graph $\Gamma_n$
	is a directed graph with vertex set equal to $\LLL_n$, the words of length $n$ in $\LLL$,
	and edges defined by $\LLL_{n+1}$ as follows: there exists a directed edge from $u\in \LLL_n$ to $v\in \LLL_n$
	if and only if there exists $w\in \LLL_{n+1}$ such that $w_{[1,n]} = u$ and $w_{[2,n+1]} = v$.
	
A directed (multi)graph is \emph{strongly connected} if for each pair of vertices $u,v$ there exists
	a directed path from $u$ to $v$.
A directed (multi)graph is \emph{weakly connected} if its associated undirected graph is connected,
	or equivalently for each pair of distinct vertices $u,v$ there exist vertices
	$w_1,w_2,\dots,w_k$ in the graph so that $w_1 = u$, $w_k = v$ and
	for each $1\leq i < k$ either an edge from $w_i$ to $w_{i+1}$
		or an edge from $w_{i+1}$ to $w_i$ 
		in the graph.

The proof of the following is a direct consequence of the definition above.
\begin{lemm}
	If $(X,S)$ is a transitive subshift then each Rauzy graph $\Gamma_n$, $n\in \NN$, is strongly connected.
\end{lemm}

\begin{exam}
		The converse to this lemma is false.
		For example, the language $\LLL$ on $\AAA = \{0,1\}$
			given by $\LLL = \{w \in \AAA^*:~|w|_1 \leq 1\}$
			is not transitive, as $1w1\not\in \LLL$ for any $w\in \LLL$.
		However, each Rauzy graph $\Gamma_n$ is strongly connected.
\end{exam}

A left special word $w\in \LLL^\ell_n$ then corresponds to a vertex in $\Gamma_n$ with in-degree
	equal to $|Ex^\ell(w)| \geq 2$, the number of left extensions of $w$.
Likewise a right special word $w\in \LLL^r_n$ corresponds to a vertex with out-degree $|Ex^r(w)| \geq 2$.
Unless a word is right and/or left special, its vertex belongs to a directed path in $\Gamma_n$ containing no branching points.
In Section \ref{SEC_ERG_TO_GRAPH}, we will ignore such paths and focus on these branching points.
A path from $u\in \LLL_n$ to $v\in \LLL_n$ in $\Gamma_n$ corresponds to a word $w\in \LLL$
	such that $w_{[1,n]} = u$ and $w_{[|w|-n+1,|w|]} = v$. A path $w$ is \emph{branchless} in $\Gamma_n$ if $w_{[i,i+n-1]}$ is neither left nor right special
	for each $2\leq i \leq |w|-n$; in other words, no subword of length $n$
	is special except possibly the first or last.
The following was defined in a slightly different manner in \cite{cDamFick2016}.

\begin{defn}\label{DEF_SPEC_RG}
	For language $\LLL$, the \emph{special Rauzy graph} for $n\in \NN$, $\Gamma_n^{\mathrm{sp}}$, is the directed (multi)graph
		with vertex set $\LLL_n^\ell \cup \LLL_n^r$, treating a bispecial $w\in \LLL_n$ as two vertices, $w^\ell$ and $w^r$.
	There is a directed edge from $u$ to $v$
		for each branchless path in $\Gamma_n$ from $u$ to $v$ and for each pair $u = w^\ell$, $v = w^r$
		for bispecial $w\in \LLL_n$.
\end{defn}

\begin{lemm}\label{LEM_RECUR_IMPLIES_NO_CLOSED_LOOPS}
	Fix $\mathfrak{s}\in \{\ell,r\}$.
	If a language $\LLL\subseteq \AAA^*$ is recurrent and for some $n$ there exists a
		\emph{closed circuit} in the Rauzy graph $\Gamma_n$, meaning edges $e_i$ from $w_i$ to $w_{i+1}$
			for $1\leq i \leq k$ such that $w_1 = w_{k+1}$ and $w_i \neq w_j$ for $1\leq i < j \leq k$,
	such that no vertex $w_i$ is $\mathfrak{s}$-special, then $\LLL$ is periodic.
\end{lemm}

\begin{proof}
	Suppose $\mathfrak{s} = \ell$ as the $\mathfrak{s}=r$ case is analogous.
	Because $w_i$ is not $\ell$-special, there is a unique left extension of $w_i$
		and the prefix of the associated length $n+1$ word is $w_{i-1}$, where the index in understood
		$\mathrm{mod}\ k$ to belong to $\{1,\dots,k\}$.
	Fix any $v\in \LLL_n$; by recurrence, there exists $u\in \LLL$ so that
		$vuw_1 \in \LLL$.
	By examining the left extensions starting at $w_1$ in $vuw_1$, it must be that $v = w_i$
		for some $i$.
	Therefore $\LLL_n = \{w_1,\dots,w_k\}$.
	
	We will finish by showing that $\LLL$ is periodic with period $k$, noting that it suffices to show that
			$$
				V_{j} = V_{j+k} \mbox{ for all }N \geq n, V\in \LLL_N\mbox{ and } 1\leq j \leq N-k.
			$$
	Fix such a word $V$. By our above reasoning each subword of length $n$ must be an element of
			$\{w_1,\dots,w_k\}$ and starting with the suffix of length $n$ and proceeding to the left
			the $w_i$'s must occur in the cyclic order given by our original circuit.
	This provides the desired periodicity.
\end{proof}

\begin{coro}
	If $\LLL\subset \AAA^*$ is recurrent and aperiodic, then for all $n\in \NN$ each edge in $\Gamma_n^{\mathrm{sp}}$
		joins distinct vertices.
\end{coro}

\begin{proof}
		Assume for a contradiction that an edge begins and ends at the same vertex $w\in \Gamma_n^{\mathrm{sp}}$.
		This vertex $w$ must be right or left special since otherwise the component of
			$w$ in $\Gamma_n^{\mathrm{sp}}$ consists of one vertex and one edge,
			and this would violate recurrence or aperiodicity.
		Because a bispecial word $\LLL_n$ is divided into two distinct vertices in $\Gamma_n^{\mathrm{sp}}$,
			one that is right special and one that is left special,
			it must be that $w$ is either left special but not right special or
			right special but not left special.
		By Lemma \ref{LEM_RECUR_IMPLIES_NO_CLOSED_LOOPS}, $\LLL$ must be periodic, and this is a contradiction.
\end{proof}

\begin{rema}\label{REM_PIGEON-HOLING_GRAPHS}
If $\LLL$ has ECG with constant $K$ and $n$ is large enough, then the number of vertices in $\Gamma_n^{\mathrm{sp}}$ is
	bounded from above by $2K$.
Moreover, if $\LLL$ satisfies RBC then by Lemma \ref{LEM_RBC_UNIQUENESS} for all large $n$,
	each $\Gamma_n^{\mathrm{sp}}$ has the same number of vertices of each type and multiplicity.
For example if for some large $n$ there are four left-special vertices in $\Gamma_n^{\mathrm{sp}}$,
		three with in-degree $2$ and the other with in-degree $4$, then the same is true for $\Gamma_{n'}^{\mathrm{sp}}$ for all $n'\geq n$.
In particular, for all large $n$ there are only a finite number of possible special Rauzy graphs (up to naming of vertices and edges).
\end{rema}

We note that, by the constructions above,
	for any language $\LLL$ and $n\in \NN$, the Rauzy graph $\Gamma_n$ is strongly connected (resp. weakly connected) if
		and only if the 
			special Rauzy graph $\Gamma_n^{\mathrm{sp}}$ is strongly connected (resp. weakly connected).

\subsection{Colorings}\label{SEC_COLOR_DEF}

The coloring assignment defined here was previously discussed in \cite{cDamFick2016}.
However, we introduce the concept for convenience and to modify some of the arguments.
As mentioned in Remark \ref{REM_PIGEON-HOLING_GRAPHS}, if $X$ is transitive and $\LLL$ satisfies RBC with growth rate $K$,
	then there are only finitely many graphs that the special Rauzy graphs $\Gamma_n^{\mathrm{sp}}$ can equal
		(in other words, the set $\{\Gamma_n^{\mathrm{sp}} : n \geq 1\}$ is finite).
In particular, we may choose a directed graph $\Lambda$ such that, up to vertex/edge naming, $\Lambda = \Gamma_n^{\mathrm{sp}}$
	for each $n\in \WWW_0$ for infinite $\WWW_0 \subset \NN$.
For each vertex $v\in \Lambda$, there is a corresponding $v^{(n)}\in \Gamma_n^{\mathrm{sp}}$ for $n\in \WWW_0$
	and likewise for each edge in $\Lambda$.
As discussed in Section \ref{SSEC_INTRO_SUBSHIFTS}, we may find an infinite $\WWW\subset \WWW_0$ so that
	for each $v\in \Lambda$ we have $\{v^{(n)}\}_{n\in \WWW} \to \mu_v \in \MMM(X,S)$
	as introduced in \eqref{EQ_WORDS_TO_MEASURE}.
	
Assuming $\LLL$ has ECG of rate $K$, for any $x\in X$ and $w\in \LLL$ we define the following \emph{upper-density} function
	\begin{equation}
		\DDD(w,x) := \limsup_{N\to \infty} \frac{1}{N} \sum_{j=1}^N r(w,x,j),
	\end{equation}
where, if $n = |w|$,
	\begin{equation}\label{EQ_r_DEFINING_DDD}
		r(w,x,j) = \begin{cases}
					1, & x_{[k,k+n-1]} = w \\ & \mbox{ for some } (j-1)(K+1)n < k \leq j (K+1) n,\\
					0, & \mathrm{otherwise.}
				\end{cases}
	\end{equation}
	
The next result is similar to \cite[Lemma 4.1]{cBosh} except we consider the $\ell$-special case as well and
	assume transitivity rather than minimality.
	
\begin{lemm}\label{LEM_THERE_IS_A_SPECIAL}
	Let $X$ be a transitive subshift whose language $\LLL$ has ECG of rate $K$.
	Let $x \in X$.
	Either $\LLL$ is periodic or for all large $n$ and any $j\in \NN$ the word $x_{[j, j + (K+2)n-2]}$
		contains at least one $w\in\LLL_n^\mathfrak{s}$ (for both $\mathfrak{s} = \ell$ and $r$)
		as a subword.
\end{lemm}

\begin{proof}
	Let $n_0$ be given by the EGC growth condition, meaning
		$|\LLL_n| = Kn+C$ for all $n \geq n_0$,
	and fix $n\geq \max\{n_0,C\}$.
	Assuming there are no $\mathfrak{s}$-special subwords in
		$x_{[j,j+(K+2)n-2]}$ for some $j\in \NN$,
		we will show that $\LLL$ is periodic.	
	For each $i$ with $j \leq i \leq j+(K+1)n-1$, let
		$$
			w_i = x_{[i,i+n-1]}
		$$
	be the length $n$ subword of $x_{[j,j+(K+2)n-2]}$ beginning at position $i$ in $x$.
	There are $(K+1)n$ such words and because $|\LLL_n| = Kn+C$ it must be that
		$w_i = w_{i'}$ for some pair $i < i'$.
	Choose distinct $i,i'$ that minimize $i'-i$, noting in particular that $w_k\neq w_{k'}$
		for $i \leq k < k' <i'$.
	On the Rauzy graph $\Gamma_n$ there is a directed circuit 
		corresponding to the word $x_{[i,i'+n-1]}$
		whose vertices are not $\mathfrak{s}$-special.
	By Lemma \ref{LEM_RECUR_IMPLIES_NO_CLOSED_LOOPS}, $\LLL$ is periodic as claimed.
\end{proof}

\begin{coro}\label{COR_THE_SPECIAL}
	Let $X$ be transitive subshift whose aperiodic language $\LLL$
		has ECG of rate $K$.
	Then for any $x\in X$, large enough $n\in \NN$ and $\mathfrak{s}\in \{\ell,r\}$
		$$
			\DDD(w,x) \geq \frac{1}{K}
		$$		
	for some $w\in \LLL_n^\mathfrak{s}$.
\end{coro}
	
\begin{proof}
	Assume $n$ is large enough to satisfy Lemma \ref{LEM_THERE_IS_A_SPECIAL}.
	For a fixed $j \in \mathbb{N}$, consider the subword $x_{[(j-1)(K+1)n + 1, j(K+1)n+n-1]}$ of length
		$(K+2)n - 1$ that contains the words that start in positions $k$ in $x$ for
			$(j-1)(K+1)n < k \leq j(K+1)n$. 
	By Lemma \ref{LEM_THERE_IS_A_SPECIAL}, there exists a word $w^{(j)}\in \LLL_n^\mathfrak{s}$
		so that $w^{(j)} = x_{[k,k+n-1]}$  for such a $k$.
	In other words $r(w^{(j)},x,j) = 1$ as in \eqref{EQ_r_DEFINING_DDD}.
	So for each $N$,
		$$
			\sum_{j=1}^N\sum_{w\in \LLL_n^\mathfrak{s}} r(w,x,j) \geq N.
		$$
	Because $K = p(n+1) - p(n)$, we have $|\LLL_n^\mathfrak{s}| \leq K$
			by \eqref{EQ_GROWTH_AND_EXTENSIONS}.
	Therefore, there must exist $w\in \LLL_n^\mathfrak{s}$ so that for infinitely many $N$
		$$
			\frac{1}{N} \sum_{j=1}^N r(w,x,j) \geq \frac{1}{K},
		$$
	which implies that $\DDD(w,x) \geq \frac{1}{K}$ as desired.
\end{proof}

For each ergodic measure $\nu \in \EEE(X,S)$ we fix a generic $x^{(\nu)}\in X$, meaning \eqref{EQ_ERGODIC_THM} holds for each $w\in \AAA^*$,
	noting that $\nu([w]) = 0$ for all $w\notin \LLL$.
Recall our infinite $\WWW_0\subset \NN$ so that $\{v^{(n)}\} \to \mu_v \in \MMM(X,S)$ as defined
	before \eqref{EQ_WORDS_TO_MEASURE}.
We then define the following notation:
	\begin{equation}\label{EQ_D(nu)_DEF}
		\DDD(\mu_v,\nu) := \limsup_{\WWW_0 \ni n \to \infty} \DDD(v^{(n)},x^{(\nu)}).
	\end{equation}
The following result has related counterparts in \cite{cDamFick2016} as well as in \cite{cBosh}.

\begin{lemm}\label{LEM_ONE_COLOR_PER_MEASURE}
	Assume $X$ is a transitive subshift whose aperiodic language $\LLL$ has ECG of rate $K$.
	If for $v\in \Lambda$ and $\nu \in\EEE(X,S)$ we have
		$$
			\DDD(\mu_v, \nu) >0,
		$$
	then $\DDD(\mu_v, \nu') = 0$ for all $\nu'\in \EEE(X,S)$ such that $\nu'\neq \nu$.
\end{lemm}

\begin{proof}
	Let $\beta  = \DDD(\mu_v, \nu)$.
	We claim that  $\nu \geq \frac{\beta}{2(K+1)} \mu_v$, meaning
		$$\nu([u]) \geq \frac{\beta}{2(K+1)}\mu_v([u]) \mbox{ for all }u\in \LLL.$$
	This implies that $\nu = \frac{\beta}{2(K+1)} \mu_v + \left(1- \frac{\beta}{2(K+1)}\right) \mu^*$
		for some $\mu^*\in \MMM(X,S)$.
	By the extremality of ergodic measures, this implies that $\mu_v = \nu$, which completes the proof.
	
	We are left proving the claim.
	Fix $u\in \LLL$ and $\eps>0$.
	Choose a large $m \in \WWW_0$ so that
		$$
			\left|\beta - \DDD(v^{(m)}, x^{(\nu)}) \right| < \eps \mbox{ and } \left|\mu_v([u]) - \frac{|v^{(m)}|_u}{|v^{(m)}|}\right|<\eps,
		$$
	and let $n = |v^{(m)}|$.
	Now choose large $N > \frac{1}{\eps}$ so that
		$$
			\left|\DDD(v^{(m)},x^{(\nu)}) - \frac{1}{N} \sum_{j=1}^N r(v^{(m)},x^{(\nu)},j)\right| < \eps \mbox{ and } \left|\nu([u]) - \frac{|x^{(\nu)}_{[1,N(K+1)n]}|_u}{N(K+1)n}\right| < \eps.
		$$
	We see that
		$$
			|x^{(\nu)}_{[1,N(K+1)n]}|_u \geq \frac{|v^{(m)}|_u}{2} \sum_{j=1}^N r(v^{(m)},x^{(\nu)},j) - 1,
		$$
	as each possible beginning of $v^{(m)}$ contributing to the sum may overlap at most pairwise and
		otherwise contributes all its occurrences of $u$ to the left-hand quantity, except possibly the $v^{(m)}$ beginning in the last $(K+1)N$ block.
	Therefore,
		$$
			\begin{array}{rcl}
				\nu([u]) & > & \frac{|x^{(\nu)}_{1,N(K+1)n]}|_u}{N(K+1)n} - \eps\\
				~\\
						& \geq & \frac{1}{2(K+1)}\frac{|v^{(m)}|_u}{n}\frac{1}{N}\sum_{j=1}^N r(v^{(m)},x^{(\nu)},j) - \frac{1}{N(K+1)n} - \eps\\
				~\\
						& > & \frac{1}{2(K+1)}\left(\mu_v([u]) - \eps\right)(\beta - 2\eps)- 2\eps.
			\end{array}
		$$
	Because this holds for any $\eps>0$, we conclude that $\nu([u]) \geq \frac{\beta}{2(K+1)}\mu_v([u])$, as claimed.
\end{proof}

Given the previous lemma, the following ``coloring'' function is well defined.
In Section \ref{SEC_ERG_TO_GRAPH}, we will expand this definition to include edges of $\Lambda$.
The assumed properties of this coloring function will be given in Section \ref{SEC_GRAPHS}
	and justified in Section~\ref{SEC_ERG_TO_GRAPH} as well.

\begin{defn}\label{DEF_COLORING_1}
	Assume the notation in this section.
	For each vertex $v\in \Lambda$, we define $\CCC(v) = \nu\in \EEE(X,S)$ if and only if $\DDD(\mu_v,\nu) >0$
		and $\CCC(v) = \mathbf{0}$ if and only if $\DDD(\mu_v,\nu) = 0$ for all $\nu\in \EEE(X,S)$,
		where $\mathbf{0}$ is a fixed symbol not in $\EEE(X,S)$.
\end{defn}

\section{Exit Words}\label{SEC_EXIT_WORDS}

	In Section \ref{SEC_GRAPHS}, we will further discuss the coloring rule $\CCC$
		from Definition \ref{DEF_COLORING_1}.
	As will become apparent at that point,
		we will become concerned with the behavior of ``$N$-loops,''
		which are finite closed and simple paths (or circuits) within the
		(special) Rauzy graphs used to construct $\Lambda$.
		
	If we consider a vertex $w\in \LLL_n$ in $\Gamma_n$
		and follow a path that returns to $w$
		that visits any other vertex at most once,
	this path may be represented by a word $P\in \LLL$
	 such that
	 	$$P_{[1,n]} = P_{[|P|-n+1,|P|]} = w,$$
	 and this path in $\Gamma_n$ is of length $|P| - n + 1$.
	In addition to $P$, we will want to discuss
		the path $P'$ that arises by beginning at $w$ and following the path $P$
		$s\in\NN$ consecutive times. Then $P'$ must be of the form
		$$
			P' = \big(P_{[1,|P|-n]}\big)^{s} w = w \big(P_{[n+1,|P|]}\big)^{s}
		$$
	with $|P'| = s(|P|-n) + n$.
	
	There are many paths in $\Gamma_n$ (corresponding to words in $\mathcal{A}^*$) that we may construct
		that begin and end at $w$ (particularly when the path length is $\gg n$).
	We will be concerned with paths that are short relative
		to $n$.
	In such cases, the occurrences of $w$ as $P_{[1,n]}$ and
		$P_{[|P|-n+1,|P|]}$ will overlap
		and so special conditions must hold concerning the word $w$
		to allow such a path to exist.
	In the following definition, we consider $q = |P| - n$ to be the
		length between occurrences of $w$ and we are fixing
		the condition $q \leq n/2$ so that the arguments that follow hold. (This condition corresponds to the path being short relative to the length of the word $w$.)
	For now, we observe that under this restriction, there will certainly
		be overlap between occurrences of $w$.

\begin{defn}
	For $n \geq 2$ and $w\in \AAA^n$, suppose $q \in \mathbb{Z}$ satisfies $1 \leq q  \leq n/2$ and
		\begin{equation}\label{EqValid1}
			w_{[q+1,n]} = w_{[1,n-q]};
		\end{equation}
	that is, the length $(n-q)$ suffix of $w$ is also its prefix of the same length.
	For $r \in \mathbb{N}$, let $w^{q  \ast r}$ be the word of length $n+(r-1)q$ such that
			\[
			(w^{q \ast r})_{[q(i-1)+1, q(i-1) + n]}  = w
			\]
	for all $1\leq i \leq r$.
	Then $q$ is a \emph{valid step for $w$ in language $\LLL$}
		if \eqref{EqValid1} holds and $w^{q \ast 2}\in \LLL$.
\end{defn}

\begin{rema}\label{RemValidDefs}
	The rule \eqref{EqValid1} for $q$ and $w\in \AAA^n$ may be stated in various
		equivalent ways.
	For example, consider the length $q$ prefix $y = w_{[1,q]}$ of $w$ and the infinite
		sequence
			$
				w^{q\ast \infty} := y y y y y y \dots \in \AAA^\NN.
			$
	Then \eqref{EqValid1} holds if and only if 
	\[
	\left(w^{q\ast \infty}\right)_{[1,n]} = w.
	\]
	Alternately, \eqref{EqValid1} holds if and only if $w^{q\ast r}$ is defined for all $r\geq 2$.
\end{rema}

\begin{lemm}\label{lem: new_lem}
	Let $w \in \LLL_n$ and let $q',q$ be valid steps for $w$ in $\LLL$ such that $q' < q$. Then both of the following hold.
	\begin{enumerate}
	\item The infinite sequences $w^{q\ast \infty}$ and $w^{q'\ast \infty}$ are equal.
	\item The word $w^{q'\ast 2}$ is a prefix of the word $w^{q\ast 2}$.
	\end{enumerate}
\end{lemm}
\begin{proof}
Using condition \eqref{EqValid1} from the definition of valid, we have
\[
w = w_{[1,q]} w_{[1,n-q]} \quad \text{ and } \quad w = w_{[1,q']} w_{[1,n-q']}.
\]
Placing the second equation in the first, we use the fact that $q+q' \leq n$ to get
\[
w = w_{[1,q]} \left( w_{[1,q']}w_{[1,n-q']} \right)_{[1, n-q]} = w_{[1,q]} w_{[1,q']} w_{[1,n-q'-q]}.
\]
Similarly, $w=w_{[1,q']} w_{[1,q]} w_{[1,n-q'-q]}$. From these two, we obtain the commutativity relation
\[
w_{[1,q']}w_{[1,q]} = w_{[1,q]}w_{[1,q']}.
\]
A consequence of this relation is that
\[
w^{q\ast q'}w^{q'\ast q} = \left(w_{[1,q]}\right)^{q'} \left( w_{[1,q']}\right)^{q} = w^{q'\ast q}w^{q\ast q'},
\]
and so $w^{q\ast q'} = w^{q'\ast q}$. From this we conclude that $w^{q\ast \infty} = w^{q'\ast \infty}$; that is, item (1) holds.

To show item (2), note that the prefix of length $n + q'$ of $w^{q'\ast \infty}$ is $w^{q'\ast 2}$, so by item (1), the prefix of length $n+q'$ of $w^{q\ast \infty}$ is also $w^{q'\ast 2}$. But since the prefix of length $n+q$ (which is $> n+q'$) of $w^{q\ast \infty}$ is $w^{q\ast 2}$, we find that $w^{q'\ast 2}$ is a prefix of $w^{q\ast 2}$.
\end{proof}

\begin{lemm}\label{LemStepDiv}
	Let $w\in \LLL_n$. If $q'$ and $q$ are valid steps for $w$ in $\LLL$ with $q' < q$,
			then their greatest common divisor $q''= \mathrm{gcd}(q,q')$ is a valid step for $w$ in $\mathcal{L}$.
\end{lemm}

\begin{proof}
	Take valid steps $q,q'$ for $w$ in $\mathcal{L}$ with $q'<q$. We first note that $q-q'$ is valid for $w$ in $\mathcal{L}$. Indeed, by item (2) of Lemma~\ref{lem: new_lem}, $w^{q'\ast 2}$ is a prefix of $w^{q\ast 2}$, and so $w$ appears in positions $1$ and $1+q-q'$ of the word $\left(w^{q\ast 2}\right)_{[q'+1,n+q]}$. This word, which is in $\mathcal{L}$ because it is a subword of $w^{q\ast 2}$, must therefore be $w^{(q-q')\ast 2}$, and so $q-q'$ is valid.
	
	To prove the lemma, we will use induction, so for $m=1, \ldots, \lfloor n/2\rfloor-1$, let $S(m)$ be the statement: for $q=m$ and any $q'$ with $1 \leq q' < q$, if $q'$ and $q$ are valid steps for $w$ in $\mathcal{L}$, then so is $q'' = \text{gcd}(q,q')$. Note that $S(1)$ is vacuously true. Suppose that for some $m$ satisfying $1 \leq m \leq \lfloor n/2 \rfloor-1$, the statement $S(k)$ holds for all $k=1, \ldots, m-1$. Let $q=m$ and let $q'$ with $1 \leq q'<q$ be such that $q'$ and $q$ are valid steps for $w$ in $\mathcal{L}$. If $q'$ divides $q$, then $\text{gcd}(q,q') = q'$ and so the lemma is trivial, so suppose that $q'$ does not divide $q$, and so in particular $q'>1$. There are two cases.
	
	In the first case, $q-q'<q'$. Then apply the statement $S(q')$. Since $q'<q$ and $1 \leq q-q' < q'$, we find that $\text{gcd}(q',q-q')$ is a valid step. But $\text{gcd}(q',q-q') = \text{gcd}(q',q)$, so $\text{gcd}(q',q)$ is a valid step. In the second case, $q-q' \geq q'$. Since $q-q' \neq q'$ (as $q'$ does not divide $q$), we must have $q'< q-q'$. Then because $1\leq q-q' < q$, we can apply $S(q-q')$ to find that $\text{gcd}(q',q-q')$ is a valid step. But $\text{gcd}(q',q-q') = \text{gcd}(q',q)$, so $\text{gcd}(q',q)$ is a valid step. This completes the induction.
\end{proof}

\begin{defn}
	If for $w\in \AAA^n$ there exists a valid step $q'$ for $w$, we call
		$$
			\min\{q'':~q''\mbox{ is valid for } w\}
		$$
	the \emph{minimal step} for $w$. By Lemma~\ref{LemStepDiv}, the minimal step divides all valid steps.
\end{defn}

\begin{lemm}\label{LEM_NO_REPEAT}
	Consider $w\in \LLL_n$ with valid step $q$.
	If for some $i,j$ with $1 \leq i < j \leq q$, one has
		\begin{equation}\label{EQ_u_i_j}
			(w^{q\ast 2})_{[i, i+ n -1]} = (w^{q\ast 2})_{[j, j+ n -1]},
		\end{equation}
	then $j - i < q$ is a valid step for $w$.
\end{lemm}

\begin{proof}
	Let $u\in \LLL_n$ be the word in positions $i$ and $j$ given in \eqref{EQ_u_i_j}.
	If $i = 1$, then $u = w$ and the claim holds, so we now assume $i > 1$.
	Note that $j-i < q$ is a valid step for $u$, meaning
		\begin{equation}
			u_{[(j-i) + 1, n]} = u_{[1, n - (j-i)]}
		\end{equation}
	We want to show this relationship for $w$.
	
	We have
		\begin{equation}
			\begin{array}{rcl}
			 w_{[i,n - (j - i)]} & =  & (w^{q\ast 2})_{[i,n - (j-i)]}\\  & = & u_{[1, n + 1- j]}\\
			 	& = & (w^{q\ast 2})_{[j,n]}\\
			 	& = & w_{[j,n]},
			 \end{array}
		\end{equation}
	and
		\begin{equation}
			\begin{array}{rcl}
				w_{[1, n + i - q -1]} & = & (w^{q\ast 2})_{[q+1,n+i - 1]}\\
						& = & u_{[q+2-i, n]}\\
						& = & (w^{q\ast 2})_{[q + 1 + (j-i), n + i - 1 + (j-i)]}\\
						& = & w_{[1 + (j-i), (n + i - q - 1) + (j-i) ]}.
			\end{array}
		\end{equation}
	
	Because $n + i - q - 1 \geq i$, it follows that
		\begin{equation}
			w_{[1,n-(j-i)]} = w_{[1,i-1]}w_{[i,n-(j-i)]} = w_{[1+ (j-i), j-1]}w_{[j,n]} = w_{[1+(j-i),n]}
		\end{equation}
	or $j-i$ is a valid step for $w$.
\end{proof}

\begin{coro}\label{COR_NO_REPEAT}
	If $q$ is the minimal step for $w \in \LLL_n$, then
		the words
			\begin{equation}
				(w^{q\ast 2})_{[i,i+n-1]} \mbox{ for } 1\leq i \leq q,
			\end{equation}
	are all distinct.
\end{coro}

\begin{proof}
	Suppose for some $i,j$ with $1\leq i < j \leq q$ we have
		\begin{equation}
			(w^{q\ast 2})_{[i, i+ n -1]} = (w^{q\ast 2})_{[j, j+ n -1]}.
		\end{equation}
	By Lemma \ref{LEM_NO_REPEAT}, $q' = j-i< q$ is a valid step for $w$,
		contradicting the minimality of $q$.
\end{proof}

	The word $w^{q\ast r}$ represents the path obtained by traversing a loop (circuit) corresponding to
		$w^{q\ast 2}$ exactly $r-1$ times.
	If our language $\LLL$ is transitive and aperiodic,
		then for $w\in \LLL$
		we cannot realize all words of the type $w^{q\ast r}$, $r\in \NN$, within the language.
	As we will argue precisely in Lemma \ref{LemWordInExit},
		if a long enough path contains $w$ (and $w$ is not at the beginning of the path)
	then the path must contain $w^{q\ast r}$ for some $r\in\NN$ (the case $r=1$ is possible)
	entering the loop at some point, and then leaving.
	The next definition
		characterizes the paths of minimum length
		that begins and ends outside of the loop $w^{q\ast 2}$
		but also visits $w$.

\begin{defn}
	Let $q$ be a valid step for $w$ in $\LLL$. The word $z$ is called an \emph{exit word} for $w$ with step $q$ if it has a representation $z=pw^{q\ast r}s$ with $|p|,|s|\leq q$ such that all of the following hold: 
		\begin{enumerate}
			\item $pw^{q \ast r}  s\in \LLL$,
			\item $p w^{q \ast r}$ is not a suffix of $w^{q \ast (r+1)}$ but $p_{[2,|p|]}w^{q\ast r}$ is, and
			\item $w^{q \ast r}s$ is not a prefix of $w^{q \ast (r+1)}$ but $w^{q\ast r}s_{[1,|s|-1]}$ is.
		\end{enumerate}				
\end{defn}

\begin{rema}
	If $w^{q\ast r} = w^{q'\ast r'}$ for some $q'<q$ and $r'>r$, it is possible that there are different choices of pairs $(p,s)$ and $(p',s')$
		so that
			$$ z = p w^{q \ast r}s \mbox{ and } z= p' w^{q' \ast r'} s'$$
		form different representations of the same exit word for $w$.
	For example, 
		if we take $w=1^4=1111$ (so $n=4$) then the exit word $z=01^{15}0$
		may be represented by the following choices of parameters:
		$p=0$, $s=110$, $q=3$, $r=4$, and $p'=01$, $s'=0$, $q'=2$, $r'=6$. 
\end{rema}

\begin{lemm}\label{LemMinRep}
	Let $w\in \LLL_n$ and let $q$ be a minimal step for $w$ in $\LLL$. If $z$ is an exit word for $w$ with step $q$, then
		\begin{enumerate}
			\item The representation of $z$ as $p w^{q\ast r}s$ is unique, and
			\item $|z|_w = r$; that is,
				$$
					\{1\leq i \leq |z|-n: z_{[i,i+n-1]} = w\}
				$$
				contains exactly $r$ values.
		\end{enumerate}
\end{lemm}

\begin{proof}
	We begin with part 2. First, let $j = |p|+1$ be the starting position of $w^{q\ast r}$ in $z$.
	Then
	\begin{equation}\label{EqLemMin}
		z_{[j + q(i-1),j + qi -1]} = w
	\end{equation}
	for $1 \leq i \leq r$, so $|z|_w \geq r$. Assume that there is some other position $k$ with $1 \leq k \leq |p|+|s|+(r-1)q+1$ that is not of the form $|p|+1 + mq$ for some $m$ in $0, \ldots, r-1$ such that $z_{[k,k+n-1]} = w$. We will obtain a contradiction.
	
	We may assume that
		there exists $i$ in $1, \ldots, r$ such that $0<|j + q(i-1) - k| =: q' <q$, because otherwise either $k=1$ and $|p|=q$, or $k=|p|+|s|+(r-1)q +1$ and $|s|=q$, and both of these can be seen to contradict the definition of an exit word.
	The assumed inequalities imply that $q'<q$ is a valid step for $w$, and this contradicts the minimality of $q$. 
	
	Now part 1 follows immediately, as the positions of $w$ in $z$ are completely determined by 
part 2 and so is the position of $w^{q\ast r}$.
\end{proof}

Recall that a sequence $x\in \AAA^\NN$ is \emph{eventually periodic}
	if for some $(n',p)\in \NN\times \NN$ we have $x_{j+p} = x_j$ for all $j\geq n'$.
	If a subshift $X$ is minimal and its language $\LLL$ is aperiodic, then no
		$x\in X$ is eventually periodic.

\begin{lemm}\label{LemWordInExit}
	Let $X$ be a subshift.
	If $x\in X$ is not eventually periodic and $q$ is a minimal valid step for $w\in \LLL_n$, then
		for each beginning position $j\in \NN$ of $w$ in $x$ (that is, $x_{[j,j+n-1]} = w$) exactly one of the following must hold:
			\begin{enumerate}
				\item $x_{[1,j+n-1]}$ is a suffix of $w^{q\ast r}$ where $r = \lceil \frac{j-1}{q}\rceil + 1$, or
				\item there exists a unique exit word $z$ of $w$ with step $q$ with an occurrence in $x$ that contains $w$ at position $j$; that is,
						there exists $k<j$ so that $z = x_{[k,k+|z|-1]}$ and $k+|z|-1 > j + n -1$.
			\end{enumerate}
\end{lemm}

\begin{proof}
Define
\[
j_1 = \min\{k \in [1,j] : x_{[k,j+n-1]} \text{ is a suffix of } w^{q\ast r}\},
\]
and note that the set in question contains $j$, so it is nonempty. The condition $j_1 = 1$, is equivalent to (1). In this case, if $x_{[j,j+n-1]}$ were part of an exit word $pw^{q\ast r_0}s$ in $x$, then by Lemma~\ref{LemMinRep}, $x_{[j,j+n-1]}$ would have to be one of the $r_0$ many occurrences of $w$ in $pw^{q\ast r_0}s$. In other words, the exit word would have to begin at position $j-mq-|p| \geq 1$ for some $m \geq 0$. By the definition of exit word, $pw^{q\ast r_0}$ is not a suffix of $w^{q\ast (r_0+1)}$ and so $x_{[j-mq-|p|,j+n-1]}$ is also not a suffix of $w^{q\ast (r_0+1)}$, which in turn implies it is not a suffix of $w^{q\ast r}$. This contradicts the definition of $j_1$. In other words, items (1) and (2) cannot hold simultaneously. If (1) fails, then $j_1>1$ and we further define
\[
j_2 = \max\{k \geq j : x_{[j,k+n-1]} \text{ is a prefix of } w^{q \ast \infty}\}.
\]
Note that the set in question contains $j$, so it is nonempty.
Furthermore,
	because $x$ is not eventually periodic,
	$j_2< \infty$.
Then $x_{[j_1-1,j_2+n]}$ is an exit word of $w$ with step $q$, with representation $p'w^{q\ast r'}s'$, where
\[
p' = x_{[j_1-1,j-\lfloor \frac{j-j_1}{q} \rfloor q-1]} \quad \text{and} \quad s'= x_{[j+\lfloor \frac{j-j_1}{q}\rfloor q +n,j_2+n]},
\]
so that $r' = \lfloor \frac{j_2-j}{q} \rfloor + \lfloor \frac{j-j_1}{q} \rfloor +1$.
Indeed, because $w$ only appears in positions $mq+1$ for $m \geq 0$ in the sequence $w^{q\ast \infty}$,
	the word $p_{[2,|p|]}w^{q\ast r'} q_{[1,|q|-1]}=x_{[j_1,j_2+n-1]}$ is a subword of $w^{q\ast \infty}$,
	and therefore $p_{[2,|p|]}w^{q\ast r'}$ is a suffix of $w^{q\ast(r'+1)}$ and $w^{q\ast r'}s_{[1,|s|-2]}$
	is a prefix of $w^{q\ast(r'+1)}$.
Furthermore by definition of $j_1$ and $j_2$,
	$pw^{q\ast r'}$ is not a suffix of $w^{q\ast(r'+1)}$ and $w^{q\ast r'}s$
	is not a prefix of $w^{q\ast(r'+1)}$.

For similar reasons, the exit word $z$ must be unique. If there were two, say $pw^{q\ast r}s$ and $p'w^{q\ast r'}s'$, then $x_{[j,j+n-1]}$ would have to appear in one of the $r$ or $r'$ many positions allotted by Lemma~\ref{LemMinRep}. This, along with the suffix and prefix properties of exit words, forces $r=r'$ and furthermore that the beginning positions of $w^{q\ast r}$ and $w^{q\ast r'}$ containing $x_{[j,j+n-1]}$ are equal. Likewise, by these same properties, the beginning and ending positions of $pw^{q\ast r}s$ and $p'w^{q\ast r'}s'$ must be equal, giving $p=p'$ and $s=s'$.
\end{proof}

\begin{coro}\label{CorExitOverlap}
	Let $x\in \AAA^\NN$ with language $\LLL = \LLL_x$ and
	  $w\in \LLL_n$ with minimal valid $q$ and $z = p w^{q\ast r}s,z' = p' w^{q\ast r'}s'$ be exit words of $w$ with step $q$.
	If $z$ begins at position $i$ of $x$ and $z'$ begins at position $i'>i$ of $x$, then
		$i'\geq i + |z| - n$.
	Moreover, even if these occurrences overlap; that is, $i' < i + |z|$,
		then $x$ still contains at least $r+r'$
		occurrences of $w$ in this common interval.
	In other words,
			$
				\left|x_{[i,i'+|z'|-1]}\right|_w \geq  r+ r'
			$.
\end{coro}

\begin{proof}
	Consider the Rauzy graph $\Gamma^*_n$ of size $n$ for the full language $\AAA^*$.
	Because $q$ is valid, $w^{q \ast 2}\in \AAA^*$ is defined.
	Moreover, $w^{q\ast 2}$ (and $w^{q\ast r}$ for $r>2$) is realized as a directed circuit $L$ of
		edges in $\Gamma^*_n$.
	Specifically, $L$ is the circuit with vertices (in order) $v_1, \ldots, v_{q+1}$,
		where $v_1$ and $v_{q+1}$ correspond to the word $w$,
		and for $i=2, \ldots, q$, $v_i$ corresponds to the word $\left(w^{q\ast 2}\right)_{[i,i+n-1]}$.
	Because of Corollary~\ref{COR_NO_REPEAT}, this circuit is vertex self-avoiding (except its initial and final points).
	In other words, if $v_i=v_j$, then either $i=j$ or 
		$i,j\in \{1,q+1\}$.
	The edges of $L$ are (in order) $e_1, \ldots, e_q$, where 
		$e_i$ begins at $v_i$ and ends at $v_{i+1}$,
		and $e_i$ corresponds to the word $\left( w^{q\ast 2}\right)_{[i,i+n]}$.
	Then the word $w^{q\ast r}$ corresponds to fully traversing this circuit $r-1$ times,
		starting at $v_1$ and ending back at $v_1 = v_{q+1}$.
	(When $r=1$, we start at $v_1$ and do not cross any edges.) 
		
	Let
		$$
			e_i = z_{[i,i+n]},~ 1 \leq i \leq |z|-n
		$$
	be the consecutive edges of the path represented by $z$ in $\Gamma^*_n$.
	We claim now that $e_1 \notin L$.
	To prove this, note that by definition of an exit word,
		the word $z_{[2,n+1]}$ must correspond to a vertex $u$ in $L$,
		whereas the vertex $u'$ corresponding to $z_{[1,n]}$ cannot equal $u''$,
		the vertex in $L$ that directly precedes $u$.
	If $u' \notin L$, then $e_1 \notin L$.
	Otherwise, if $u' \in L$ but $e_1 \in L$,
		then the 
		edge from $u'$ to $u$ and the edge from
			$u''$ to $u$ are both in $L$.
	They are distinct because $u'\neq u''$,
		and they are nontrivial (not self-loops) because if either of $u'$ or $u''$ were equal to $u$,
		then $L$ would contain only one vertex, and this would contradict $u'\neq u''$.
	However there is only one edge in $L$ which ends at $u$,
		since $L$ is a vertex self-avoiding path. This shows that $e_1 \notin L$.
	
	A argument similar to the above shows that $e_{|z|-n} \notin L$.
	Furthermore, the definition of an exit word mandates that for all $i=2, \ldots, |z|-n-1$, $e_i \in L$.
	Likewise, the edges $e'_i$ for $z'$ (defined analogously) in $\Gamma^*_n$ also satisfy
		$e'_1,e'_{|z'|-n}\not\in L$ while $e'_i\in L$ otherwise.
	Therefore the position $i'$ of $z'$ must be at least at the position of $e_{|z|-n}$ in $z$ at $i$, or
		$$
			i' \geq i + |z| - n.
		$$
	
	If $z$ and $z'$ do not overlap ($i' \geq i+|z|$),
		then the inequality $|x_{[i,i'+|z'|-1]}|_w \geq r+r'$ is immediate.
	If the words overlap, we can conclude the same lower bound using the inequality $i' \geq i+|z|-n$.
	Indeed, there are $r$ occurrences of $w$ in the word $z$ at locations
		$|p|+1, |p|+q+1, \ldots, |p|+(r-1)q+1$ (corresponding to $w^{q\ast r}$),
		and this last occurrence at position $|p|+(r-1)q+1$ satisfies
	\[
	|p|+(r-1)q+1 < |z|-n,
	\]
	This means that there are $r$ occurrences of $w$ at indices  less than $i+|z|-n$ in $x$.
	Similarly, there are $r$ occurrences of $w$ at indices greater than $i'$ in $x$ (corresponding to $w^{q\ast r'}$).
	This gives a total of at least $r+r'$ occurrences of $w$ in the common interval.
\end{proof}

\begin{rema}
	Concerning the last statement of Corollary \ref{CorExitOverlap},
		in the common overlapping interval of $z$ and $z'$ a $w$ may begin in $z$ and end in $z'$ (but not
		occur as subword of either).
	More precisely, $w$ may begin at position $j$ for $i + |z| - n < j < i'$.
	Although not necessary for this work, one can show the following:
		if $q$ is the minimal valid step for $w\in \LLL_n$ and there exists $q' \leq n - q$ so that
			\eqref{EqValid1} holds for $q'$, then $q'$ is a multiple of $q$ and for any $r\geq 1$ we have
			$w^{q'\ast r} = w^{q\ast(kr - k + 1)}$, where $q' = kq$.
	Using this fact, it may be shown that $j$ must belong to an interval of length at most $4q - n - 4$.
	We would then have $|x_{[i,i'+|z'| -1]}|_w \leq r + r' + 2$ as $q \leq n/2$ and
		in fact $|x_{[i,i'+|z'| -1]}|_w = r + r'$ if, for instance, $q \leq n/4$.
\end{rema}

\begin{rema}\label{REM_EXIT_WORDS_SPECIALS}
	Given $w$ in $\LLL$ of length $n$ with minimal step $q$,
		let $z = p w^{q\ast r} s$ be an exit word,
		noting that $|z| = n + q(r-1) + |s| + |p| \geq n + 2$.
	Let $u = z_{[2,n+1]}$ and $v = z_{[|z|-n,|z|-1]}$.
	By definition, $u$ and $v$ are subwords of $w^{q\ast 2}\in \LLL$,
		so let their respective positions in $w^{q\ast 2}$ be
		$i>1$ and $j\leq q$, 
		noting that the word at position $1$ is $w$ which is also in position $q+1$.
	Let $a$ be the letter such that the left extension $au$ satisfies $au = w^{q\ast 2}_{[i-1,i+n-1]}$.
	By definition of an exit word,
		if $a'$ is the letter satisfying $a'u = z_{[1,n+1]}$,
		then $a'\neq a$.
	Likewise if $b$ is the letter satisfying $vb = w^{q\ast 2}_{[j,j+n]}$,
		then $vb' = z_{[|z|-n,|z|]}$ where $b'\neq b$.
	Furthermore,
		the words $a'u$, $vb'$ and the parameter $r$ uniquely define $z$.
	We can see this by using the notation of the previous proof of Corollary~\ref{CorExitOverlap}.
	Write $L$ for the self-avoiding circuit in the Rauzy graph $\Gamma_n^*$ with the property that $w^{q\ast 2}$ (with $q$ minimal)
		corresponds to beginning at $w$,
		traversing this circuit once, and ending back at $w$.
	Then we have seen that an exit word $pw^{q\ast r}s$ corresponds to starting at a vertex,
		taking an edge not in $L$ but whose other endpoint is in $L$, traversing $L$ until we reach $w$,
		following the circuit $L$ $r-1$ times, then moving to another vertex of $L$, and then taking one edge not in $L$.
		(The last portion of the path after the $r-1$ iterations of $L$ follows $L$ only partway and does not touch $w$,
		since Lemma~\ref{LemMinRep} mandates that the traversal only touch $w$ a total of $r$ times.)
	The words $a'u$ and $vb'$ above indicate the entrance and exit edges to and from $L$,
		and the parameter $r$ indicates the number of traversals of $L$.
\end{rema}
	
\begin{lemm}\label{LemExitCount}
	Under RBC with related constant $K$ and large enough word $w$ in the language,
		the set of exit words $\XXX$ of minimal valid step $q$ for $w$ 
		contains at most $2K^2$ elements.
\end{lemm}

\begin{proof}
	We begin by proving that if $z = pw^{q\ast r}s \in \XXX$ and $z'= pw^{q \ast r'} s \in \XXX$,
		$r \leq r'$,
		have the same prefix and suffix, but possibly different middle repetition,
		then either $r' = r$ or $r' = r+1$.
	Suppose by contradiction that $r' \geq r +2$ and consider the word
		$y = z_{[2,|z|-1]}$ obtained by removing the first and last letters of $z$.
	Note that
		$$
			y = z'_{[2,1+|y|]} = z'_{[2+q,1 + q + |y|]} = \dots = z'_{[2+(r'-r)q,1 + (r'-r)q + |y|]}
		$$
		
	We will argue that $y$ is bispecial but not regular bispecial, violating RBC.
	First $y$ is left special, as the occurrences of $y$ at positions $2$ and $2+q$ in $z'$ have different left extensions.
	Likewise, $y$ is right special, as the occurrences of $y$ at positions $2 + (r'-r-1)q$ and $2+ (r'-r)q$ have different right extensions.
	Therefore $y$ is bispecial.
	Let $b$ be defined by $yb = z_{[2,|z|]}$, noting that $yb = z'_{[2+(r'-r)q,2 + (r'-r)q + |y|]}$ as well.
	Because $z$ is an exit word, the left extension $z$ of $yb$ must differ from the left extension $z'_{[1+(r'-r)q,2 + (r'-r)q + |y|]}$
		and so $yb$ is left special.
	However, if $b'$ is defined by $yb' = z'_{[2,2+|y|]}$ then $yb' = z'_{[2+q,2 + q + |y|]}$ as well and $b'\neq b$.
	Similarly $z'_{[1,2+|y|]}$ and $z'_{[1+q,2 + q + |y|]}$ are distinct left extensions of $yb'$.
	Because $y$ admits two left special right extensions, it is not regular bispecial as claimed.
	
		Recall from Remark~\ref{REM_EXIT_WORDS_SPECIALS} that an exit word $z$ for $w$ is uniquely determined
			by the words $z_{[1,n+1]} = a'u$, $z_{[|z|-n,|z|]} = vb'$, and the parameter $r$.
		Note that any such $u$ is left-special because, in addition to the left extension $a'u$,
			there is also another left extension $au$ which is a subword of $w^{q\ast 2}$
				(it corresponds to an edge which follows the circuit to which $w^{q\ast 2}$ corresponds).
		So by \eqref{EQ_GROWTH_AND_EXTENSIONS}, there are at most $K$
			choices for $a'u$.
			Likewise there are at most $K$ choices for $vb'$ and therefore at most $K^2$ choices for pairs $(a'u,vb')$.
			Each such pair deremines
			 $p$ and $s$ in any representation $z=pw^{q\ast r}s$ by the
			method described in Remark~\ref{REM_EXIT_WORDS_SPECIALS}.
			(One enters the circuit $L$ using the edge described by $a'u$, moves to $w$,
				makes $r-1$ traversals of $L$,
				and leaves through an edge described by $vb'$.
				Then $p$ is the prefix of any such $z$ up until the first location of $w$.)
		Once $p$ and $s$ are determined, there are at most two choices for $r$, and 
			therefore there are at most $2K^2$ exit words as desired.
\end{proof}

\section{Just graphs and colors}\label{SEC_GRAPHS}

	Instead of remembering where $\Lambda$'s and the corresponding $\CCC$ rules
		come from (the underlying shift and ergodic measures respectively), in this section we will
		define abstract graphs and coloring functions satisfying a set of properties.
		We will then show
		(in  Sections \ref{SEC_NLOOPS_MOVES} -- \ref{SEC_NLOOPS_NEWG} and proven in the appendix) that such graphs cannot have more than
			$\frac{K+1}{2}$ $N$-loops of different colors (see Corollary~\ref{cor: main_corollary}).
	In Section \ref{SEC_2LOOPS}, we will first show this conclusion for $2$-loops.
	The $2$-loop case is presented first (although it falls under the general case)
		because it is easier to prove and conveys
		some of the principal arguments used in the general case.
	 
	 These results will be used in Section \ref{SEC_ERG_TO_GRAPH},
	 	where we will construct special Rauzy graphs and coloring functions
	 		(from subshifts with that obey RBC)
	 		satisfying the properties outlined in this section.
	 	We will then conclude that such graphs cannot support more than $\frac{K+1}{2}$
	 		$N$-loops of different colors, and deduce our main theorem, Theorem~\ref{THM_MAIN}.
			
	\begin{nota}\label{NOTA1}
		With or without decoration (primes, tildes, etc.), $\Lambda$ will be a directed (multi)graph such that, for given integers $K>2$ and $1\leq K_\ell,K_r \leq K$,
				\begin{enumerate}
					\item $\Lambda$ has $K_\ell + K_r$ vertices,
					\item $K_\ell$ of the vertices are \emph{left special},
						meaning they each have at least two incoming edges and exactly one outgoing edge,
					\item $K_r$ of the vertices are \emph{right special}, meaning they each have exactly one incoming edge and at least two outgoing edges,
					\item $\Lambda$ has $K+K_\ell+K_r$ edges.
							It is possible that distinct edges $e$ and $e'$ share the same initial and terminal vertices (because this
								is a multigraph).
					\item $\Lambda$ is \emph{strongly connected}, meaning for any two vertices $u$ and $v$ in $\Lambda$, there exists a directed path in $\Lambda$
							from $u$ to $v$.
				\end{enumerate}
				Items (1)-(5) above imply in particular that $\Lambda$ cannot have any ``self-loops,''
					meaning edges that begin and end at the same vertex.
				Indeed, if $u$ had a self-loop and were, for example, left special,
					then there would not be a path from $u$ to any other distinct vertex $v$ in $\Lambda$.
				
		With each graph, there will be an associated coloring rule $\CCC$ on the vertices and edges of $\Lambda$ that
			takes values in $\NN\cup\{0\}$.
		By this we mean:
				\begin{enumerate}
					\setcounter{enumi}{5}
					\item $\CCC(v)\in \NN\cup\{0\}$ is defined for each vertex $v$ in $\Lambda$,
					\item 
						$\CCC(e) \in \NN\cup\{0\}$ is defined for each edge $e$ in $\Lambda$,
					\item 
						if $\CCC(e) \neq 0$ for an edge $e$ from $u$ to $v$ in $\Lambda$, then $\CCC(u) = \CCC(v) = \CCC(e)$.
				\end{enumerate}
	\end{nota}
	
		\begin{rema}
			In this section, we are treating $\CCC$ as a function whose image is a subset of $\NN\cup\{0\}$.
			However, when we justify these rules and definitions for graphs derived from an appropriate subshift $X$ in Section \ref{SEC_ERG_TO_GRAPH},
				$\CCC$ will have image $\EEE(X)\cup \{\textbf{0}\}$.
			Upon first reading here, it may help to associate the number $0$ to the symbol $\textbf{0}$ (``not colored'')
				and associate any positive integer to a measure in $\EEE(X)$.
		\end{rema}

	We impose the following rules on $\CCC$.
			
	\begin{rules}[Rules concerning one graph]\label{RULE1}
		For $\Lambda$, $\CCC$, and some integer $E>0$, we require that:
			\begin{enumerate}
				\item $\CCC$ only assigns values in $\{0,1,\dots,E\}$.
				\item For each $\nu \in \{1,\dots,E\}$, there exists at least one left special $u$ and at least one right special $v$
						such that $\CCC(u) = \CCC(v) = \nu$.
				\item  If for a vertex $v$ we have $\CCC(v) \neq 0$, 
						then there exist edges $e,e'$ such that $e$ ends at $v$, $e'$ begins at $v$ and $\CCC(e)=\CCC(e') = \CCC(v)$.
						As a consequence, $\CCC(u) = \CCC(w) = \CCC(v)$ as well,
							where $e$ begins at $u$ and $e'$ ends at $w$.
						Note that by our assumptions in Notation \ref{NOTA1}, $v \notin \{u,w\}$, although it is possible that $u=w$.
				\item If $\CCC(e) \neq 0$ for an edge $e$ in $\Lambda$,
					then there exists a closed directed path (that is, a directed circuit) in $\Lambda$ containing the edge $e$
							such that each edge $e'$ in the path satisfies $\CCC(e') = \CCC(e)$.
			\end{enumerate}
	\end{rules}
	
	\subsection{Regular Bispecial Moves}\label{SSEC_RBS_DEF}

	We now consider \emph{regular bispecial (RBS)} moves to transition from one graph $\Lambda$ to a new graph $\Lambda'$.
	These moves give us dynamics on the sets of graphs and coloring functions, and only RBS moves are allowed to transition between graphs.
	(These moves have counterparts on special Rauzy graphs constructed from subshifts, as will be shown in Section \ref{SEC_ERG_TO_GRAPH}).
	An edge $e_0$ from $u$ to $v$ is said to be \emph{bispecial} if $u$ is left special and $v$ is right special.
	Note that if $e_0$ is bispecial, then it is the unique edge from $u$ to $v$.
		Let $e_1,\dots e_L$ be the distinct edges that end at $u$
			and $\tilde{e}_1,\dots,\tilde{e}_R$ be the distinct edges beginning at $v$.
		It may be the case that $e_i = \tilde{e}_j$ for some $i$'s and $j$'s.
		For each $i=1, \dots, L$, we define $y_i$ to be the vertex at the beginning of $e_i$
			and for each $j$ we let $z_j$ be the ending vertex of $\tilde{e}_j$.
			
	We will now define an \emph{RBS move} from $\Lambda$ to a graph $\Lambda'$.
	The rigorous definition will follow this paragraph, but informally we will be
		removing the edge $e_0$ from the graph as well as the vertices $u$ and $v$,
		reattaching $v$ as the ending vertex of $e_{i_0}$ for a single $i_0$ with $1\leq i_0 \leq L$,
		reattaching $u$ as the beginning vertex of $\tilde{e}_{j_0}$ for a single $j_0$ with $1\leq j_0 \leq R$,
		reattaching the $e_i$'s, $i\neq i_0$, to end at $u$, reattaching the $\tilde{e}_j$'s, $j\neq j_0$,
			to begin at $v$, and then replacing $e_0$ with an edge from $v$ to $u$.
	See Figure \ref{FIG_RBS} for an example with $L=R = 3$ and $i_0 = j_0 = 1$.
		
	For the full definition, we first define an auxiliary graph $\tilde{\Lambda}$.
		This graph $\tilde{\Lambda}$ has the vertices from $\Lambda$ except $u$ and $v$,
			and we add new vertices $x_1,\dots, x_L,\tilde{x}_1,\dots,\tilde{x}_R$.
		For each edge $e\not\in\{e_0,e_1,\dots,e_L,\tilde{e}_1,\dots,\tilde{e}_R\}$ from $\Lambda$
			we have an edge in $\tilde{\Lambda}$ with the same beginning and ending vertices.
		For each $e_i$, $1\leq i \leq L$, we include in $\tilde{\Lambda}$ an edge that begins at $y_i$ but ends at $x_i$
			if $e_i \not\in \{\tilde{e}_1,\dots,\tilde{e}_R\}$.
		If $e_i = \tilde{e}_j$ for some $1\leq j \leq R$, then we include in $\tilde{\Lambda}$ an edge
			beginning at $\tilde{x}_j$ and ending at $x_i$.
		For each $\tilde{e}_j$, $1\leq j \leq R$, that is not an element of $\{e_1,\dots,e_L\}$
			we include an edge in $\tilde{\Lambda}$ from $\tilde{x}_j$ to $z_j$.
		Then for the single $1\leq i_0 \leq L$ and $1 \leq j_0 \leq R$, we replace $e_0$ in $\Lambda$
			with an edge from $x_{i_0}$ to $\tilde{x}_{j_0}$ in $\tilde{\Lambda}$.
		Note that this graph is not strongly connected as no edge begins at $x_i \neq x_{i_0}$ and
			no edge ends at $\tilde{x}_j\neq \tilde{x}_{j_0}$.
		
	To create $\Lambda'$, we have the same vertices as $\Lambda$.
	For each edge $e$ in $\tilde{\Lambda}$ let $q$ and $q'$
		be its beginning vertex and ending vertex respectively.
	We then include in $\Lambda'$ an edge from $\Psi(q)$ to $\Psi(q')$
		where
		$\Psi$ is a function from the vertex set of $\tilde{\Lambda}$ to the vertex
		set of $\Lambda'$ given by:
			$$
				\Psi(p) = \begin{cases}
					u,& p = \tilde{x}_{j_0}\mbox{ or } p = x_i \mbox{ for } i\neq i_0,\\
					v,& p = x_{i_0}\mbox{ or } p = \tilde{x}_j \mbox{ for } j\neq j_0,\\
					p, & \mbox{othwerwise.}
				\end{cases}
			$$
	In particular, if an edge $e$ was copied from $\Lambda$ to $\tilde{\Lambda}$ without
		modification then that edge will remain unaltered in $\Lambda'$.
	Moreover, each edge $e$ in $\Lambda$ may be naturally associated to an edge in $\Lambda'$
		and we refer to this edge also as $e$ without ambiguity.
		
	An RBS move is a transformation as described above such that the resulting graph $\Lambda'$
		remains strongly connected.
	For example, if $y_2 = y_3 = z_1$ and $y_1 = z_2 = z_3$ in Figure \ref{FIG_RBS} then
		the choice $i_0 = j_0 = 1$ shown would not define an RBS move.
	However, the choice $i_0 = 2,j_0 = 1$ would define an allowed RBS move.

	\begin{figure}[t]
		{\centering
			\begin{tikzpicture}[vertstyle/.style={draw=black, top color = black!5, bottom color = black!35},
			edgestyle/.style={->,thick,black}, dotedgestyle/.style={thick,dotted,black}]
	\def\asp{0.13}

	\draw [thick,black!20,fill=black!10] (-6*\asp, 9*\asp) rectangle (36*\asp,-9*\asp);
	\draw [thick,black!20,fill=black!10] (44*\asp, 9*\asp) rectangle (86*\asp,-9*\asp);
	
	\draw [dotedgestyle] (0,6*\asp) -- +(-5*\asp, 2*\asp);
	\draw [dotedgestyle] (0,6*\asp) -- +(-5*\asp, -2*\asp);
	\draw [dotedgestyle] (0,-6*\asp) -- +(-5*\asp, 2*\asp);
	\draw [dotedgestyle] (0,-6*\asp) -- +(-5*\asp, -2*\asp);
	\draw [dotedgestyle] (0,0*\asp) -- +(-5*\asp, 2*\asp);
	\draw [dotedgestyle] (0,0*\asp) -- +(-5*\asp, -2*\asp);

	\draw [edgestyle] (1.4*\asp,4.6*\asp) -- (8.6*\asp,1.4*\asp);
	\draw [edgestyle] (2*\asp,0*\asp) -- (8*\asp,0*\asp);
	\draw [edgestyle] (1.4*\asp,-4.6*\asp) -- (8.6*\asp,-1.4*\asp);

	\draw [edgestyle] (12*\asp,0) -- +(6*\asp,0);
	
	\draw [edgestyle] (21.4*\asp,1.4*\asp) -- (28.6*\asp,4.6*\asp);
	\draw [edgestyle] (22*\asp,0) -- (28*\asp,0);
	\draw [edgestyle] (21.4*\asp,-1.4*\asp) -- (28.6*\asp,-4.6*\asp);
	
	\draw [dotedgestyle] (30*\asp,6*\asp) -- +(5*\asp, 2*\asp);
	\draw [dotedgestyle] (30*\asp,6*\asp) -- +(5*\asp, -2*\asp);
	\draw [dotedgestyle] (30*\asp,-6*\asp) -- +(5*\asp, 2*\asp);
	\draw [dotedgestyle] (30*\asp,-6*\asp) -- +(5*\asp, -2*\asp);
	\draw [dotedgestyle] (30*\asp,0*\asp) -- +(5*\asp, 2*\asp);
	\draw [dotedgestyle] (30*\asp,0*\asp) -- +(5*\asp, -2*\asp);

	\shade [vertstyle] (0,6*\asp) circle (2*\asp) node {$y_1$};
	\shade [vertstyle] (0,0) circle (2*\asp) node {$y_2$};
	\shade [vertstyle] (0,-6*\asp) circle (2*\asp) node {$y_3$};
	
	\shade [vertstyle] (10*\asp,0*\asp) circle (2*\asp) node {$u$};
	\shade [vertstyle] (20*\asp,0*\asp) circle (2*\asp) node {$v$};
	
	\shade [vertstyle] (30*\asp,6*\asp) circle (2*\asp) node {$z_3$};
	\shade [vertstyle] (30*\asp,0) circle (2*\asp) node {$z_2$};
	\shade [vertstyle] (30*\asp,-6*\asp) circle (2*\asp) node {$z_1$};
	
	\draw (15*\asp, - 11*\asp) node {\scalebox{1.5}{$\Lambda$}};

	\draw [->,very thick, black!20] (17*\asp, - 11*\asp) -- (63*\asp, - 11*\asp);
	\draw (40*\asp,-13*\asp) node {RBS};

	\draw [dotedgestyle] (50*\asp,6*\asp) -- +(-5*\asp, 2*\asp);
	\draw [dotedgestyle] (50*\asp,6*\asp) -- +(-5*\asp, -2*\asp);
	\draw [dotedgestyle] (50*\asp,-6*\asp) -- +(-5*\asp, 2*\asp);
	\draw [dotedgestyle] (50*\asp,-6*\asp) -- +(-5*\asp, -2*\asp);
	\draw [dotedgestyle] (50*\asp,0*\asp) -- +(-5*\asp, 2*\asp);
	\draw [dotedgestyle] (50*\asp,0*\asp) -- +(-5*\asp, -2*\asp);

	\draw [edgestyle] (52*\asp,6*\asp) -- (58*\asp,6*\asp);
	\draw [edgestyle] (52*\asp,0*\asp) -- (68.3*\asp,-5*\asp);
	\draw [edgestyle] (52*\asp,-6*\asp) -- (68*\asp,-6*\asp);

	\draw [edgestyle] (61*\asp,4.3*\asp) -- (69*\asp,-4.3*\asp);
	
	\draw [edgestyle] (62*\asp,6*\asp) -- (80*\asp,6*\asp);
	\draw [edgestyle] (61.7*\asp,5*\asp) -- (78*\asp,0);
	\draw [edgestyle] (72*\asp,-6*\asp) -- (78*\asp,-6*\asp);
	
	\draw [dotedgestyle] (80*\asp,6*\asp) -- +(5*\asp, 2*\asp);
	\draw [dotedgestyle] (80*\asp,6*\asp) -- +(5*\asp, -2*\asp);
	\draw [dotedgestyle] (80*\asp,-6*\asp) -- +(5*\asp, 2*\asp);
	\draw [dotedgestyle] (80*\asp,-6*\asp) -- +(5*\asp, -2*\asp);
	\draw [dotedgestyle] (80*\asp,0*\asp) -- +(5*\asp, 2*\asp);
	\draw [dotedgestyle] (80*\asp,0*\asp) -- +(5*\asp, -2*\asp);

	\shade [vertstyle] (50*\asp,6*\asp) circle (2*\asp) node {$y_1$};
	\shade [vertstyle] (50*\asp,0) circle (2*\asp) node {$y_2$};
	\shade [vertstyle] (50*\asp,-6*\asp) circle (2*\asp) node {$y_3$};
	
	\shade [vertstyle] (70*\asp,-6*\asp) circle (2*\asp) node {$u$};
	\shade [vertstyle] (60*\asp,6*\asp) circle (2*\asp) node {$v$};
	
	\shade [vertstyle] (80*\asp,6*\asp) circle (2*\asp) node {$z_3$};
	\shade [vertstyle] (80*\asp,0) circle (2*\asp) node {$z_2$};
	\shade [vertstyle] (80*\asp,-6*\asp) circle (2*\asp) node {$z_1$};

	\draw (65*\asp, - 11*\asp) node {\scalebox{1.5}{$\Lambda'$}};

\end{tikzpicture}
			
		}
		\caption{A regular bispecial move.}\label{FIG_RBS}
	\end{figure}
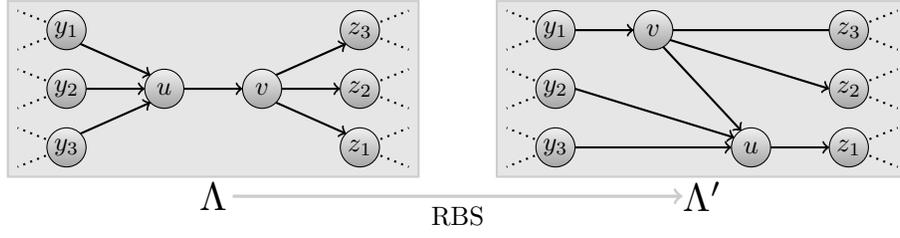
	
	The coloring rule $\CCC'$ for $\Lambda'$ will be related to the rule $\CCC$ for $\Lambda$.
	\begin{rules}[RBS moves and colors]\label{RULE2}
		Let $\Lambda$ have coloring $\CCC$ and let $\Lambda'$ be the result of the RBS move
			with 
			the vertex and edge names defined in this section.
		Then coloring rule $\CCC'$ for $\Lambda'$ must satisfy the following.
		\begin{enumerate}
			\item If $q$ is a vertex in $\Lambda'$ not equal to $u$ or $v$,
				then
					$\CCC'(q) = \CCC(q).$
			\item 
					If $e \neq e_0$ is an edge in $\Lambda$,
						then $\CCC'(e) = \CCC(e)$.
		\end{enumerate}
		Furthermore, both $\CCC$ and $\CCC'$ must satisfy Rules List \ref{RULE1}.
	\end{rules}
	
	\subsection{Just 2-loops}\label{SEC_2LOOPS}
	
	In this section, we give a special case of our main result on abstract graphs:
		under the rules listed in the last section,
		$\Lambda$ can have distinct colors on at most $\frac{K+1}{2}$ many $2$-loops. (See Corollary~\ref{COR2LOOPS}.)
	The argument we employ will be referred to as a {\it separation argument}.
	The basic strategy is to show that if $\mathcal{C}$ assigns distinct colors to a collection of $2$-loops,
		then the graph produced from $\Lambda$ by removing these loops must be connected
		(that is, the loops cannot separate the graph).
	A counting argument then shows that the number of such loops is at most $\frac{K+1}{2}$.
	
	Given Rules List \ref{RULE1}, the minimal preimage of $\nu\in \{1,\dots,E\}$ comes from
		a \emph{2-loop}, meaning
		a pair of vertices $u$ and $v$, with $u$ left special and $v$ right special,
			an edge $e_0$ from $u$ to $v$ and an edge $e_1$ from $v$ to $u$ such that
			$$ \CCC^{-1}(\nu) = \{u ,v, e_0, e_1\}.$$
	We recall the edges $e_1,\dots e_L,\tilde{e}_1,\dots, \tilde{e}_R$
		and vertices $y_1,\dots,y_L$ and $z_1,\dots,z_R$ from Section \ref{SSEC_RBS_DEF}
		and assume in this case that $e_1 = \tilde{e}_1$.
	We also note that $y_1 = v$ and $z_1 = u$. 
	As before, an RBS move is defined by the choices $1\leq i_0\leq L$ and $1\leq j_0 \leq R$.
	
	There are exactly two types of allowed RBS moves on the edge $e_0$
		:
	\begin{enumerate}
		\item 
			If $i_0 = j_0=1$, the RBS move actually preserves the $2$-loop from $\Lambda$ to $\Lambda'$.
				Here the roles of $e_0$ and $e_1$ switch as $e_0$ is from $v$ to $u$ in $\Lambda'$ and
					$e_1$ is from $u$ to $v$.
			
		\item 
			If $i_0 \neq 1$ and $j_0 \neq 1$, the $2$-loop becomes ``undone'' meaning both edges
					$e_0$ and $e_1$ are from $v$ to $u$ in $\Lambda'$.
	\end{enumerate}
	If only one of $i_0$ or $j_0$ is equal to $1$, then one can check that the resulting graph $\Lambda'$
		is not strongly connected.
	
	See Figure \ref{FIG_RBS_2LOOP} for an example of each type of move with $L= R = 3$.

	For graphs constructed from subshifts in Section \ref{SEC_ERG_TO_GRAPH},
		a 2-loop with nonzero color is preserved (and keeps its color) for a number of RBS moves of the first type above until
		the RBS move of the second type occurs.
	In the resulting graph, the color must pass along at least one incoming and at least one outgoing edge 
		into the rest of the graph.
	As Rules List \ref{RULE2} shows, the color must have been passed along those analogous edges in $\Lambda$ as well.
		We will explain the subtleties that complicate this idea in Section \ref{SEC_ERG_TO_GRAPH}.
		However for now we will
		state what we need in terms of colorings and graphs.

	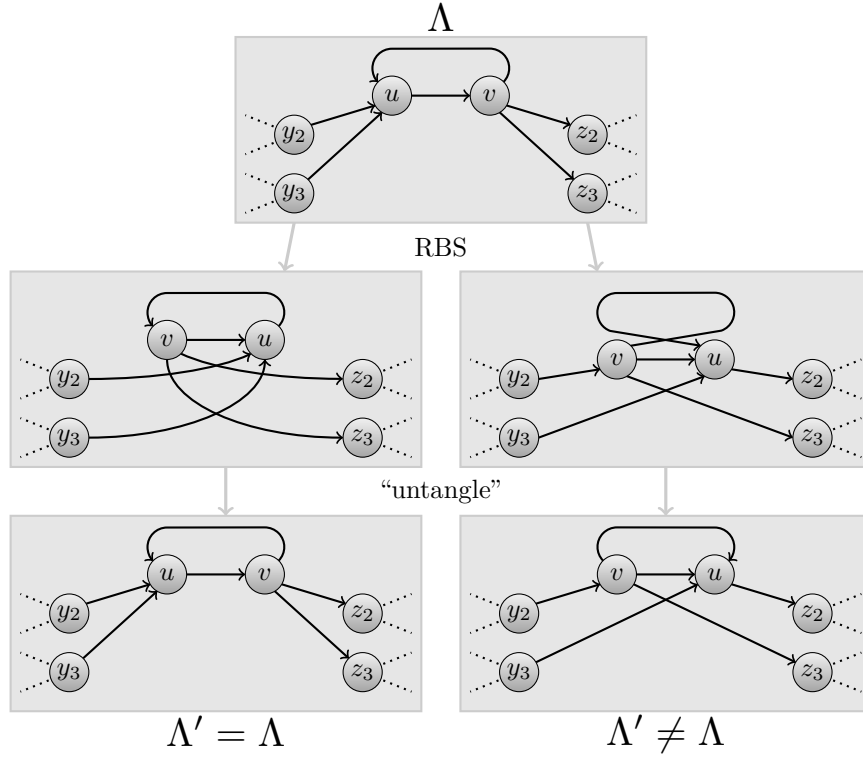
\begin{figure}[t]
		\begin{center}
			\begin{tikzpicture}[vertstyle/.style={draw=black, top color = black!5, bottom color = black!35},
			edgestyle/.style={->,thick,black}, dotedgestyle/.style={thick,dotted,black},
			gbox/.style={thick,black!20,fill=black!10}]
	\def\asp{0.13}

	\draw [gbox] (-6*\asp, 6*\asp) rectangle (36*\asp,-13*\asp);
	\draw [gbox] (17*\asp, -18*\asp) rectangle +(42*\asp,-20*\asp);
	\draw [gbox] (-29*\asp, -18*\asp) rectangle +(42*\asp,-20*\asp);
	\draw [gbox] (17*\asp, -43*\asp) rectangle +(42*\asp,-20*\asp);
	\draw [gbox] (-29*\asp, -43*\asp) rectangle +(42*\asp,-20*\asp);
	
	\draw [->,very thick, black!20] (0*\asp, -13*\asp) -- (-1*\asp, -18*\asp);
	\draw [->,very thick, black!20] (30*\asp, -13*\asp) -- (31*\asp, -18*\asp);
	\draw (15*\asp,-15.5*\asp) node {RBS};

	\draw [->,very thick, black!20] (38*\asp, -38*\asp) -- +(0, -5*\asp);
	\draw [->,very thick, black!20] (-7*\asp,  -38*\asp) -- +(0, -5*\asp);
	\draw (15*\asp,-40.5*\asp) node {``untangle''};
	
	
	\draw [dotedgestyle] (0,-4*\asp) -- +(-5*\asp, 2*\asp);
	\draw [dotedgestyle] (0,-4*\asp) -- +(-5*\asp, -2*\asp);
	\draw [dotedgestyle] (0,-10*\asp) -- +(-5*\asp, 2*\asp);
	\draw [dotedgestyle] (0,-10*\asp) -- +(-5*\asp, -2*\asp);

	\draw [edgestyle] (1.7*\asp,-3*\asp) -- (8.3*\asp,-\asp);
	\draw [edgestyle] (1.4*\asp,-8.6*\asp) -- (9*\asp,-1.7*\asp);

	\draw [edgestyle] (12*\asp,0) -- +(6*\asp,0);
	
	\draw [edgestyle] (21.7*\asp,-\asp) -- (28.3*\asp,-3*\asp);
	\draw [edgestyle] (21*\asp,-1.7*\asp) -- (28.6*\asp,-8.6*\asp);
	
	\draw [dotedgestyle] (30*\asp,-4*\asp) -- +(5*\asp, 2*\asp);
	\draw [dotedgestyle] (30*\asp,-4*\asp) -- +(5*\asp, -2*\asp);
	\draw [dotedgestyle] (30*\asp,-10*\asp) -- +(5*\asp, 2*\asp);
	\draw [dotedgestyle] (30*\asp,-10*\asp) -- +(5*\asp, -2*\asp);
	
	\draw [edgestyle] (21.4*\asp,1.4*\asp) arc (-45:90:2*\asp) -- (10*\asp,4.8*\asp) arc (90:225:2*\asp);
	
	\shade [vertstyle] (0,-4*\asp) circle (2*\asp) node {$y_2$};
	\shade [vertstyle] (0,-10*\asp) circle (2*\asp) node {$y_3$};
	
	\shade [vertstyle] (10*\asp,0*\asp) circle (2*\asp) node {$u$};
	\shade [vertstyle] (20*\asp,0*\asp) circle (2*\asp) node {$v$};
	
	\shade [vertstyle] (30*\asp,-4*\asp) circle (2*\asp) node {$z_2$};
	\shade [vertstyle] (30*\asp,-10*\asp) circle (2*\asp) node {$z_3$};

	
	\draw [dotedgestyle] (-23*\asp,-29*\asp) -- +(-5*\asp, 2*\asp);
	\draw [dotedgestyle] (-23*\asp,-29*\asp) -- +(-5*\asp, -2*\asp);
	\draw [dotedgestyle] (-23*\asp,-35*\asp) -- +(-5*\asp, 2*\asp);
	\draw [dotedgestyle] (-23*\asp,-35*\asp) -- +(-5*\asp, -2*\asp);

	\draw [edgestyle] (-21*\asp,-29*\asp) arc (-90: -30: \asp*19.15 and 5.17*\asp);
	\draw [edgestyle] (-21*\asp,-35*\asp) arc (-90:0: \asp*18 and 8*\asp); 

	\draw [edgestyle] (-11*\asp,-25*\asp) -- +(6*\asp,0);
	
	\draw [edgestyle] (-11.6*\asp, -26.4*\asp) arc (210: 270: \asp*19.15 and 5.17*\asp);
	\draw [edgestyle] (-13*\asp, -27*\asp) arc (180:270: \asp*18 and 8*\asp);
	
	\draw [dotedgestyle] (7*\asp,-29*\asp) -- +(5*\asp, 2*\asp);
	\draw [dotedgestyle] (7*\asp,-29*\asp) -- +(5*\asp, -2*\asp);
	\draw [dotedgestyle] (7*\asp,-35*\asp) -- +(5*\asp, 2*\asp);
	\draw [dotedgestyle] (7*\asp,-35*\asp) -- +(5*\asp, -2*\asp);
	
	\draw [edgestyle] (-1.6*\asp,-23.6*\asp) arc (-45:90:2*\asp) -- (-13*\asp,-20.2*\asp) arc (90:225:2*\asp);
	
	\shade [vertstyle] (-23*\asp,-29*\asp) circle (2*\asp) node {$y_2$};
	\shade [vertstyle] (-23*\asp,-35*\asp) circle (2*\asp) node {$y_3$};
	
	\shade [vertstyle] (-13*\asp,-25*\asp) circle (2*\asp) node {$v$};
	\shade [vertstyle] (-3*\asp,-25*\asp) circle (2*\asp) node {$u$};
	
	\shade [vertstyle] (7*\asp,-29*\asp) circle (2*\asp) node {$z_2$};
	\shade [vertstyle] (7*\asp,-35*\asp) circle (2*\asp) node {$z_3$};

	
	\draw [dotedgestyle] (23*\asp,-29*\asp) -- +(-5*\asp, 2*\asp);
	\draw [dotedgestyle] (23*\asp,-29*\asp) -- +(-5*\asp, -2*\asp);
	\draw [dotedgestyle] (23*\asp,-35*\asp) -- +(-5*\asp, 2*\asp);
	\draw [dotedgestyle] (23*\asp,-35*\asp) -- +(-5*\asp, -2*\asp);

	\draw [edgestyle] (25*\asp,-29*\asp) -- (31.3*\asp,-28*\asp);
	\draw [edgestyle] (25*\asp,-35*\asp) -- (42*\asp,-28.7*\asp);

	\draw [edgestyle] (35*\asp,-27*\asp) -- +(6*\asp,0);
	
	\draw [edgestyle] (44.7*\asp,-28* \asp) -- (51*\asp,-29*\asp);
	\draw [edgestyle] (34*\asp,-28.7*\asp) -- (51*\asp,-35*\asp);
	
	\draw [dotedgestyle] (53*\asp,-29*\asp) -- +(5*\asp, 2*\asp);
	\draw [dotedgestyle] (53*\asp,-29*\asp) -- +(5*\asp, -2*\asp);
	\draw [dotedgestyle] (53*\asp,-35*\asp) -- +(5*\asp, 2*\asp);
	\draw [dotedgestyle] (53*\asp,-35*\asp) -- +(5*\asp, -2*\asp);
	
	\draw [edgestyle] (34.4*\asp, -25.6*\asp) -- (43.68*\asp, -24.05*\asp) 
			arc (-70:90:2*\asp)
		-- (33*\asp,-20.2*\asp) arc (90:250:2*\asp) -- (41.6*\asp,-25.6*\asp);
	
	\shade [vertstyle] (23*\asp,-29*\asp) circle (2*\asp) node {$y_2$};
	\shade [vertstyle] (23*\asp,-35*\asp) circle (2*\asp) node {$y_3$};
	
	\shade [vertstyle] (33*\asp,-27*\asp) circle (2*\asp) node {$v$};
	\shade [vertstyle] (43*\asp,-27*\asp) circle (2*\asp) node {$u$};
	
	\shade [vertstyle] (53*\asp,-29*\asp) circle (2*\asp) node {$z_2$};
	\shade [vertstyle] (53*\asp,-35*\asp) circle (2*\asp) node {$z_3$};


	\draw [dotedgestyle] (-23*\asp,-53*\asp) -- +(-5*\asp, 2*\asp);
	\draw [dotedgestyle] (-23*\asp,-53*\asp) -- +(-5*\asp, -2*\asp);
	\draw [dotedgestyle] (-23*\asp,-59*\asp) -- +(-5*\asp, 2*\asp);
	\draw [dotedgestyle] (-23*\asp,-59*\asp) -- +(-5*\asp, -2*\asp);

	\draw [edgestyle] (-21.3*\asp,-52*\asp) -- (-14.7*\asp,-50*\asp);
	\draw [edgestyle] (-21.6*\asp,-57.6*\asp) -- (-14*\asp,-50.7*\asp);

	\draw [edgestyle] (-11*\asp,-49*\asp) -- +(6*\asp,0);
	
	\draw [edgestyle] (-1.3*\asp,-50*\asp) -- (5.3*\asp,-52*\asp);
	\draw [edgestyle] (-2*\asp,-50.7*\asp) -- (5.6*\asp,-57.6*\asp);
	
	\draw [dotedgestyle] (7*\asp,-53*\asp) -- +(5*\asp, 2*\asp);
	\draw [dotedgestyle] (7*\asp,-53*\asp) -- +(5*\asp, -2*\asp);
	\draw [dotedgestyle] (7*\asp,-59*\asp) -- +(5*\asp, 2*\asp);
	\draw [dotedgestyle] (7*\asp,-59*\asp) -- +(5*\asp, -2*\asp);
	
	\draw [edgestyle] (-1.6*\asp,-47.6*\asp) arc (-45:90:2*\asp) -- (-13*\asp,-44.2*\asp) arc (90:225:2*\asp);
	
	\shade [vertstyle] (-23*\asp,-53*\asp) circle (2*\asp) node {$y_2$};
	\shade [vertstyle] (-23*\asp,-59*\asp) circle (2*\asp) node {$y_3$};

	\shade [vertstyle] (-13*\asp,-49*\asp) circle (2*\asp) node {$u$};
	\shade [vertstyle] (-3*\asp,-49*\asp) circle (2*\asp) node {$v$};
	
	\shade [vertstyle] (7*\asp,-53*\asp) circle (2*\asp) node {$z_2$};
	\shade [vertstyle] (7*\asp,-59*\asp) circle (2*\asp) node {$z_3$};



	\draw [dotedgestyle] (23*\asp,-53*\asp) -- +(-5*\asp, 2*\asp);
	\draw [dotedgestyle] (23*\asp,-53*\asp) -- +(-5*\asp, -2*\asp);
	\draw [dotedgestyle] (23*\asp,-59*\asp) -- +(-5*\asp, 2*\asp);
	\draw [dotedgestyle] (23*\asp,-59*\asp) -- +(-5*\asp, -2*\asp);

	\draw [edgestyle] (24.7*\asp,-52*\asp) -- (31.3*\asp,-50*\asp);
	\draw [edgestyle] (24.7*\asp,-58*\asp) -- (41.3*\asp,-50*\asp);

	\draw [edgestyle] (35*\asp,-49*\asp) -- +(6*\asp,0);
	
	\draw [edgestyle] (44.7*\asp,-50*\asp) -- (51.3*\asp,-52*\asp);
	\draw [edgestyle] (34.7*\asp,-50*\asp) -- (51.3*\asp,-58*\asp);
	
	\draw [dotedgestyle] (53*\asp,-53*\asp) -- +(5*\asp, 2*\asp);
	\draw [dotedgestyle] (53*\asp,-53*\asp) -- +(5*\asp, -2*\asp);
	\draw [dotedgestyle] (53*\asp,-59*\asp) -- +(5*\asp, 2*\asp);
	\draw [dotedgestyle] (53*\asp,-59*\asp) -- +(5*\asp, -2*\asp);
	
	\draw [edgestyle] (31.6*\asp,-47.6*\asp) arc (225:90:2*\asp) -- (43*\asp,-44.2*\asp) arc (90:-45:2*\asp);
	
	\shade [vertstyle] (23*\asp,-53*\asp) circle (2*\asp) node {$y_2$};
	\shade [vertstyle] (23*\asp,-59*\asp) circle (2*\asp) node {$y_3$};
	
	\shade [vertstyle] (33*\asp,-49*\asp) circle (2*\asp) node {$v$};
	\shade [vertstyle] (43*\asp,-49*\asp) circle (2*\asp) node {$u$};
	
	\shade [vertstyle] (53*\asp,-53*\asp) circle (2*\asp) node {$z_2$};
	\shade [vertstyle] (53*\asp,-59*\asp) circle (2*\asp) node {$z_3$};


	\draw (15*\asp, 8*\asp) node {\scalebox{1.5}{$\Lambda$}};
	
	\draw (-7*\asp, -65*\asp) node {\scalebox{1.5}{$\Lambda' = \Lambda$}};
	\draw (38*\asp, -65.3*\asp) node {\scalebox{1.5}{$\Lambda' \neq \Lambda$}};

\end{tikzpicture}
		\end{center}
		\caption{The RBS moves possible on a 2-loop. The left move preserves the 2-loop while the right move elimates it.
				}\label{FIG_RBS_2LOOP}
	\end{figure}

	\begin{rules}[Acceleration for 2-loops]\label{RULE3}
		Given a fixed choice of 2-loops in $\Lambda$ with coloring $\CCC$ satisfying
			(a) Rules List \ref{RULE1} and
			(b) all the fixed 2-loops have distinct nonzero colors,
			there exists $\Lambda'$ with coloring $\CCC'$ (also satisfying Rules List \ref{RULE1}) such that:
			\begin{enumerate}
				\item $\Lambda'$ is the result of a finite sequence of RBS moves acting on $\Lambda$.
				\item The fixed 2-loops remain in $\Lambda'$ from $\Lambda$, meaning no RBS moves other than type 1 above from $\Lambda$ to
						$\Lambda'$ have acted
						on these 2-loops.
						Furthermore, $\CCC'$ agrees with $\CCC$ on the fixed
					2-loops that are colored by $\CCC$.
				\item There exists at least one fixed 2-loop in $\Lambda'$ colored by $\nu\neq 0$ such that $\CCC'$ colors at least one edge
							leaving the 2-loop and at least one edge entering the loop by $\nu$.
					In other words, the color $\nu$ ``spreads'' outside the 2-loop in $\Lambda'$.
			\end{enumerate}
	\end{rules}
	
	\begin{defn}
			For a graph $\Lambda$ and a set $L$ of edges in $\Lambda$, $\Lambda\setminus L$ is the
					graph obtained from $\Lambda$ by removing the edges in $L$.
			$L$ \emph{disconnects} $\Lambda$ if $\Lambda\setminus L$ is weakly disconnected, meaning
					$\Lambda\setminus L$ contains more than one weakly connected component.
	\end{defn}
	
	In this work, $L$ will contain loops, and we want to consider $L$ to be a collection of loops rather than a collection of edges.
	
	\begin{defn}
			If $\PPP = \{L_1,\dots,L_\kappa\}$ is a partition of a set $L$ of edges of $\Lambda$ and $L$ disconnects $\Lambda$,
					then $\PPP$ is a \emph{minimal disconnecting partition} if
					for each $i \in \{1,\dots,\kappa\}$, $L\setminus L_i$ does not disconnect $\Lambda$.
	\end{defn}
	
	When $\PPP$ is understood, we will suppress its appearance in the notation.
	According to item (3) of Rules List \ref{RULE3}, it may be the case that the $\nu$-colored
		edge leaving the $2$-loop $L_0$ that ``spreads'' is also the $\nu$-colored edge entering the $2$-loop.
		This is possible if there are multiple edges from right special $v$ to left special $u$ that are visited by $L_0$.
		Such a case remains compatible with our arguments below as only one edge from $v$ to $u$ is
			included in $L_0$ while any other such edge is part of the graph $\Lambda' \setminus L_0$.
		In particular, $L_0$ would never be part of a minimal disconnecting partition (as in the proof of Lemma~\ref{LEM2LOOPS}) because there
			exists at least one edge from $v$ to $u$ in $\Lambda'\setminus L_0$.

	\begin{rema}
		If $L$ disconnects $\Lambda$ and $\PPP$ is a partition of $L$,
			then there exists $\PPP'\subseteq \PPP$ that is a minimal disconnecting partition
			with disconnecting set $L' = \cup_{L_i \in \PPP'} L_i$,
			although the choice of $\PPP'$ may not be unique.
	\end{rema}
	
	\begin{lemm}\label{LEM2LOOPS}
		Let $\Lambda$ and $\CCC$ satisfy Rules List \ref{RULE1} and suppose we may apply Rules List \ref{RULE3}
				to any choice of nonzero colored 2-loops. 
			If $L$ is a union of all edges from a collection of distinctly colored 2-loops in $\Lambda$, then $L$ cannot disconnect $\Lambda$.
	\end{lemm}
	
	\begin{proof}
		We will use a separation argument.
		Assume for a contradiction that $L$ disconnects $\Lambda$.
		Let $\PPP$ be the partition of $L$ into the 2-loops,
			and choose a minimal disconnecting partition $\PPP' \subseteq \PPP$.
		Then $L' = \bigcup_{L_i\in \PPP'} L_i$ disconnects $\Lambda$.
		If $u$ and $v$ are the left special and right special vertices respectively
			of a 2-loop in $L_i\in \PPP'$,
			let $C_u$ and $C_v$ be the respective weakly connected components containing
			$u$ and $v$ in $\Lambda \setminus L'$.
		If $C_u = C_v$, then $L\setminus L_i$ is a disconnecting set of edges for $\Lambda$,
			contradicting that $\PPP'$ is a minimal partition.
		Therefore $C_u \neq C_v$.
			
		By Rules List \ref{RULE3}, accelerate to $\Lambda'$ with $\CCC'$
			using $\PPP'$ as our chosen fixed 2-loops and let
				 $\mathfrak{L}\in \PPP'$ be a
			2-loop whose color ``spreads'' in $\Lambda'$ with non-zero color $\nu$
			(from item (3) of Rules List~\ref{RULE3}).
		Because only moves of type 1 (those that preserve a 2-loop) occur from $\Lambda$ to $\Lambda'$,
			any other RBS moves must occur on edges whose vertices are not in the 2-loops from $\PPP'$ and these do not
			alter the connected components (viewed as vertex sets) from $\Lambda\setminus \PPP'$.
		Therefore $\Lambda'$ is disconnected by $\mathcal{P}'$.
		If $u$ (resp. $v$) is the left (resp. right) special vertex in the 2-loop $\mathfrak{L}$, then
			by Rules List \ref{RULE1} there must exist a directed path in $\Lambda'$ with color $\nu$
			that begins at $v$, ends at $u$ and does not contain the edge 
			from $v$ to $u$ in $\mathfrak{L}$.
		Furthermore, this path cannot intersect any edge in $L'$ as the other 2-loops
			have non-zero colors different than $\nu$.
		This implies that the weakly connected component containing $v$ in $\Lambda'\setminus \mathfrak{L}$
			is equal the weakly connected component containing $u$, a contradiction.
	\end{proof}
	
	\begin{coro}\label{COR2LOOPS}
		If $\Lambda$ and $\CCC$ as are in Lemma \ref{LEM2LOOPS}, then there may be at most $\frac{K+1}{2}$
			distinct colors on 2-loops.
	\end{coro}
	
	\begin{proof}
		Let $\PPP = \{L_1,\dots,L_{\kappa}\}$ be a set of distinctly colored 2-loops in $\Lambda$ such that
			every color $\nu$ on a 2-loop in $\Lambda$ is included and let $L = \bigcup_{i=1}^\kappa L_i$.
		Let $K_\mathfrak{s}$ be the number of $\mathfrak{s}$-special vertices in $\Lambda$, $\mathfrak{s}\in \{\ell,r\}$.
		Removing the $2\kappa$ edges of $L$ from $\Lambda$ results in graph $\Lambda \setminus L$ with
			$K_\ell + K_r$ vertices and $K + K_\ell + K_r - 2 \kappa$ edges.
		By Lemma \ref{LEM2LOOPS}, $L$ cannot disconnect $\Lambda$, and to be weakly connected,
			a graph with $V$ vertices must have at least $V-1$ edges.
		Therefore,
			$$ K + K_\ell + K_r - 2\kappa \geq K_\ell + K_r - 1 \Rightarrow \kappa \leq \frac{K+1}{2},$$
		as desired.
	\end{proof}
	
	\subsection{Larger Loops: Behavior}\label{SEC_NLOOPS_MOVES}
	
	While 2-loops are relatively simple, Rules List \ref{RULE1} does allow for $\nu\in \{1,\dots,E\}$
		to color more than two vertices.
	However, in such an event we may still associate to $\nu$ an \emph{$N$-loop}, $N\geq 2$,
		meaning a directed circuit of $N$ edges in $\Lambda$ colored by $\nu$.
	Note that $\nu$ may color more edges and we choose for simplicity to restrict our attention to just one $N$-loop
		per color.
	Having extra colored edges will not alter or otherwise interfere with our arguments.
	By item (4) of Rules List~\ref{RULE1}, for every $\nu \in \{1, \ldots, E\}$, there is an $N$-loop with color $\nu$.

	The allowable RBS moves on an $N$-loop, $N\geq 3$, will now be described.
	Two of the moves will be analogous to those for 2-loops, while a third move
		will only apply to larger loops. 
	As in the previous sections, let $e_0$ denote the bispecial edge from $u$ to $v$
		in our $N$-loop.
	Using prior notation for edges $e_1,\dots,e_L,\tilde{e}_1,\dots,\tilde{e}_R$
		and vertices $y_1,\dots,y_L,z_1,\dots,z_R$,
		assume $e_1$ is the edge ending at $u$ in the $N$-loop and $\tilde{e}_1$ is the edge leaving $v$
		in the $N$-loop.
		The following list describes all allowable types of RBS moves.

      \begin{figure}[t!]
	    \begin{center}
		\subfigure[A twist RBS move.]{\label{FIG_NLOOP_TWIST}
		      \begin{tikzpicture}[vertstyle/.style={draw=black, top color = black!5, bottom color = black!35},
			edgestyle/.style={->,thick,black}, dotedgestyle/.style={thick,dotted,black},
			gbox/.style={thick,black!20,fill=black!10}]
	\def\asp{0.13}

	\draw [gbox] (19*\asp, -39*\asp) rectangle +(42*\asp,-24*\asp);
	\draw [gbox] (-31*\asp, -39*\asp) rectangle +(42*\asp,-24*\asp);
		
	\draw [->,very thick, black!20] (11*\asp, -52*\asp) -- (19*\asp, -52*\asp);
	\draw (15*\asp,-50*\asp) node {RBS};
	

	\draw [dotedgestyle] (-25*\asp,-53*\asp) -- +(-5*\asp, 2*\asp);
	\draw [dotedgestyle] (-25*\asp,-53*\asp) -- +(-5*\asp, -2*\asp);
	\draw [dotedgestyle] (-25*\asp,-59*\asp) -- +(-5*\asp, 2*\asp);
	\draw [dotedgestyle] (-25*\asp,-59*\asp) -- +(-5*\asp, -2*\asp);

	\draw [edgestyle] (-23.3*\asp,-52*\asp) -- (-16.7*\asp,-50*\asp);
	\draw [edgestyle] (-23.6*\asp,-57.6*\asp) -- (-16*\asp,-50.7*\asp);

	\draw [edgestyle] (-13*\asp,-49*\asp) -- +(6*\asp,0);
	
	\draw [edgestyle] (-3.3*\asp,-50*\asp) -- (3.3*\asp,-52*\asp);
	\draw [edgestyle] (-4*\asp,-50.7*\asp) -- (3.6*\asp,-57.6*\asp);
	
	\draw [dotedgestyle] (5*\asp,-53*\asp) -- +(5*\asp, 2*\asp);
	\draw [dotedgestyle] (5*\asp,-53*\asp) -- +(5*\asp, -2*\asp);
	\draw [dotedgestyle] (5*\asp,-59*\asp) -- +(5*\asp, 2*\asp);
	\draw [dotedgestyle] (5*\asp,-59*\asp) -- +(5*\asp, -2*\asp);
	
	\draw [edgestyle] (-20*\asp, -47*\asp) arc (180:270:2*\asp) -- (-17*\asp, -49*\asp);
	\draw [edgestyle] (-3*\asp, -49*\asp) -- (-2*\asp, -49*\asp) arc (270:360:2*\asp);
	\draw [dotedgestyle,->]  (0*\asp, -43*\asp) arc (0:90:2*\asp) -- (-18*\asp, -41*\asp) arc (90:180:2*\asp);
	\draw (-10*\asp,-43*\asp) node {\scalebox{0.75}{$N$-loop}};
	
	\shade [vertstyle] (-25*\asp,-53*\asp) circle (2*\asp) node {$y_2$};
	\shade [vertstyle] (-25*\asp,-59*\asp) circle (2*\asp) node {$y_3$};

	\shade [vertstyle] (-15*\asp,-49*\asp) circle (2*\asp) node {$u$};
	\shade [vertstyle] (-5*\asp,-49*\asp) circle (2*\asp) node {$v$};
	
	\shade [vertstyle] (5*\asp,-53*\asp) circle (2*\asp) node {$z_2$};
	\shade [vertstyle] (5*\asp,-59*\asp) circle (2*\asp) node {$z_3$};
	
	\shade [vertstyle] (-20*\asp, -45*\asp) circle (2*\asp) node {$y_1$};
	\shade [vertstyle] (0*\asp, -45*\asp) circle (2*\asp) node {$z_1$};



	\draw [dotedgestyle] (25*\asp,-53*\asp) -- +(-5*\asp, 2*\asp);
	\draw [dotedgestyle] (25*\asp,-53*\asp) -- +(-5*\asp, -2*\asp);
	\draw [dotedgestyle] (25*\asp,-59*\asp) -- +(-5*\asp, 2*\asp);
	\draw [dotedgestyle] (25*\asp,-59*\asp) -- +(-5*\asp, -2*\asp);

	\draw [edgestyle] (27*\asp,-53*\asp) -- (43.3*\asp,-50*\asp);
	\draw [edgestyle] (27*\asp,-59*\asp) -- (44*\asp,-50.7*\asp);

	\draw [edgestyle] (37*\asp,-49*\asp) -- +(6*\asp,0);
	
	\draw [edgestyle] (36*\asp,-50.0*\asp) -- (36*\asp, -52*\asp) arc(180:270:1*\asp) -- (53*\asp,-53*\asp);
	\draw [edgestyle] (34*\asp,-50.7*\asp) -- (34*\asp, -57*\asp) arc(180:270:2*\asp) -- (53*\asp,-59*\asp);
	
	\draw [dotedgestyle] (55*\asp,-53*\asp) -- +(5*\asp, 2*\asp);
	\draw [dotedgestyle] (55*\asp,-53*\asp) -- +(5*\asp, -2*\asp);
	\draw [dotedgestyle] (55*\asp,-59*\asp) -- +(5*\asp, 2*\asp);
	\draw [dotedgestyle] (55*\asp,-59*\asp) -- +(5*\asp, -2*\asp);
	
	\draw [edgestyle] (30*\asp, -47*\asp) arc (180:270:2*\asp) -- (33*\asp, -49*\asp);
	\draw [edgestyle] (47*\asp, -49*\asp) -- (48*\asp, -49*\asp) arc (270:360:2*\asp);
	\draw [dotedgestyle,->]  (50*\asp, -43*\asp) arc (0:90:2*\asp) -- (32*\asp, -41*\asp) arc (90:180:2*\asp);
	\draw (40*\asp,-43*\asp) node {\scalebox{0.75}{$N$-loop}};
	
	\shade [vertstyle] (25*\asp,-53*\asp) circle (2*\asp) node {$y_2$};
	\shade [vertstyle] (25*\asp,-59*\asp) circle (2*\asp) node {$y_3$};

	\shade [vertstyle] (35*\asp,-49*\asp) circle (2*\asp) node {$v$};
	\shade [vertstyle] (45*\asp,-49*\asp) circle (2*\asp) node {$u$};
	
	\shade [vertstyle] (55*\asp,-53*\asp) circle (2*\asp) node {$z_2$};
	\shade [vertstyle] (55*\asp,-59*\asp) circle (2*\asp) node {$z_3$};

	\shade [vertstyle] (30*\asp, -45*\asp) circle (2*\asp) node {$y_1$};
	\shade [vertstyle] (50*\asp, -45*\asp) circle (2*\asp) node {$z_1$};

	\draw (-10*\asp, -65*\asp) node {\scalebox{1.5}{$\Lambda$}};
	\draw (40*\asp, -65*\asp) node {\scalebox{1.5}{$\Lambda'$}};

\end{tikzpicture}
		     	}
	    \end{center}
	    \begin{center}
		\subfigure[A shrink RBS move. In $\Lambda'$ a path from $y_3$ to $z_3$ exists that avoids the loop.]{\label{FIG_NLOOP_SHRINK}
		      \begin{tikzpicture}[vertstyle/.style={draw=black, top color = black!5, bottom color = black!35},
			edgestyle/.style={->,thick,black}, dotedgestyle/.style={thick,dotted,black},
			gbox/.style={thick,black!20,fill=black!10}]
	\def\asp{0.13}

	\draw [gbox] (19*\asp, -39*\asp) rectangle +(42*\asp,-24*\asp);
	\draw [gbox] (-31*\asp, -39*\asp) rectangle +(42*\asp,-24*\asp);
		
	\draw [->,very thick, black!20] (11*\asp, -52*\asp) -- (19*\asp, -52*\asp);
	\draw (15*\asp,-50*\asp) node {RBS};
	

	\draw [dotedgestyle] (-25*\asp,-53*\asp) -- +(-5*\asp, 2*\asp);
	\draw [dotedgestyle] (-25*\asp,-53*\asp) -- +(-5*\asp, -2*\asp);
	\draw [dotedgestyle] (-25*\asp,-59*\asp) -- +(-5*\asp, 2*\asp);
	\draw [dotedgestyle] (-25*\asp,-59*\asp) -- +(-5*\asp, -2*\asp);

	\draw [edgestyle] (-23.3*\asp,-52*\asp) -- (-16.7*\asp,-50*\asp);
	\draw [edgestyle] (-23.6*\asp,-57.6*\asp) -- (-16*\asp,-50.7*\asp);

	\draw [edgestyle] (-13*\asp,-49*\asp) -- +(6*\asp,0);
	
	\draw [edgestyle] (-3.3*\asp,-50*\asp) -- (3.3*\asp,-52*\asp);
	\draw [edgestyle] (-4*\asp,-50.7*\asp) -- (3.6*\asp,-57.6*\asp);
	
	\draw [dotedgestyle] (5*\asp,-53*\asp) -- +(5*\asp, 2*\asp);
	\draw [dotedgestyle] (5*\asp,-53*\asp) -- +(5*\asp, -2*\asp);
	\draw [dotedgestyle] (5*\asp,-59*\asp) -- +(5*\asp, 2*\asp);
	\draw [dotedgestyle] (5*\asp,-59*\asp) -- +(5*\asp, -2*\asp);
	
	\draw [edgestyle] (-20*\asp, -47*\asp) arc (180:270:2*\asp) -- (-17*\asp, -49*\asp);
	\draw [edgestyle] (-3*\asp, -49*\asp) -- (-2*\asp, -49*\asp) arc (270:360:2*\asp);
	\draw [dotedgestyle,->]  (0*\asp, -43*\asp) arc (0:90:2*\asp) -- (-18*\asp, -41*\asp) arc (90:180:2*\asp);
	\draw (-10*\asp,-43*\asp) node {\scalebox{0.75}{$N$-loop}};
	
	\shade [vertstyle] (-25*\asp,-53*\asp) circle (2*\asp) node {$y_2$};
	\shade [vertstyle] (-25*\asp,-59*\asp) circle (2*\asp) node {$y_3$};

	\shade [vertstyle] (-15*\asp,-49*\asp) circle (2*\asp) node {$u$};
	\shade [vertstyle] (-5*\asp,-49*\asp) circle (2*\asp) node {$v$};
	
	\shade [vertstyle] (5*\asp,-53*\asp) circle (2*\asp) node {$z_2$};
	\shade [vertstyle] (5*\asp,-59*\asp) circle (2*\asp) node {$z_3$};
	
	\shade [vertstyle] (-20*\asp, -45*\asp) circle (2*\asp) node {$y_1$};
	\shade [vertstyle] (0*\asp, -45*\asp) circle (2*\asp) node {$z_1$};



	\draw [dotedgestyle] (25*\asp,-53*\asp) -- +(-5*\asp, 2*\asp);
	\draw [dotedgestyle] (25*\asp,-53*\asp) -- +(-5*\asp, -2*\asp);
	\draw [dotedgestyle] (25*\asp,-59*\asp) -- +(-5*\asp, 2*\asp);
	\draw [dotedgestyle] (25*\asp,-59*\asp) -- +(-5*\asp, -2*\asp);

	\draw [edgestyle] (26.7*\asp,-52*\asp) -- (33.3*\asp,-50*\asp);
	\draw [edgestyle] (27*\asp,-59*\asp) -- (43*\asp,-59*\asp);

	\draw [edgestyle] (36.4*\asp,-50.4*\asp) -- (43.6*\asp, -57.6*\asp);
	
	\draw [edgestyle] (36.7*\asp,-50*\asp) -- (53.3*\asp,-52*\asp);
	\draw [edgestyle] (47*\asp,-59*\asp) -- (53*\asp,-59*\asp);
	
	\draw [dotedgestyle] (55*\asp,-53*\asp) -- +(5*\asp, 2*\asp);
	\draw [dotedgestyle] (55*\asp,-53*\asp) -- +(5*\asp, -2*\asp);
	\draw [dotedgestyle] (55*\asp,-59*\asp) -- +(5*\asp, 2*\asp);
	\draw [dotedgestyle] (55*\asp,-59*\asp) -- +(5*\asp, -2*\asp);
	
	\draw [edgestyle] (30*\asp, -47*\asp) arc (180:270:2*\asp) -- (33*\asp, -49*\asp);
	\draw [edgestyle] (37*\asp, -49*\asp) -- (48*\asp, -49*\asp) arc (270:360:2*\asp);
	\draw [dotedgestyle,->]  (50*\asp, -43*\asp) arc (0:90:2*\asp) -- (32*\asp, -41*\asp) arc (90:180:2*\asp);
	\draw (40*\asp,-43*\asp) node {\scalebox{0.75}{$(N-1)$-loop}};
	
	\shade [vertstyle] (25*\asp,-53*\asp) circle (2*\asp) node {$y_2$};
	\shade [vertstyle] (25*\asp,-59*\asp) circle (2*\asp) node {$y_3$};

	\shade [vertstyle] (35*\asp,-49*\asp) circle (2*\asp) node {$v$};
	\shade [vertstyle] (45*\asp,-59*\asp) circle (2*\asp) node {$u$};
	
	\shade [vertstyle] (55*\asp,-53*\asp) circle (2*\asp) node {$z_2$};
	\shade [vertstyle] (55*\asp,-59*\asp) circle (2*\asp) node {$z_3$};

	\shade [vertstyle] (30*\asp, -45*\asp) circle (2*\asp) node {$y_1$};
	\shade [vertstyle] (50*\asp, -45*\asp) circle (2*\asp) node {$z_1$};

	\draw (-10*\asp, -65*\asp) node {\scalebox{1.5}{$\Lambda$}};
	\draw (40*\asp, -65*\asp) node {\scalebox{1.5}{$\Lambda'$}};

\end{tikzpicture}
		     }
	    \end{center}
	    \begin{center}
		\subfigure[A collapse RBS move. Here $\Lambda'$ has been ``untangled.'']{\label{FIG_NLOOP_SPREAD}
		      \begin{tikzpicture}[vertstyle/.style={draw=black, top color = black!5, bottom color = black!35},
			edgestyle/.style={->,thick,black}, dotedgestyle/.style={thick,dotted,black},
			gbox/.style={thick,black!20,fill=black!10}]
	\def\asp{0.13}

	\draw [gbox] (19*\asp, -39*\asp) rectangle +(42*\asp,-24*\asp);
	\draw [gbox] (-31*\asp, -39*\asp) rectangle +(42*\asp,-24*\asp);
		
	\draw [->,very thick, black!20] (11*\asp, -52*\asp) -- (19*\asp, -52*\asp);
	\draw (15*\asp,-50*\asp) node {RBS};
	

	\draw [dotedgestyle] (-25*\asp,-53*\asp) -- +(-5*\asp, 2*\asp);
	\draw [dotedgestyle] (-25*\asp,-53*\asp) -- +(-5*\asp, -2*\asp);
	\draw [dotedgestyle] (-25*\asp,-59*\asp) -- +(-5*\asp, 2*\asp);
	\draw [dotedgestyle] (-25*\asp,-59*\asp) -- +(-5*\asp, -2*\asp);

	\draw [edgestyle] (-23.3*\asp,-52*\asp) -- (-16.7*\asp,-50*\asp);
	\draw [edgestyle] (-23.6*\asp,-57.6*\asp) -- (-16*\asp,-50.7*\asp);

	\draw [edgestyle] (-13*\asp,-49*\asp) -- +(6*\asp,0);
	
	\draw [edgestyle] (-3.3*\asp,-50*\asp) -- (3.3*\asp,-52*\asp);
	\draw [edgestyle] (-4*\asp,-50.7*\asp) -- (3.6*\asp,-57.6*\asp);
	
	\draw [dotedgestyle] (5*\asp,-53*\asp) -- +(5*\asp, 2*\asp);
	\draw [dotedgestyle] (5*\asp,-53*\asp) -- +(5*\asp, -2*\asp);
	\draw [dotedgestyle] (5*\asp,-59*\asp) -- +(5*\asp, 2*\asp);
	\draw [dotedgestyle] (5*\asp,-59*\asp) -- +(5*\asp, -2*\asp);
	
	\draw [edgestyle] (-20*\asp, -47*\asp) arc (180:270:2*\asp) -- (-17*\asp, -49*\asp);
	\draw [edgestyle] (-3*\asp, -49*\asp) -- (-2*\asp, -49*\asp) arc (270:360:2*\asp);
	\draw [dotedgestyle,->]  (0*\asp, -43*\asp) arc (0:90:2*\asp) -- (-18*\asp, -41*\asp) arc (90:180:2*\asp);
	\draw (-10*\asp,-43*\asp) node {\scalebox{0.75}{$N$-loop}};
	
	\shade [vertstyle] (-25*\asp,-53*\asp) circle (2*\asp) node {$y_2$};
	\shade [vertstyle] (-25*\asp,-59*\asp) circle (2*\asp) node {$y_3$};

	\shade [vertstyle] (-15*\asp,-49*\asp) circle (2*\asp) node {$u$};
	\shade [vertstyle] (-5*\asp,-49*\asp) circle (2*\asp) node {$v$};
	
	\shade [vertstyle] (5*\asp,-53*\asp) circle (2*\asp) node {$z_2$};
	\shade [vertstyle] (5*\asp,-59*\asp) circle (2*\asp) node {$z_3$};
	
	\shade [vertstyle] (-20*\asp, -45*\asp) circle (2*\asp) node {$y_1$};
	\shade [vertstyle] (0*\asp, -45*\asp) circle (2*\asp) node {$z_1$};



	\draw [dotedgestyle] (25*\asp,-53*\asp) -- +(-5*\asp, 2*\asp);
	\draw [dotedgestyle] (25*\asp,-53*\asp) -- +(-5*\asp, -2*\asp);
	\draw [dotedgestyle] (25*\asp,-59*\asp) -- +(-5*\asp, 2*\asp);
	\draw [dotedgestyle] (25*\asp,-59*\asp) -- +(-5*\asp, -2*\asp);

	\draw [edgestyle] (26.7*\asp,-52*\asp) -- (33.3*\asp,-50*\asp);
	\draw [edgestyle] (26.7*\asp,-58*\asp) -- (43.3*\asp,-50*\asp);

	\draw [edgestyle] (43*\asp,-49*\asp) -- +(-6*\asp,0);
	
	\draw [edgestyle] (46.7*\asp,-50*\asp) -- (53.3*\asp,-52*\asp);
	\draw [edgestyle] (36.7*\asp,-50*\asp) -- (53.3*\asp,-58*\asp);
	
	\draw [dotedgestyle] (55*\asp,-53*\asp) -- +(5*\asp, 2*\asp);
	\draw [dotedgestyle] (55*\asp,-53*\asp) -- +(5*\asp, -2*\asp);
	\draw [dotedgestyle] (55*\asp,-59*\asp) -- +(5*\asp, 2*\asp);
	\draw [dotedgestyle] (55*\asp,-59*\asp) -- +(5*\asp, -2*\asp);
	
	\draw [edgestyle] (30*\asp, -47*\asp) arc (180:270:2*\asp) -- (33*\asp, -49*\asp);
	\draw [edgestyle] (47*\asp, -49*\asp) -- (48*\asp, -49*\asp) arc (270:360:2*\asp);
	\draw [dotedgestyle,->]  (50*\asp, -43*\asp) arc (0:90:2*\asp) -- (32*\asp, -41*\asp) arc (90:180:2*\asp);
	\draw (40*\asp,-43*\asp) node {\scalebox{0.75}{not a loop}};
	
	\shade [vertstyle] (25*\asp,-53*\asp) circle (2*\asp) node {$y_2$};
	\shade [vertstyle] (25*\asp,-59*\asp) circle (2*\asp) node {$y_3$};

	\shade [vertstyle] (35*\asp,-49*\asp) circle (2*\asp) node {$u$};
	\shade [vertstyle] (45*\asp,-49*\asp) circle (2*\asp) node {$v$};
	
	\shade [vertstyle] (55*\asp,-53*\asp) circle (2*\asp) node {$z_2$};
	\shade [vertstyle] (55*\asp,-59*\asp) circle (2*\asp) node {$z_3$};

	\shade [vertstyle] (30*\asp, -45*\asp) circle (2*\asp) node {$y_1$};
	\shade [vertstyle] (50*\asp, -45*\asp) circle (2*\asp) node {$z_1$};

	\draw (-10*\asp, -65*\asp) node {\scalebox{1.5}{$\Lambda$}};
	\draw (40*\asp, -65*\asp) node {\scalebox{1.5}{$\Lambda'$}};

\end{tikzpicture}
		      }
	    \end{center}
	    \caption{The types of RBS moves on $u\to v$ in an $N$-loop.}\label{FIG_NLOOPS}
      \end{figure}
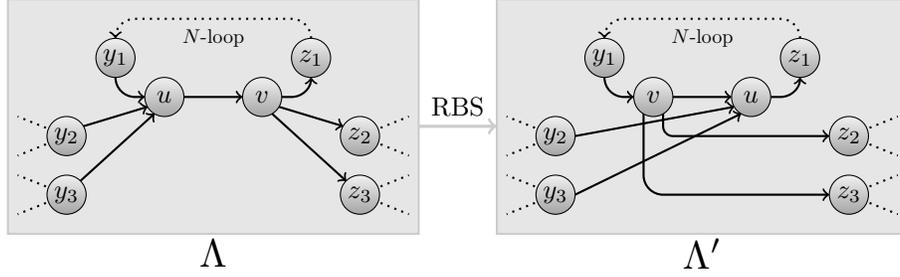
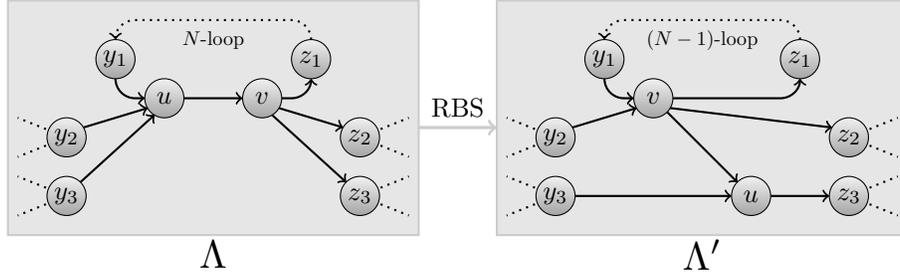
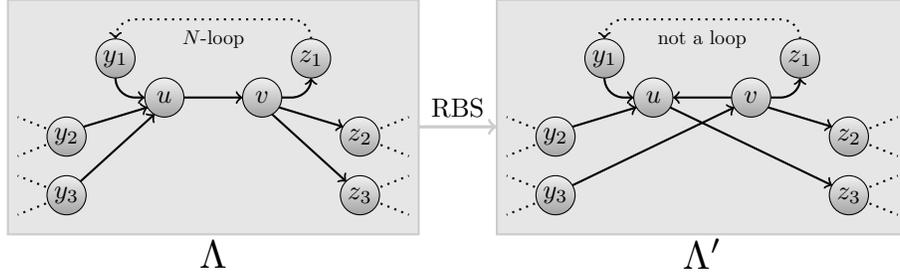
	
		\begin{enumerate}
			\item \emph{Twist move:} 
					In this case $i_0 = j_0 = 1$.
					After untangling, the order of $u$ and $v$ is reversed in the $N$-loop, while all other edges are preserved.
					In a 2-loop, this would simply preserve the loop.
					See Figure \ref{FIG_NLOOP_TWIST}.
					
			\item \emph{Shrink move(s):} 
					Either $i_0 =1$ and $j_0 \neq 1$ or $i_0 \neq 1$ and $j_0 = 1$.
					After the move, one of the vertices is removed from the $N$-loop, resulting in an $(N-1)$-loop.
					Figure \ref{FIG_NLOOP_SHRINK} shows a shrink move that removes $u$ from the $N$-loop
						and there is an analogous move that removes $v$ from the $N$-loop instead. 
						This move is only allowed if the removed $\mathfrak{s}$-special vertex is not the only
						$\mathfrak{s}$-special vertex in the $N$-loop initially.
						For example, in Figure \ref{FIG_NLOOP_SHRINK} either $y_1$ or $z_1$ must be left sepcial.
						This is never allowed in a $2$-loop.

			\item \emph{Collapse move:}
				In this last case 
					$i_0 \neq 1$ and $j_0 \neq 1$.
				After the move, the $N$-loop becomes collapsed
					into two paths given by edge $e_0$ from $v$ to $u$ and a new path
					from $v$ to $u$ along the remaining $N-1$ edges.
					This is shown in Figure \ref{FIG_NLOOP_SPREAD} and is the analogue of the second type of move for 2-loops described at the beginning of
						Section \ref{SEC_2LOOPS}. 
					By Rules Lists \ref{RULE1} and \ref{RULE2}, a non-zero color on the $N$-loop in $\Lambda$
						must agree with the coloring on  $e_{i_0}$ and $\tilde{e}_{j_0}$ (``outside'' edges) by both $\CCC$ and $\CCC'$.
						
		\end{enumerate}

		Note that a twist move only permutes the order of two vertices within the $N$-loop itself but does not otherwise alter the
			graph structure.
		For example, the outside edges that end at $u$ still do so after the move and likewise for
				the outside edges beginning at $v$.
		Thus each $N$-loop is
			preserved for a number of moves of the first type until a move of either the
			second or third type occurs.
		Similar to the 2-loop case, we want to focus on a collection of $N$-loops
			until either a shrink move occurs or an $N$-loop passes its color to outside edges.
		We again do not know (or care) what happens to the other colors and want to apply a
			separation argument.
			
		However, we are met with a complication due to shrink moves.
		The general strategy of the separation argument, as introduced in Section \ref{SEC_2LOOPS},
			is that we begin with a purported minimal disconnecting partition
			(of $N$-loops here)
			in $\Lambda$ and accelerate the graph to $\Lambda'$ when one of the $N$-loops spreads its color.
		If the given collection of $N$-loops is still a minimal disconnecting partition in $\Lambda'$,
			we get a contradiction, as one loop spreads color to another,
			but this assumption may no longer hold.
		Consider the move in Figure \ref{FIG_NLOOP_SHRINK} with $N$-loop colored by $\nu$. 
		Let us assume that every directed path from $y_3$ to $z_3$ in $\Lambda$ touches $u$,
			which is part of the pictured $N$-loop.
		Thus the $N$-loop disconnects $y_3$ from $z_3$,
			and so no color $\nu' \neq \nu$ can pass from $y_3$ to $z_3$.
		However, after the shrink move that transforms $\Lambda$ to $\Lambda'$,
			the resulting $(N-1)$-loop no longer disconnects $y_3$ from $z_3$.
		If $u$ is not colored by $\nu$ in $\Lambda'$,
			then the color $\nu'$ can freely pass from $y_3$ to $z_3$.
		This demonstrates the fact that a minimal disconnecting partition in $\Lambda$ need not be one in $\Lambda'$.
		We will address this issue in two steps.
		In Section \ref{SEC_NLOOPS_ITIN} we will form the $N$-loop analogue of Rules List \ref{RULE3}.
		In Section \ref{SEC_NLOOPS_NEWG} we will make an auxiliary graph on which we will form our
			minimal separating partition argument for our main proof.
			
	\subsection{Larger Loops: Itineraries}\label{SEC_NLOOPS_ITIN}
	
		We will use the previous section's notions of $N$-loops and outside edges,
			meaning edges that are not part of the $N$-loop but either begin or end at a vertex visited by the $N$-loop.
		The next definition is meant to articulate the following idea:
			we want to select a set of loops and ``follow them'' while they undergo RBS moves.
			Twist moves do not alter the graph outside of the loops, so we want to accelerate through RBS moves
				until at least one loop either shrinks or passes its color along at least
				one pair of outside edges. 
		In the case of 2-loops, this was given simply by Rules List \ref{RULE3}.

		\begin{defn}
			If $\Lambda$ has a colored $N$-loop $L$ (colored by the function $\mathcal{C}$)
				and $\Lambda'$ is the result of applying a finite ordered list of RBS moves to $\Lambda$ such that:
				\begin{enumerate}
					\item the only moves acting on vertices of $L$ are twist and shrink moves,
					\item the color on $L$ is preserved by $\CCC'$ for $\Lambda'$ and
					\item the color of $L$ passes along at least one outside edge entering and one
						outside edge leaving $L$ in $\Lambda'$ (this may be one edge that both leaves and enters the $N$-loop),
				\end{enumerate}
				then we say that $L$ experiences a \emph{spread event} from $\Lambda$ to $\Lambda'$.
		\end{defn}
		
		As we will discuss in Section \ref{SEC_ERG_TO_GRAPH}, in the graphs we construct from subshifts,
			every $N$-loop must eventually undergo a collapse move and must therefore experience
			a spread event.
			
		The next definition is used in our final Rules List.
		As with the other lists, this one will be justified in Section \ref{SEC_ERG_TO_GRAPH} for graphs $\Lambda$ constructed from subshifts.
	
		\begin{defn}\label{DEF_ITIN}
			Given $\Lambda$ with coloring $\CCC$ and fixed collection of $N$-loops $\PPP$ (each colored by a distinct nonzero color),
				an \emph{itinerary} is a sequence
					$$ \left(\Lambda^{(i)},\CCC^{(i)},\PPP^{(i)}\right)_{i=0}^{M}$$\\
				of graphs, colorings, and $N$-loop collections respectively, and a sequence $\mathsf{L} = (\mathsf{L}_i)_{i=0}^{M-1}$ of finite ordered lists of RBS moves so that $$\left(\Lambda^{(0)},\CCC^{(0)},\PPP^{(0)}\right)
					= (\Lambda,\CCC,\PPP)$$ and if $i\in \{0,1,\dots,M-1\}$:
					\begin{enumerate}
						\item $\Lambda^{(i+1)}$ is obtained from $\Lambda^{(i)}$ by using the list $\mathsf{L}_i$ 
								 of RBS moves.
						\item The only moves acting on $N$-loops from $\PPP^{(i)}$ from $\Lambda^{(i)}$ to $\Lambda^{(i+1)}$
								are the twist and shrink moves described in the previous section.
						\item 
							At least one of the following occurs from $\Lambda^{(i)}$ to $\Lambda^{(i+1)}$:
								\begin{enumerate}
								\item at least one loop in $\PPP^{(i)}$ is shrunk by at least one RBS move,
								\item at least one loop in $\PPP^{(i)}$ undergoes a spread event.
								\end{enumerate}
								At most one type of event (shrink or spread) may occur per
									$N$-loop in $\PPP^{(i)}$ from $\Lambda^{(i)}$ to $\Lambda^{(i+1)}$.
									Multiple shrink moves are permitted to simultaneously
										occur on a particular $N$-loop in $\PPP^{(i)}$
										(these moves will act on edges that pairwise do not share
											endpoints in the $N$-loop).

						\item If an $N$-loop in $\PPP^{(i)}$ undergoes shrink moves and twist moves from $\Lambda^{(i)}$ to $\Lambda^{(i+1)}$,
									then the list $\mathsf{L}_i$ contains all twist moves before any of the shrink moves.
						\item 	The loops from $\PPP^{(i+1)}$ are the loops in $\PPP^{(i)}$ that did not spread, minus
								the edges (and therefore vertices) jettisoned by shrink moves.
								
						\item $\CCC^{(i+1)}$ agrees with $\CCC^{(i)}$ on $\PPP^{(i+1)}$ and any $N$-loops in $\PPP^{(i)}$ that spread. 
							Any vertices lost by shrink moves will be uncolored by $\CCC^{(i+1)}$.
						\item $\PPP^{(i)} \neq \emptyset$ for $i<M$ and $\PPP^{(M)} = \emptyset$.
					\end{enumerate}
		\end{defn}
		
		We now are able to give our new rules list.
		
		\begin{rules}[Acceleration for $N$-loops]\label{RULE4}
			Given a fixed choice $\PPP$ of $N$-loops in $\Lambda$ with coloring $\CCC$ satisfying Rules List \ref{RULE1} 
			(and each colored by a distinct nonzero color),
				there exists an itinerary
					$$\left(\Lambda^{(i)},\CCC^{(i)},\PPP^{(i)}\right)_{i=0}^{M}.$$
			For any other choice $\PPP_*\subsetneq \PPP$ of $N$-loops,
				there exists an itinerary
					$$\left(\Lambda_*^{(j)},\CCC_*^{(j)},\PPP_*^{(j)}\right)_{j=0}^{M_*}$$
				such that for some $0= i_0 < \dots < i_{M_*} \leq M$ and all $0 \leq j \leq M_*$,
					\begin{enumerate}
						\item $\Lambda_*^{(j)} = \Lambda^{(i_j)}$,
						\item $\CCC_*^{(j)} = \CCC^{(i_j)}$ and
						\item $ \PPP_*^{(j)} = \PPP^{(i_j)} \wedge \PPP_*$, meaning $\PPP_*^{(j)}$ is the collection of all
								sets of the form $L \cap L'$ where $L \in \PPP^{(i_j)}$ and $L'\in \PPP_*$
								such that the intersection is non-empty.
					\end{enumerate}
		\end{rules}
		
		An itinerary notes the important sequential events that occur for each loop in $\PPP$
			until they are no longer of interest.
		We may then start over with a smaller set of loops $\PPP_*\subseteq \PPP$.
		For these loops, the relevant events we initially observed for $\PPP$ will remain and occur in the same order.

	\subsection{Larger Loops: Graph \texorpdfstring{$\Xi$}{Xi}}\label{SEC_NLOOPS_NEWG}
	
		In order to find a minimal separating subset for a collection of $N$-loops $\PPP$
			we need to take into account the effect of shrink moves because these moves can merge components together:
		as discussed previously, a shrink move may increase connectivity in $\Lambda^{(i+1)}\setminus \PPP^{(i+1)}$ as
			compared to that in $\Lambda^{(i)}\setminus \PPP^{(i)}$.
			
		For each $\nu\in \{1,\dots,E\}$, choose a colored $N$-loop and let $\PPP$ be the set of all such loops.
		By Rules List \ref{RULE4},
			there exists an itinerary for $\mathcal{P}$.
		From this itinerary,
			we may extract an ordered sequence of shrink RBS moves
			and spread events.
		If a spread event occurs from $\Lambda^{(i)}$ to
					$\Lambda^{(i+1)}$ for a loop $L \in \PPP^{(i)}$ we may mark or identify outside edges
					from which the color of $L$ spreads to $\Lambda^{(i+1)}\setminus L$.
		These edges are consistently identified in $\Lambda^{(j)}$ for all $i\leq j$
			by using the associations made when defining RBS moves in Section \ref{SSEC_RBS_DEF}.
				
		Let $m_1,\dots,m_T$ be the twist and shrink moves, in order, extracted from the list $\mathsf{L}$ of moves in our itinerary for $\PPP$.
		We will now construct a new undirected (multi)graph $\Xi$.
		Informally, we will act on $\Lambda$ by the twist and shrink moves $m_1,\dots,m_T$
			(ignoring all other moves)
			to arrive at new a graph $\Lambda^*$.
		As discussed in previous sections, we may without ambiguity refer to the same loops $L^*_1,\dots,L^*_E$
			in $\Lambda^*$.
		From $\Lambda^*$ we will then remove all edges from the loops in $L^*_1,\dots,L^*_E$
			to get a new undirected graph $\overline{\Lambda^*}$.
		Finally, we will create $\Xi$ by 
			identifying all $\mathfrak{s}$-special vertices in $L^*_\nu$ as one new vertex $\nu_{\mathfrak{s}}$
			for each $\nu\in \{1,\dots,E\}$ and $\mathfrak{s}\in \{\ell,r\}$.
			
		More specifically, act on $\Lambda$ by the moves $m_1,\dots,m_T$ on
			loops $L_1,\dots,L_E$ given by the itinerary to arrive at new graph $\Lambda^*$.
		Note that each loop $L_\nu$ must exist in $\Lambda^*$, as twist moves only permute vertices
			while shrink moves may only act on $L_\nu$
			if it has at least $3$ edges.
		Call the resulting loops $L^*_\nu$ in $\Lambda^*$ for each $\nu\in \{1,\dots,E\}$.
		Traversing each loop $L^*_\nu$ in $\Lambda^*$, 
			we visit the same vertices we would visit when traversing $L_\nu$ in $\Lambda$ except for
				those lost due to shrink moves, and the order of 
				visitation may change due to twist moves.
			
		We now create the undirected graph $\overline{\Lambda^*}$.
		The vertices of $\overline{\Lambda^*}$ are the same as those of $\Lambda$ (and $\Lambda^*$).
		For each edge $e$ in $\Lambda^*$ that does not belong to 
			a loop $L_\nu^*$ for some $\nu\in \{1,\dots,E\}$,
			we add an undirected edge  $e'$ that connects in $\overline{\Lambda^*}$
			the initial vertex of $e$ in $\Lambda^*$ to the terminal vertex of $e $ in $\Lambda^*$.
			
		Finally we create $\Xi$ from $\overline{\Lambda^*}$.
		Let $V^-$ be the collection of all vertices in $\Lambda^*$ that 
			are contained in
			a loop $L^*_\nu$ for some $\nu\in \{1,\dots,E\}$ and
			let $$V^+ = \{\nu_\mathfrak{s}:~\mathfrak{s}\in \{\ell,r\},~\nu\in \{1,\dots,E\}\}$$
			be a new set of $2E$ vertices that do not belong to $\Lambda^*$.
		The vertices of $\Xi$ are then all those that belong to $\overline{\Lambda^*}$ except those in $V^-$,
			along with the new vertices in $V^+$.
		Define a function $\psi$ on the vertices of $\overline{\Lambda^*}$ to be 
			$\psi(w) = \nu_\mathfrak{s}$ if $w \in V^-$ is a $\mathfrak{s}$-special
				vertex $L_\nu^*$
			and $\psi(w)=w$ if $w\not\in V^-$.
		For each edge $e$ in $\overline{\Lambda^*}$ (not $\Lambda^*$)
			we add an edge $e'$ to the edges of $\Xi$ where $e'$ is induced by the image of $\psi$.
		See Figure \ref{FIG_MAKING_XI} for an example of how to construct
			$\Xi$ about a $4$-loop.
		
		\begin{figure}[t]
			\begin{center}
				\scalebox{0.8}{\begin{tikzpicture}[vertstyle/.style={draw=black, top color = black!5, bottom color = black!35},
			edgestyle/.style={->,thick,black}, uedgestyle/.style={thick,black}, dotedgestyle/.style={thick,dotted,black},
			gbox/.style={thick,black!20,fill=black!10}]
	\def\asp{0.13}
	
	\def\xa{0}
	\def\ya{0}
	\def\xb{0}
	\def\yb{-26}
	\def\xc{0}
	\def\yc{-52}

	\draw [gbox] (\xa*\asp, \ya*\asp) rectangle +(50*\asp,-24*\asp);
	\draw [gbox] (\xb*\asp, \yb*\asp) rectangle +(50*\asp,-24*\asp);
	\draw [gbox] (\xc*\asp, \yc*\asp) rectangle +(50*\asp,-24*\asp);
		

	\draw [dotedgestyle] (7*\asp + \xa*\asp, -5*\asp + \ya*\asp) -- +(-5*\asp, 2*\asp);
	\draw [dotedgestyle] (7*\asp + \xa*\asp, -5*\asp + \ya*\asp) -- +(-5*\asp, -2*\asp);
	
	\draw [dotedgestyle] (7*\asp + \xa*\asp, -20*\asp + \ya*\asp) -- +(-5*\asp, 2*\asp);
	\draw [dotedgestyle] (7*\asp + \xa*\asp, -20*\asp + \ya*\asp) -- +(-5*\asp, -2*\asp);
	
	\draw [dotedgestyle] (43*\asp + \xa*\asp, -5*\asp + \ya*\asp) -- +(5*\asp, 2*\asp);
	\draw [dotedgestyle] (43*\asp + \xa*\asp, -5*\asp + \ya*\asp) -- +(5*\asp, -2*\asp);
		
	\draw [dotedgestyle] (43*\asp + \xa*\asp, -12.5*\asp + \ya*\asp) -- +(5*\asp, 2*\asp);
	\draw [dotedgestyle] (43*\asp + \xa*\asp, -12.5*\asp + \ya*\asp) -- +(5*\asp, -2*\asp);
		
	\draw [dotedgestyle] (43*\asp + \xa*\asp, -20*\asp + \ya*\asp) -- +(5*\asp, 2*\asp);
	\draw [dotedgestyle] (43*\asp + \xa*\asp, -20*\asp + \ya*\asp) -- +(5*\asp, -2*\asp);
		
	\draw [edgestyle] (9*\asp + \xa*\asp, -5*\asp + \ya*\asp) -- (18.3*\asp + \xa*\asp, -7.75*\asp + \ya*\asp);
	\draw [edgestyle] (18.3*\asp + \xa*\asp, -17.25*\asp + \ya*\asp) -- (9*\asp + \xa*\asp, -20*\asp + \ya*\asp);
	\draw [edgestyle] (41*\asp + \xa*\asp, -5*\asp + \ya*\asp) -- (31.7*\asp + \xa*\asp, -7.75*\asp + \ya*\asp);
	\draw [edgestyle] (41*\asp + \xa*\asp, -12.5*\asp + \ya*\asp) -- (32*\asp + \xa*\asp, -8.75*\asp + \ya*\asp);
	\draw [edgestyle] (31.7*\asp + \xa*\asp, -17.25*\asp + \ya*\asp) -- (41*\asp + \xa*\asp, -20*\asp + \ya*\asp);
	
	\draw [edgestyle] (18.6*\asp + \xa*\asp, -10.15*\asp + \ya*\asp) arc (135:225:3.303*\asp);
	\draw [edgestyle] (21.4*\asp + \xa*\asp, -17.65*\asp + \ya*\asp) arc (225:315:5.071*\asp);
	\draw [edgestyle] (31.4*\asp + \xa*\asp, -14.85*\asp + \ya*\asp) arc (-45:45:3.303*\asp);
	\draw [edgestyle] (28.6*\asp + \xa*\asp, -7.35*\asp + \ya*\asp) arc (45:135:5.071*\asp);
	
	\draw [black!75] (25*\asp + \xa*\asp, -12.5*\asp + \ya*\asp) node {\scalebox{1.5}{$L_\nu$}};

	\shade [vertstyle] (7*\asp + \xa*\asp, -5*\asp + \ya*\asp) circle (2*\asp) node {$x_1$};
	\shade [vertstyle] (7*\asp + \xa*\asp, -20*\asp + \ya*\asp) circle (2*\asp) node {$x_2$};
	
	\shade [vertstyle] (20*\asp + \xa*\asp, -8.75*\asp + \ya*\asp) circle (2*\asp) node {$u_1$};
	\shade [vertstyle] (20*\asp + \xa*\asp, -16.25*\asp + \ya*\asp) circle (2*\asp) node {$v_1$};

	\shade [vertstyle] (30*\asp + \xa*\asp, -8.75*\asp + \ya*\asp) circle (2*\asp) node {$u_2$};
	\shade [vertstyle] (30*\asp + \xa*\asp, -16.25*\asp + \ya*\asp) circle (2*\asp) node {$v_2$};

	\shade [vertstyle] (43*\asp + \xa*\asp, -5*\asp + \ya*\asp) circle (2*\asp) node {$x_5$};
	\shade [vertstyle] (43*\asp + \xa*\asp, -12.5*\asp + \ya*\asp) circle (2*\asp) node {$x_4$};
	\shade [vertstyle] (43*\asp + \xa*\asp, -20*\asp + \ya*\asp) circle (2*\asp) node {$x_3$};



	\draw [dotedgestyle] (7*\asp + \xb*\asp, -5*\asp + \yb*\asp) -- +(-5*\asp, 2*\asp);
	\draw [dotedgestyle] (7*\asp + \xb*\asp, -5*\asp + \yb*\asp) -- +(-5*\asp, -2*\asp);
	
	\draw [dotedgestyle] (7*\asp + \xb*\asp, -20*\asp + \yb*\asp) -- +(-5*\asp, 2*\asp);
	\draw [dotedgestyle] (7*\asp + \xb*\asp, -20*\asp + \yb*\asp) -- +(-5*\asp, -2*\asp);
	
	\draw [dotedgestyle] (43*\asp + \xb*\asp, -5*\asp + \yb*\asp) -- +(5*\asp, 2*\asp);
	\draw [dotedgestyle] (43*\asp + \xb*\asp, -5*\asp + \yb*\asp) -- +(5*\asp, -2*\asp);
		
	\draw [dotedgestyle] (43*\asp + \xb*\asp, -12.5*\asp + \yb*\asp) -- +(5*\asp, 2*\asp);
	\draw [dotedgestyle] (43*\asp + \xb*\asp, -12.5*\asp + \yb*\asp) -- +(5*\asp, -2*\asp);
		
	\draw [dotedgestyle] (43*\asp + \xb*\asp, -20*\asp + \yb*\asp) -- +(5*\asp, 2*\asp);
	\draw [dotedgestyle] (43*\asp + \xb*\asp, -20*\asp + \yb*\asp) -- +(5*\asp, -2*\asp);
		
	\draw [uedgestyle] (7*\asp + \xb*\asp, -5*\asp + \yb*\asp) -- (20*\asp + \xb*\asp, -8.75*\asp + \yb*\asp);
	\draw [uedgestyle] (20*\asp + \xb*\asp, -16.25*\asp + \yb*\asp) -- (7*\asp + \xb*\asp, -20*\asp + \yb*\asp);
	\draw [uedgestyle] (45*\asp + \xb*\asp, -5*\asp + \yb*\asp) -- (30*\asp + \xb*\asp, -8.75*\asp + \yb*\asp);
	\draw [uedgestyle] (45*\asp + \xb*\asp, -12.5*\asp + \yb*\asp) -- (30*\asp + \xb*\asp, -8.75*\asp + \yb*\asp);
	\draw [uedgestyle] (30*\asp + \xb*\asp, -16.25*\asp + \yb*\asp) -- (45*\asp + \xb*\asp, -20*\asp + \yb*\asp);
	
	
	\draw [black!40] (25*\asp + \xb*\asp, -12.5*\asp + \yb*\asp) node {\scalebox{1.5}{$L_\nu$}};	
	
	\shade [vertstyle] (7*\asp + \xb*\asp, -5*\asp + \yb*\asp) circle (2*\asp) node {$x_1$};
	\shade [vertstyle] (7*\asp + \xb*\asp, -20*\asp + \yb*\asp) circle (2*\asp) node {$x_2$};
	
	\shade [vertstyle] (20*\asp + \xb*\asp, -8.75*\asp + \yb*\asp) circle (2*\asp) node {$u_1$};
	\shade [vertstyle] (20*\asp + \xb*\asp, -16.25*\asp + \yb*\asp) circle (2*\asp) node {$v_1$};

	\shade [vertstyle] (30*\asp + \xb*\asp, -8.75*\asp + \yb*\asp) circle (2*\asp) node {$u_2$};
	\shade [vertstyle] (30*\asp + \xb*\asp, -16.25*\asp + \yb*\asp) circle (2*\asp) node {$v_2$};

	\shade [vertstyle] (43*\asp + \xb*\asp, -5*\asp + \yb*\asp) circle (2*\asp) node {$x_5$};
	\shade [vertstyle] (43*\asp + \xb*\asp, -12.5*\asp + \yb*\asp) circle (2*\asp) node {$x_4$};
	\shade [vertstyle] (43*\asp + \xb*\asp, -20*\asp + \yb*\asp) circle (2*\asp) node {$x_3$};



	\draw [dotedgestyle] (7*\asp + \xc*\asp, -5*\asp + \yc*\asp) -- +(-5*\asp, 2*\asp);
	\draw [dotedgestyle] (7*\asp + \xc*\asp, -5*\asp + \yc*\asp) -- +(-5*\asp, -2*\asp);
	
	\draw [dotedgestyle] (7*\asp + \xc*\asp, -20*\asp + \yc*\asp) -- +(-5*\asp, 2*\asp);
	\draw [dotedgestyle] (7*\asp + \xc*\asp, -20*\asp + \yc*\asp) -- +(-5*\asp, -2*\asp);
	
	\draw [dotedgestyle] (43*\asp + \xc*\asp, -5*\asp + \yc*\asp) -- +(5*\asp, 2*\asp);
	\draw [dotedgestyle] (43*\asp + \xc*\asp, -5*\asp + \yc*\asp) -- +(5*\asp, -2*\asp);
		
	\draw [dotedgestyle] (43*\asp + \xc*\asp, -12.5*\asp + \yc*\asp) -- +(5*\asp, 2*\asp);
	\draw [dotedgestyle] (43*\asp + \xc*\asp, -12.5*\asp + \yc*\asp) -- +(5*\asp, -2*\asp);
		
	\draw [dotedgestyle] (43*\asp + \xc*\asp, -20*\asp + \yc*\asp) -- +(5*\asp, 2*\asp);
	\draw [dotedgestyle] (43*\asp + \xc*\asp, -20*\asp + \yc*\asp) -- +(5*\asp, -2*\asp);
		
	\draw [uedgestyle] (7*\asp + \xc*\asp, -5*\asp + \yc*\asp) -- (25*\asp + \xc*\asp, -8.75*\asp + \yc*\asp);
	\draw [uedgestyle] (25*\asp + \xc*\asp, -16.25*\asp + \yc*\asp) -- (7*\asp + \xc*\asp, -20*\asp + \yc*\asp);
	\draw [uedgestyle] (45*\asp + \xc*\asp, -5*\asp + \yc*\asp) -- (25*\asp + \xc*\asp, -8.75*\asp + \yc*\asp);
	\draw [uedgestyle] (45*\asp + \xc*\asp, -12.5*\asp + \yc*\asp) -- (25*\asp + \xc*\asp, -8.75*\asp + \yc*\asp);
	\draw [uedgestyle] (25*\asp + \xc*\asp, -16.25*\asp + \yc*\asp) -- (45*\asp + \xc*\asp, -20*\asp + \yc*\asp);
	
	

	\shade [vertstyle] (7*\asp + \xc*\asp, -5*\asp + \yc*\asp) circle (2*\asp) node {$x_1$};
	\shade [vertstyle] (7*\asp + \xc*\asp, -20*\asp + \yc*\asp) circle (2*\asp) node {$x_2$};
	

	
	\shade [vertstyle] (25*\asp + \xc*\asp, -8.75*\asp + \yc*\asp) circle (2*\asp) node {$\nu_\ell$};
	\shade [vertstyle] (25*\asp + \xc*\asp, -16.25*\asp + \yc*\asp) circle (2*\asp) node {$\nu_r$};

	\shade [vertstyle] (43*\asp + \xc*\asp, -5*\asp + \yc*\asp) circle (2*\asp) node {$x_5$};
	\shade [vertstyle] (43*\asp + \xc*\asp, -12.5*\asp + \yc*\asp) circle (2*\asp) node {$x_4$};
	\shade [vertstyle] (43*\asp + \xc*\asp, -20*\asp + \yc*\asp) circle (2*\asp) node {$x_3$};


	\draw (-3*\asp+ \xa*\asp,-12 *\asp + \ya*\asp) node {\scalebox{1.5}{$\Lambda$}};
	\draw (-3*\asp+ \xb*\asp,-12 *\asp + \yb*\asp) node {\scalebox{1.5}{$\overline{\Lambda^*}$}};
	\draw (-3*\asp+ \xc*\asp,-12 *\asp + \yc*\asp) node {\scalebox{1.5}{$\Xi$}};
	
\end{tikzpicture}}
			\end{center}
		\caption{The local construction of $\overline{\Lambda^*}$ and $\Xi$ from $\Lambda$ around $4$-loop $L_\nu$.
			Here we assume the $x_i$'s do not belong to an $N$-loop.}\label{FIG_MAKING_XI}
		\end{figure}
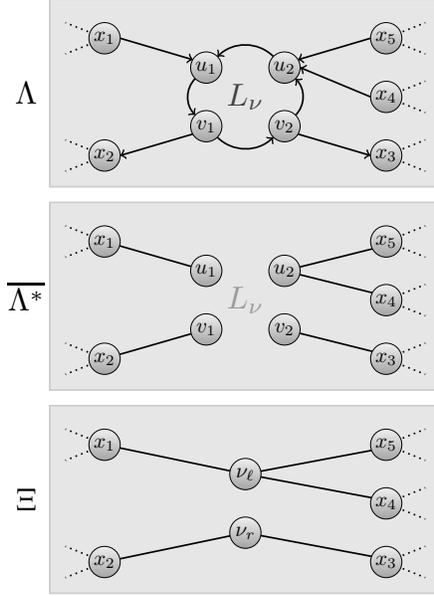
						
		Recall that $\Lambda$ has $K_\ell + K_r$ vertices and $K + K_\ell + K_r$ edges,
			where $K_\mathfrak{s}$ is the number of $\mathfrak{s}$-special vertices in $\Lambda$ for $\mathfrak{s}\in \{\ell,r\}$.
		For each $\nu\in \{1,\dots,E\}$ let $N^*_\nu$ be the number of vertices (equivalently edges) in
		 the loop $L^*_\nu$  in $\Lambda^*$ (after all shrink and twist moves have been applied)
		 	and let $\mathbf{N}^*$ be the sum of all the $N^*_\nu$'s.
			We note that: $\Lambda^*$ has the same number of vertices (resp. edges) as $\Lambda$,
				$\overline{\Lambda^*}$ and $\Xi$ have $\mathbf{N}^*$ fewer edges than $\Lambda$ and
					$\Xi$ had $\mathbf{N}^*$ vertices removed and $2E$ vertices added. Therefore
			\begin{equation}
				\begin{array}{l}
				\#_\mathrm{edge} (\Xi) = K + K_\ell + K_r - \mathbf{N}^*,\mbox{ and}\\
				\#_\mathrm{vert} (\Xi) = K_\ell + K_r + 2E - \mathbf{N}^*,
				\end{array}
			\end{equation}
		and so
			\begin{equation}\label{EQ_CONN_XI}
				\#_{\mathrm{edge}}(\Xi) - \#_{\mathrm{vert}}(\Xi) = K - 2E.
			\end{equation}
		
		In the next proposition, we will show that $\Xi$ must be 
			weakly connected.
		As in the proof of Corollary \ref{COR2LOOPS}, $\#_{\mathrm{edge}}(\Xi) - \#_{\mathrm{vert}}(\Xi) \geq -1$ or
			\begin{equation}
					K - 2E \geq -1 \iff E \leq \frac{K+1}{2}
			\end{equation}
		which is our main result.
		The proof of this proposition is given in the appendix.
		
		\begin{prop}\label{PROP_CONNECT}
			Let $\Lambda$ with coloring $\CCC$ satisfy Rules List \ref{RULE1} and let
				$\PPP$ be a collection of distinctly colored $N$-loops.
			If we may find itineraries for $\Lambda$, $\CCC$ and $\PPP$ according to Rules List \ref{RULE4}
				and construct $\Xi$ as in this section, then
				$\Xi$ must be connected.
		\end{prop}
		
		\begin{coro}\label{cor: main_corollary}
			Any coloring $\CCC$ on $\Lambda$ satisfying the conditions in the previous proposition
				must use at most $\frac{K+1}{2}$ colors.
		\end{coro}

\section{Ergodic theory on subshifts}\label{SEC_ERG}

We now build on the results from \cite{cDamFick2016} and Section \ref{SEC_COLOR_DEF} concerning the upper density
		function $\DDD(w,x)$, which is used to define the coloring function $\CCC$.
In Section \ref{SEC_ERG_TO_GRAPH}, these results along with those from Section \ref{SEC_COLOR_DEF} will be used to justify the Rules Lists from Section \ref{SEC_GRAPHS}
	for graphs $\Lambda$ constructed from subshifts.

Assume $X$ is a transitive subshift on $\AAA$ with language $\LLL$ satisfying RBC with growth constant $K$.
By Corollary \ref{COR_REC_IFF_UREC} and Lemma \ref{LEM_UREC_IFF_MIN} it must be that $X$ is actually minimal.
Recall that $r(w,x,j)$, $w\in \LLL$ and $x\in X$, is $1$ if $w$ occurs in $x$ beginning at a position
	within the $j^{th}$ block of length $(K+1)|w|$ in $x$
	and $0$ otherwise.
If we define
	\begin{equation}
		\SSS_{N}(w,x) = \sum_{j=1}^N r(w,x,j)
			\mbox{ and }
		\DDD_{N}(w,x) = \frac{1}{N} \SSS_{N}(w,x),
	\end{equation}
then $\DDD(w,x) = \limsup_{N\to \infty} \DDD_N(w,x)$.

The next lemma gives a lower bound for densities in the following situation.
If in any $x$, one of $z^{(1)}, \dots, z^{(p)}$ occurs between any two occurrences of $w$,
	then there is a $z^{(j)}$ whose density in $x$ can be bounded below in terms of the density of $w$ in $x$.

\begin{lemm}\label{LemLoopWords}
	Let $w\in \LLL_n$ and suppose
		$$ \BBB = \{z^{(1)},z^{(2)},\dots,z^{(p)}\}\subseteq \LLL_m$$
	satisfies the following:
	for each $y\in \LLL$, $|y|>\max\{n,m\}$, and $j,j'$ with $1\leq j<j'\leq |y|-\max\{m,n\}$ such that
		$$ y_{[j,j+n-1]} = y_{[j',j'+n-1]} = w,$$
	(that is, $w$ occurs at least twice in $y$), then there exist $i,k$ with $1\leq i \leq p$
		and $j\leq k < j'$ so that $y_{[k,k+m-1]} = z^{(i)}$ (that is, $z^{(i)}$ begins between the two occurrences of $w$).
		
	Then for $x\in X$ there exists $j$ with $1\leq j \leq p$ so that
		$$ \DDD(z^{(j)},x) \geq \frac{1}{p (1 + 3n/m)}\DDD(w,x).$$
\end{lemm}

\begin{rema}
	If $m \geq n$ then the coefficient $1/4p$ may be used instead.
\end{rema}

\begin{proof}[Proof of Lemma \ref{LemLoopWords}]
	For any large $M\in \NN$,
		consider
		$ x_{[1,M(K+1)nm]}$
	which is both the first $Mn$ of $x$'s $(K+1)m$-blocks and the first $Mm$ of $x$'s $(K+1)n$-blocks.
	Between any two beginnings of $w$ in different $(K+1)n$-blocks,
		there exists at least one $z^{(i)}$ in between.
	So there must be $z^{(j_M)}$ that occurs at least $\frac{1}{p}(\SSS_{Mm}(w,x)-1)$ times.
	Of these occurrences, we claim that at most $3+\frac{m}{n}$ of them may begin in the same $(K+1)m$-block.

	Indeed, if $m\leq n$ then any $(K+1)m$-block may overlap with at most $2$ $(K+1)n$-blocks.
		If $m = kn$, then each $(K+1)m$-block overlaps exactly $k = \frac{m}{n}$ $Ln$-blocks.
		If $m = kn + k'$ for integer $0<k'<n$, then each $(K+1)m$-block may overlap at most $2 + \frac{m}{n}$ $(K+1)n$-blocks.
		In all cases, the number of beginnings of $z^{(j_M)}$ in an $(K+1)m$-block is at most $1$ plus the number of
			overlapped $(K+1)n$-blocks.
	
	It follows then that
		$$ \SSS_{Mn}(z^{(j_M)},x) \geq \frac{1}{p(3+m/n)}(\SSS_{Mm}(w,x)- 1)$$
	or
		$$
			\begin{array}{rcl}
			\DDD_{Mn}(z^{(j_M)},x) &\geq& \displaystyle \frac{1}{p(3+m/n)} \frac{m}{n}\DDD_{Mm}(w,x) - \frac{1}{Mp(3n+m)}\GAP\\
				& = & \displaystyle \frac{1}{p(1 + 3n/m)} \DDD_{Mm}(w,x) - \frac{1}{Mp(3n+m)}.
			\end{array}
		$$
	Fix an infinite subset $\MMM\subseteq \NN$ so that
		\begin{enumerate}
			\item $ \DDD(w,x) = \lim_{\MMM\ni M \to \infty} \DDD_{Mm}(w,x)$, and
			\item $ j = j_M$ for fixed $j$ and all $M\in \MMM$.
		\end{enumerate}
	By taking the limsup over $\MMM$ we have
		$$
			\begin{array}{rcl}
			\DDD(z^{(j)},x) & \geq & \displaystyle \limsup_{\MMM\ni M \to\infty} \DDD_{Mn}(z^{(j)},x)\GAP\\
				& \geq & \displaystyle \lim_{\MMM\ni M \to\infty} \frac{1}{p(1+ 3n/m)} \DDD_{Mm}(w,x) - \frac{1}{Mp(3n+m)}\GAP\\
				& = & \displaystyle \frac{1}{p(1 + 3n/m)} \DDD(w,x),
			\end{array}
		$$
	as desired.
\end{proof}

The next lemma appeared in \cite{cDamFick2016} and the proof is similar
		to the one for Lemma \ref{LemLoopWords} above.

\begin{lemm}[Lemma 3.7 in \cite{cDamFick2016}]\label{LemSubword}
	If $w'$ is a subword of $w\in \AAA^*$ and $x\in \AAA^\NN$, then
		$$ \DDD(w',x) \geq \frac{|w'|}{2|w|} \DDD(w,x).$$
\end{lemm}

	In the arguments contained in Section \ref{SEC_ERG_TO_GRAPH},
	we will want to relate the density of $w$ to that of 
	one or more of its exit words $z$ of step $q$.
	In some cases, for a sequence $w^{(n)}$
		of words we use in those constructions, the
		related exit words $z^{(n)}$ will satisfy
		$
			|w^{(n)}|/|z^{(n)}| \to 0
		$
		as $n\to\infty$.
	Under this condition, the previous lemma would become trivial.
	The next lemma will be used instead, as it provides
		bounds relating these densities that do not depend on
		the lengths of the $w^{(n)}$'s or $z^{(n)}$'s.

\begin{lemm}\label{LemExitDens}
	For fixed $w\in \LLL_n$, with minimal valid step $q$ for $w$ in $\LLL$, let
		$\XXX=\XXX_q(w)$ denote the set of exit words of $w$ in $\LLL$ with step $q$.
	Then for any $x\in X$ and $z \in \XXX$ 
		\begin{equation}\label{EqPeriodSub}
			\DDD(w,x) \geq \frac{1}{3 K + 9} \DDD(z,x).
		\end{equation}
	Also, there exists $z'\in \XXX$ so that
		\begin{equation}\label{EqExitWordDens}
			\DDD(z',x) \geq \frac{1}{(2K+3)|\XXX|} \DDD(w,x).
		\end{equation}
\end{lemm}

Note that \eqref{EqPeriodSub} is similar to \cite[Lemma 3.10]{cDamFick2016}.
Also, for the subshifts $X$ we are considering the set of exit words $\XXX$ for $w$ in $\LLL$ with minimal step $q$ must
	be finite.
We will assume $|\XXX|<\infty$ in the proof of \eqref{EqExitWordDens},
	as the last inequality is trivial when $\XXX$ is infinite.
The results \cite[Lemma 3.10 \& Corollary 3.11]{cDamFick2016} collectively provided a similar result that was used in that
	paper to treat $2$-loops.
However, RBC was not assumed in that paper, so a uniform (over all word lengths) bound on the number of exit words,
	as provided in Lemma \ref{LemExitCount}, associated to a $2$-loop
	could not be assumed.

\begin{proof}
	We first will show \eqref{EqPeriodSub}.
	Let $pw^{q \ast r}s$ be the representation of $z\in \XXX$ and
	 $m = |z|$,
	 noting that 
	 	$$
	 		m = |p| + |w^{q\ast r}| + |s| \leq n + (r+1)q.
	 	$$
	 	
	We fix a parameter $\alpha = \frac{3K+9}{2}$.
	If $m \leq \alpha n$, then we may apply Lemma \ref{LemSubword} to get
		$$ \DDD(w,x) \geq \frac{1}{2\alpha} \DDD(z,x).$$
	Now suppose $m > \alpha n$ and fix $M \in \NN$.
	If $z=pw^{q \ast r}s$ begins in an $(K+1)m$-block of $x_{[1,M(K+1)mn]}$
		except the last one, then $w$ must begin in at least $\left\lfloor \frac{(r-1)q}{(K+1)n} \right\rfloor + 1$ consecutive $(K+1)n$-blocks.
	Two occurrences of $z$ may overlap in $x$
		but this overlap will not reduce the number of occurrences of $w$ by Corollary \ref{CorExitOverlap}.
	However, it is possible that at most the first appearance of $w$ in $z$
		occurs in the same $(K+1)n$-block as the last appearance of $w$ in the previous $z$.
		
	We therefore should remove at most $1$ occurrence in an $(K+1)n$-block of $w$ from each $z$ in our count,
		as up to two per $z$ may be over-counted by a factor of two.
	Therefore,
		$$
			\SSS_{Mm}(w,x) \geq \left\lfloor \frac{(r-1)q}{(K+1)n}\right\rfloor\left(\SSS_{Mn}(z,x) - 1\right)
		$$
	and so
		$$
			\DDD(w,x) \geq \frac{n}{m}\left\lfloor \frac{(r-1)q}{(K+1)n}\right\rfloor \DDD(z,x).
		$$
	Because $ m < (r-1)q + 3n$ and $m > \alpha n$,
		$$
			\begin{array}{rcl}
				\frac{n}{m}\left\lfloor \frac{(r-1)q}{(K+1)n}\right\rfloor &
					> & \frac{n}{m}\frac{(r-1)q - (K+1)n}{(K+1)n}\\
					\\
					& > & \frac{m - (K+4)n}{Lm} \\
					\\
					& > & \frac{1}{K+1} - \frac{K+4}{(K+1)\alpha},
			\end{array}
		$$
	or
		$$
			\DDD(w,x) \geq \left(\frac{1}{K+1} - \frac{K+4}{(K+1)\alpha}\right) \DDD(z,x).
		$$
	It follows that the global lower bound is given by
		$$
			\DDD(w,x) \geq \min\left\{\frac{1}{2\alpha},\frac{1}{K+1} - \frac{K+4}{(K+1)\alpha}\right\} \DDD(z,x),
		$$
	which, with our value of $\alpha$, gives \eqref{EqPeriodSub}.
	
	We will now prove \eqref{EqExitWordDens}.
	Consider the $i^{th}$ $Ln$-block in $x_{[1,M(K+1)mn]}$ where
		$m$ is the least common multiple of the lengths $|z|$ for each $z\in \XXX$.
	Let $\III=\III_M$ be the integers $i\in[1,Mm]$ such that $w$ begins in the $i^{th}$ $(K+1)n$-block
		of $x_{[1,M(K+1)mn]}$, i.e. $w$ begins in $x$ at a position in  $[(K+1)n(i-1)+1,(K+1)ni]$.
	For each $i\in \III$,
		let $j_i$ be the minimum beginning position of $w$ in $[(K+1)n(i-1)+1,(K+1)ni]$,
			in other words the minimum $j_i$ in this interval so that
				$$ x_{[j_i,j_i+n-1]} = w.$$
	Let $\beta = \left\lceil \frac{m}{(K+1)n} \right\rceil$ and $\III' = \III \cap [\beta, Mm - \beta]$.
	If $i\in \III'$, there exists a
		unique position $k_i$ and unique $z_i\in \XXX$ such that
			$$
				z_i = x_{[k_i,k_i+|z_i|-1]}
				\mbox{ and }
				k_i \leq j_i < k_i + |z_i| - n,
			$$
	by Lemma \ref{LemWordInExit}.
	Then there exists $z'\in \XXX$ such that
		$$
			\frac{|\{i\in \III': z_i = z'\}|}{|\III'|}
				\geq
			\frac{1}{|\XXX|},
		$$
	and let $\KKK = \{k_i: i\in \III', z_i = z'\}$
		be the positions of $z'$ in our full block of $x$.
	For each $k\in \KKK$ there may be at most $\gamma$ elements $i\in \III'$ such that
		$k_i= k$, where
		$$
			\gamma = \left\lfloor \frac{(r-1)q -1}{(K+1)n}\right\rfloor + 2 \leq \frac{|z'|}{(K+1)n} + 2
		$$
	given decomposition $z' = p w^{q\ast r} s$.
	To justify this, note that Corollary \ref{CorExitOverlap} ensures that
		each beginning of $w$ can be a subword of at most one occurrence of $z'$ and must occur within $w^{q\ast r}$
		in the decomposition of $z'$.
		At worst: $z'$ may be positioned so that the first $w$ in $w^{q\ast r}$ begins at the end of a $(K+1)n$-block
			(requiring an overlap of only the beginning position in this block) and
			the last $w$ in $w^{q\ast r}$ begins at the beginning of an $(K+1)n$-block (requiring the entire word $w$ to
				overlap in this block).
		The remaining portion of $w^{q\ast r}$ that fully covers $(K+1)n$-blocks is length at most
			$$
				|w^{q\ast r}| - n -1 =  n + (r-1)q - n -1 = (r-1) q - 1,
			$$
		and may therefore only contribute beginnings in at most $\frac{(r-1)q - 1}{(K+1)n}$ consecutive $(K+1)n$-blocks.
	
	We claim that at most $K+1$ occurrences of $z'$ may begin within one $(K+1)|z'|$-block.
		Suppose there are $B$ beginning positions of $z'$ in a $(K+1)|z'|$-block and $i_1$ is the first beginning
				position and $i_B$ is the last.
		Because $i_B$ and $i_1$ occur in the same $(K+1)|z'|$-block and by Corollary \ref{CorExitOverlap},
				$$
					(K+1)|z'| - 1 \geq i_B - i_1 \geq (B-1)(|z'| + n),
				$$				
		which implies that $B < K+1$.

	If $m' = \frac{m}{|z'|}$, then we conclude that
		$$
			\SSS_{Mm'n}(z',x)
				\geq
				\frac{1}{(K+1) \gamma | \XXX|}
			(\SSS_{Mm}(w,x) - 2\beta).
		$$
	If we divide each side by $Mm'n$ to get $\DDD(z',x)$ in the $\limsup$ on the left hand side,
		the right hand side has leading coefficient
		$$
			\frac{1}{Mm'n}\cdot \frac{1}{(K+1) \gamma |\XXX|} = \frac{|z'|}{(K+1)\gamma|\XXX|n}\frac{1}{Mm}
		$$
	and so
		$$
			\DDD(z',x) \geq \frac{|z'|}{(K+1)\gamma|\XXX|n} \DDD(w,x).
		$$
	We observe that
		$$
			\begin{array}{rcl}
				\frac{|z'|}{(K+1)\gamma|\XXX|n} & \geq & \frac{|z'|(K+1)}{(K+1)|\XXX|(|z'| + 2(K+1)n)}\\
					\\
					 & \geq & \frac{|z'|}{|\XXX|\cdot |z'|(2(K+1) + 1)}\\
					 \\
					 & = & \frac{1}{(2K+3)|\XXX|}
			\end{array}
		$$
	concluding the proof.
\end{proof}

\section{From ergodic theory to graphs}\label{SEC_ERG_TO_GRAPH}

	Fix a transitive subshift $X$ whose aperiodic language $\LLL$ satisfies RBC.
	By Lemma \ref{LEM_RBC_then_ECG}, $\LLL$ has eventually constant growth as well.
	Let $K$ be the growth constant and $N_0$ be such that
		$p(n+1) - p(n) = K$ for all $n\geq N_0$
		and each bispecial $w\in \LLL_n$ is regular for all $n\geq N_0$.

	Recall that for $\nu\in \EEE(X)$ we fix a generic $x^{(\nu)} \in X$ for $\nu$,
		meaning \eqref{EQ_ERGODIC_THM} holds for all $w\in \AAA^*$.
	By Corollary \ref{COR_REC_IFF_UREC}, $X$ is in fact minimal and so each fixed $x^{(\nu)}$
		is not eventually periodic (as defined before Lemma \ref{LemWordInExit}).
	Following the discussion in Section \ref{SSEC_INTRO_SUBSHIFTS},
		we may find an infinite $\WWW \subseteq \NN$ and base graph
			$\Lambda$ so that, up to fixing a vertex naming,
	$\Lambda = \Gamma_n^{\mathrm{sp}}$ for each $n\in \WWW$.
	For $w\in \Lambda$, let $w_n$ denote the corresponding vertex in $\Gamma_n^{\mathrm{sp}}$, $n\in \WWW$.
	We may refine $\WWW$ further so that:
		\begin{enumerate}
			\item For each $w\in \Lambda$ there exists $\mu_w \in \MMM(X)$ such that
				 $\{ w_n\}_{n\in \WWW} \to \mu_w$ as described before \eqref{EQ_WORDS_TO_MEASURE}.
			\item For each $w\in \Lambda$,
				$$
					\DDD(\mu_w, \nu) = \lim_{\WWW \ni n \to \infty} \DDD(w^{(n)}, x^{(\nu)}),
				$$
				which differs from \eqref{EQ_D(nu)_DEF} as we have $\lim$ rather than $\limsup$ on the
					right-hand side.
		\end{enumerate}
		
	In the following subsections, we will verify the rules lists from Section~\ref{SEC_GRAPHS} for our graph $\Lambda$ defined above.
	
\subsection{Building \texorpdfstring{$\Lambda$}{Lambda} and \texorpdfstring{$\CCC$}{C}: Notation \ref{NOTA1}}

	Items (1) -- (5) in Notation \ref{NOTA1}
		follow naturally from the counting arguments presented in Section \ref{SEC_DEF}.
	Recall the coloring function $\CCC$ from Definition \ref{DEF_COLORING_1}
		which assigns to each $w\in \Lambda$ a value in $\EEE(X) \cup \{\mathbf{0}\}$
		as follows: $\CCC(w) = \nu \in \EEE(X)$ if and only if $\DDD(\mu_w,\nu) >0$
			and $\CCC(w) = \mathbf{0}$ if and only if $\DDD(\mu_w,\nu) =0$ for all $\nu$.
	Item (6) is satisfied, as $\mathcal{C}$ is well-defined once we identify $\EEE(X)$
		with $\{1,\dots,|\EEE(X)|\}$ and $\mathbf{0}$ with the number $0$.
	We want to then show that items (7) and (8) are satisfied.
	To do so, we must extend our coloring rule $\CCC$ to edges and show
		consistency between the coloring of an edge and its endpoint vertices.
	
	First, for each edge $e$ in $\Lambda$ and $n\in \WWW$, there exists
		a branchless path $P$ in $\Gamma_n$, which we consider to be a word $P\in \LLL$ of
		some length $N \geq n$.
	For each $1 \leq j \leq N+1-n$, let $w_j = P_{[j,j+n-1]}$
		be the $j^{th}$ vertex in $\Gamma_n$ visited by $P$.
	Let $u$ be the starting vertex and $v$ be the terminal vertex of $e$ in $\Lambda$
		with $u^{(n)}$ and $v^{(n)}$ the corresponding vertices in $\Gamma_n$ for this fixed $n$.
	It follows that $w_1 = u^{(n)}$, $w_{N-n+1} = v^{(n)}$,
	 	$w_j$ is neither left nor right special if $1< j < N-n+1$
	 	and $w_j \neq w_{j'}$ for any $j\neq j'$.
	We have that $N = n$ if and only if $u^{(n)} = v^{(n)}$, and in such cases $e$ is a bispecial
		edge in $\Lambda$.
		
	We say that $w_j$ is \emph{representative} of $e$ in $\Gamma_n$ as long as the following
		two conditions are satisfied:
		\begin{itemize}
			\item if $u$ is right special, then $j >1$ and
			\item if $v$ is left special, then $j< N-n+1$.
		\end{itemize}
	In other words, $w_j$ is representative of $e$ in $\Gamma_n$ if
		any path visiting $w_j$ must agree with the path $P$ until first visiting $v^{(n)}$ in the forward direction and
			until first reaching $u^{(n)}$ in the reverse.
	These two restrictions in the definition are necessary as there are multiple branchless paths that start 
		from (resp. end at) a right (resp. left) special vertex in $\Gamma_n$.

	For each $n\in \WWW$ fix $w^{(n)} = w_j$ that is representative of $e$ and
		for $\nu\in \EEE(X)$ we assign $\CCC(e) = \nu$ if and only if
		\begin{equation}\label{EQ_EDGE_DEFPOS}
			\limsup_{\WWW\ni n \to \infty} \DDD(w^{(n)},x^{(\nu)}) > 0,
		\end{equation}
	and $\CCC(e) = \mathbf{0}$ if and only if this limit is $0$ for all $\nu$.
	
	As we will now show, condition \eqref{EQ_EDGE_DEFPOS} will imply that $\CCC(u) = \CCC(v) = \nu$
		as well, satisfying item (8).
	This implies that we may assign at most one color to an edge.
	The final statement of the following lemma shows that our definition of $\CCC(e)$ is 
		independent of the choice of representative $w^{(n)}$'s.
	This will allow for more flexible arguments in the remainder of Section \ref{SEC_ERG_TO_GRAPH}
		as we may freely choose convenient representatives as desired.
	\begin{lemm}
		Let $e$ from $u$ to $v$ be as above with a sequence $w^{(n)}$ of representative words.
			\begin{enumerate}
				\item For a fixed $\nu \in \EEE(X)$, let
					$\alpha = \DDD(u,\nu)$ and $\beta = \DDD(v,\nu)$.
					Then $$\min\{\alpha,\beta\} \geq \frac{1}{4} \limsup_{\WWW\ni n \to \infty} \DDD(w^{(n)},x^{(\nu)}).$$
				\item If $\tilde{w}^{(n)}$ is another sequence of representative words for $e$, let
					$$\gamma = \limsup_{\WWW\ni n \to \infty} \DDD(w^{(n)},x^{(\nu)})
						\mbox{ and }
						\tilde{\gamma} = \limsup_{\WWW\ni n \to \infty} \DDD(\tilde{w}^{(n)},x^{(\nu)}).$$
					Then
					$
						\gamma \geq \frac{1}{4} \tilde{\gamma}
						$ and $
						\tilde{\gamma} \geq \frac{1}{4} \gamma.
					$
			\end{enumerate}
	\end{lemm}
	
	\begin{proof}
		We prove the first item for $\alpha$ and then discuss the remaining parts.
		By the definition of a representative word,
			for each $n\in \WWW$,
			we may apply Lemma \ref{LemLoopWords} to $w^{(n)}$ with $\BBB = \{u^{(n)}\}$
			on the sequence $x^{(\nu)}$
			to get
				$$
					\DDD(u^{(n)},x^{(\nu)}) \geq \frac{1}{4} \DDD(w^{(n)},x^{(\nu)}).
				$$
		The claim then holds by taking the $\limsup$ and recalling that $\WWW$ has been refined so
				that $\alpha$ is obtained by a limit (the inequality for $\beta$ is similar).
		The remaining results all follow from similar applications of Lemma \ref{LemLoopWords}.
	\end{proof}
	
	Similar to our treatment of vertices, we finish by refining $\WWW$ so that
		\begin{equation}\label{EQ_EDGE_MOSTDEFPOS}
			\liminf_{\WWW\ni n \to \infty} \DDD(w^{(n)},x^{(\nu)})>0
			\mbox{ if \eqref{EQ_EDGE_DEFPOS} holds.}
		\end{equation}
	Once we have fixed this refinement, \eqref{EQ_EDGE_MOSTDEFPOS} will
		also hold for any choice of representative sequence $w^{(n)}$.
	We are then able to refine further as required in the next sections
		and do so using any representative sequence without altering the assignment of $\CCC$.
		
\subsection{More on \texorpdfstring{$\CCC$}{C}: Rules List \ref{RULE1}}\label{SSEC_RULES1_PROOF}

	As already discussed, we may rename $\EEE(X)$ so that $\EEE(X)$ is
		identified with $\{1,\dots,E\}$ and $\textbf{0}$ is identified with the number $0$,
	satisfying item (1).
	By Corollary \ref{COR_THE_SPECIAL}, for each $\nu\in \EEE(X)$ and $n\in \WWW$, there exist
		left special $u_n \in \Gamma_n^{\mathrm{sp}}$ and
		right special $v_n \in \Gamma_n^{\mathrm{sp}}$ such that
		$$\DDD(u_n,x^{(\nu)})\geq\frac{1}{K}
		\mbox{ and }
	\DDD(v_n,x^{(\nu)}) \geq \frac{1}{K}.$$
	There exist $u,v\in \Lambda$ so that $u$ corresponds to $u_n$ for infinitely many $n\in \WWW$
		and likewise for $v$ and the $v_n$'s.
	This implies that $\min\{\DDD(\mu_u,\nu), \DDD(\mu_v,\nu)\}\geq \frac{1}{K}$, or $\CCC(u) = \CCC(v) = \nu$,
		so item (2) holds.
		
	We now want to show that item (3) holds.
	As discussed in that item, we are left showing
		that for a vertex $v$ such that $\CCC(v) = \nu \neq \mathbf{0}$
		there exists an edge $e$ ending at $v$ and and edge $e'$ beginning at $v$
		such that $\CCC(e) = \CCC(e') = \nu$.
		
	Let $e_1,\dots,e_L$ be the edges that end at $v$, where $1 \leq L \leq K$ is the in-degree of $v$.
	For each $n\in \WWW$ and $1\leq k \leq L$, fix $w_k^{(n)}$ representative of $e_k$ in $\Gamma_n$.
	If $v^{(n)}$ corresponds to $v$ in $\Gamma_n^{\mathrm{sp}}$, then
			by applying Lemma \ref{LemLoopWords} to $v^{(n)}$
				with $\BBB = \{w^{(n)}_{1},\dots, w^{(n)}_{L}\}$
			there exists $k_n$ so that 
			$$
				\DDD(w^{(n)}_{k_n},x^{(\nu)}) \geq \frac{1}{4K} \DDD(v^{(n)},x^{(\nu)})
			$$	
	for each $n\in \WWW$.
	There exists at least one $k$ with $1\leq k \leq L$ so that $k = k_n$ for infinitely many $n$ in $\WWW$.
	For any such $k$, it must be that $\CCC(e_k) = \nu$ as \eqref{EQ_EDGE_DEFPOS} must hold
		for that edge.
	Similarly, we may find at least one $e'$ that begins at $v$ such that $\CCC(e') = \nu$.
	
	It remains to show that item (4) holds.
	Suppose for a contradiction that for an edge $e$ such that $\CCC(e) = \nu \neq 0$
		and for every circuit containing $e$ there exists an edge $\tilde{e}$ in that circuit
			such that $\CCC(\tilde{e}) \neq \nu$.
	Let $w^{(n)}$ be a representative sequence for $e$.
	It must be that \eqref{EQ_EDGE_DEFPOS} fails for any representative
		sequence $\tilde{w}^{(n)}$ for $\tilde{e}$ as above.
	Let $\mathfrak{E}$ be the collection of all such $\tilde{e}$'s.
	For fixed $n\in \WWW$ let $\BBB^{(n)} = \{\tilde{w}_1^{(n)},\dots, \tilde{w}_p^{(n)}\}$
		be the set of all chosen representative words for the edges $\tilde{e}$ in $\mathfrak{E}$.
	Because there are at most $3K$ edges in $\Lambda$, it must be that $p \leq 3K$.
	By applying Lemma \ref{LemLoopWords} to $w^{(n)}$ and this set $\BBB^{(n)}$ we have that
		$$
			\DDD(\tilde{w}^{(n)}_{k_n},x^{(\nu)}) \geq \frac{1}{12K} \DDD(w^{(n)},x^{(\nu)})
		$$
		for some $k_n$ with $1\leq k_n\leq p$.
	As in our argument for item (3) above, this implies that for at least one
		edge $\tilde{e}$ represented by the words in the $\BBB^{(n)}$ sets
		we have $\CCC(\tilde{e}) = \nu$. This is a contradiction.

\subsection{RBS Moves and Rules List \ref{RULE2}}\label{SSEC_RBS_EXPLAINED}
We first justify the definition of regular bispecial moves as described
	in Section \ref{SSEC_RBS_DEF}.
To begin, we will examine how for $n\in \WWW$ the graphs $\Gamma_n^{\mathrm{sp}}$
	and $\Gamma_{n+1}^{\mathrm{sp}}$ are related by natural identifications.
	
Assume that $n$ is large enough so that the conclusions of Lemma \ref{LEM_RBC_UNIQUENESS} hold.
In particular, if $w\in \LLL_n$ is $\mathfrak{s}$-special, its unique $\mathfrak{s}'$-extension $w'\in \LLL_{n+1}$,
	$\mathfrak{s}'\neq \mathfrak{s}$, that is $\mathfrak{s}$-special satisfies
	$Ex^\mathfrak{s}(w) = Ex^\mathfrak{s}(w')$,
	or in other words $w$ and $w'$ may each be $\mathfrak{s}$-extended by the same letters.
Because our graphs $\Lambda$ are created using an infinite sequence $n\in \WWW$
	this assumption is applicable.
	
If $v^{(n)}$ is a vertex in $\Gamma_n^{\mathrm{sp}}$ associated with the word $w\in \LLL_n$,
	identify $w$ with $w'\in \LLL_{n+1}$
	as defined above.
We then name $v^{(n+1)}$ as the vertex in $\Gamma_{n+1}^{\mathrm{sp}}$ associated with $w'$.
This naming is consistent even when $w$ is bispecial as there will be one right special word $w'\in \LLL_{n+1}$
	and one left special word $\tilde{w}'\in \LLL_{n+1}$ defined this way from $w$.
Therefore, for fixed vertex $v\in \Lambda$ and corresponding $v^{(n)}$ in $\Gamma_{n}^{\mathrm{sp}}$
	we have a natural associated vertex $v^{(n +1)}$ in $\Gamma_{n+1}^{\mathrm{sp}}$.
	
Now consider an edge $e$ from $u$ to $v$ in $\Lambda$ as represented in $\Gamma_n$ by a path word $P$.
We want to associate to $e$ an appropriate path word $P'$ in $\Gamma_{n+1}$ consistent with
	the vertex transformations $u^{(n)}$ to $u^{(n+1)}$ and $v^{(n)}$ to $v^{(n+1)}$.
In all cases, $P'$ will contain $P$ as a subword and will possibly be extended on each side.

We define $P'$ by cases on the special words $w$ and $z$
	of length $n$ that respectively begin and end $P$.
If neither $w$ nor $z$ is bispecial then we may define $P'$
	to be the minimal word that:
	\begin{enumerate}
	\item[(a)] contains $P$ as a subword,
	\item[(b)] begins with the length $n+1$ extension $w'$ associated to $w$ as described above, and
	\item[(c)] ends with the length $n+1$ extension $z'$ associated to $z$.
	\end{enumerate}
In these cases $P'$ is either equal to $P$ or is obtained via extending $P$
	by at most one letter on each side.
	
If $w\neq z$ (so $|P| >n$) but either $w$ or $z$ is bispecial (or both) we change our definition $P'$ as follows:
	\begin{itemize}
		\item[(b$'$)] If $w$ is bispecial we consider the length $n+1$ prefix $wc$ of $P$.
			If $wc = w'$ is left-special (this is then the unique such right-extension of $w$)
				then we begin $P'$ with $P$.
				
			If $wc$ is not left-special then $wc$ ends a unique branchless path $Q$ in $\Gamma_{n+1}$
				that begins at a special vertex.
			In this case we extend $P'$ to start with $Q$ before $P$.
			
			If $w$ is not bispecial we use (b) above to decide how to extend $P$ on the left.
			
		\item[(c$'$)] Similarly, if $z$ is bispecial we let $dz$ be the length $n+1$ suffix of $P$.
			If $dz$ is right-special we end $P'$ with $P$.
			
			If $dz$ is not right-special we let $Q'$ be the unique branchless path from $dz$ to
				a special vertex in $\Gamma_{n+1}$.
			$P'$ is then extended after $P$ to include $Q'$.
			
			If $z$ is not bispecial we use (c) above to decide how to extend $P$ on the right.
	\end{itemize}
	
Finally, if $w = z$ then $w=z$ is bispecial and we have the unique right-special extension $aw$ and
	unique left special extension $wb$.
In this case $P' = awb$.

In all cases, we then name as $e$ the edge in $\Gamma_{n+1}^{\mathrm{sp}}$ associated this path $P'$.

Using these identifications, we first note that the only changes between $\Gamma_n^{\mathrm{sp}}$
	and $\Gamma_{n+1}^{\mathrm{sp}}$ occur when an edge $e_0$ is bispecial
		and its path word $P$ in $\Gamma_n$ is a bispecial word of length $n$.
It is important to note that a bispecial edge must exist in $\Gamma_n^{\mathrm{sp}}$ (the graph is strongly connected)
	and
	cannot share an endpoint with another bispecial edge in $\Gamma_n^{\mathrm{sp}}$.
Therefore each change
	from $\Gamma_n^{\mathrm{sp}}$ to $\Gamma_{n+1}^{\mathrm{sp}}$
	is ``local,'' meaning it only depends on the edges adjacent to the bispecial edge.
Note the paths for each bispecial edge decrease in length from $\Gamma_n$ to 
	$\Gamma_{n+1}$ as its new path word $P'$ is actually equal to $P$.
We may then choose $\tilde{n}$ to be the minimal integer of size at least $n$
	so that a bispecial word exists in $\LLL_{\tilde{n}}$.
Then let $n' = \tilde{n}+1$ and, because $|P| \leq (K+1)n + C$, we have
	\begin{equation}\label{EQ_BOUND_ON_Nprime}
		n' \leq Kn + C
	\end{equation}
	where $K$
	and $C$ are from the defintion of ECG.
The graph $\Gamma_{n'}^{\mathrm{sp}}$ will undergo at least one \emph{Regular Bispecial Move}
	from $\Gamma_{\tilde{n}}^{\mathrm{sp}}$ (which by induction must be equal to $\Gamma_n^{\mathrm{sp}}$),
	meaning at least one bispecial edge  will be replaced with an edge that starts at
	a right special vertex and ends at a left special vertex as in the last case of the definitions above.

Before continuing, we refine $\WWW$ so that all $n$ are large enough to satisfy the assumptions above (the conclusions of Lemma~\ref{LEM_RBC_UNIQUENESS} hold).
Furthermore, we assume that the map $n \mapsto n'$ is a bijection by removing all but one $n$ value from $\WWW$ that maps to $n'$
	(for each such $n'$).
Because of how we have created $\Lambda$ and $\CCC$, these definitions will not change under these or any further
	refinements to the infinite set $\WWW$.

Returning to the graph $\Lambda$ with coloring $\CCC$ defined using $\WWW$,
	we now address how to construct $\Lambda'$ and $\CCC'$.
We let $\widetilde{\WWW}'$ be the collection of all $n'$ values derived from all $n$ values in $\WWW$.
We restrict $\widetilde{\WWW}'$ (and $\WWW$  by the natural bijection) so that there is one graph $\Lambda'$
	such that $\Gamma_{n'}^{\mathrm{sp}} = \Lambda'$ for all $\widetilde{\WWW}'$.
Finally,
	refine to $\WWW'$ to get coloring $\CCC'$ on $\Lambda'$ satisfying Notation \ref{NOTA1} and Rules List \ref{RULE1}.

We now direct our attention to Rules List \ref{RULE2}.
If we show that item (2) holds, then item (1) follows from Rules List \ref{RULE1} as it applies to $\CCC$ and $\CCC'$.
For fixed $n\in \WWW$, let $e$ from $y$ to $z$ be an edge in $\Lambda$ that does not undergo an RBS move from $\Lambda$ to $\Lambda'$.
Let $P$ from $y^{(n)}$ to $z^{(n)}$ be the path word identified with $e$ in $\Gamma_n$
	and let $P'$ from $y^{(n')}$ to $z^{(n')}$ be the path word for $e$ in $\Gamma_{n'}$.
	
We will argue in full detail assuming that $y$ and $z$ are both left  special and leave
	the other cases to the reader.
Here $|P'|> |P|$.
Fix any length $n$ subword $w \neq z^{(n)}$ of $P$ and any length $n'$ subword of $w' \neq z^{(n')}$ of $P'$.
Then $w$ is representative of $e$ (from $\Lambda$) in $\Gamma_n$ and
	$w'$ is representative of $e$ (from $\Lambda'$) in $\Gamma_{n'}$.
If $i$ is the position of $w$ in $P$ and $i'$ is the position of $w'$ in $P'$,
	then for any (long) word $W\in \LLL$,
		any $k$ such that $1\leq k \leq |W| - n + 1$ and $1 \leq k - i + i' \leq |W| - n'$
		we have that $w$ occurs at position $k$ iff $w'$ occurs at position $k + i' - i$.
In particular we may apply Lemma \ref{LemLoopWords} twice to conclude that
	$$
		\DDD(w',x^{(\nu)}) \geq \frac{1}{4} \DDD(w,x^{(\nu)})
		\mbox{ and }
		\DDD(w,x^{(\nu)}) \geq \frac{1}{3 K + 4} \DDD(w',x^{(\nu)})
	$$
as $n+1 \leq n' \leq (K+1)n$ by \eqref{EQ_BOUND_ON_Nprime} if we assume $n \geq C$.
Because both ratios are positive and independent of $n$,
	we conclude that the definitions of $\CCC(e)$ and $\CCC'(e)$ must agree.
	
It may be possible that more than one RBS move occurs at the same time.
This would be the case if there are at least two distinct edges in $\Lambda$ that
	each are identified with a (bispecial) path word of length $n'-1$ in $\Gamma_{n'-1}$
	for infinitely many $n'\in \WWW'$.
Because the transformations described above are all local, we
	may instead view these moves as occurring one at a time
	and the resulting graph does not depend on the order of occurrence.
(See Lemma 5.2 from \cite{cDamFick2016} for a more technical treatment of this issue.)

\subsection{Building itineraries: Rules Lists \ref{RULE3} and \ref{RULE4}}\label{SSEC_ITIN_PROOF}
	We now show how to construct an itinerary as introduced in Definition \ref{DEF_ITIN}.

For subshift $X$, fix a generic $x^{(\nu)}\in X$, $\nu\in \EEE(X)$, and constructed $\WWW$, $\Lambda$ and $\CCC$
	satisfying Notation \ref{NOTA1} and Rules List \ref{RULE1}.
Let $\PPP$ be a selected set of $N$-loops in $\Lambda$ such that each $L\in \PPP$ is colored by an element of $\EEE(X)$
	and no distinct $L,\tilde{L}$ in $\PPP$ have the same color.
	
We describe in detail how to create an itinerary for the case $\PPP = \{L\}$, $L$ colored by $\nu$, and then briefly explain how to extend to larger lists.
Suppose first that another $N$-loop $\tilde{L}$ is colored by $\nu$ and contains at least one of the vertices in $L$,
	so in particular a $\nu$-colored path exists in $\Lambda$ that enters and leaves $L$.
In this case we set $\Lambda^{(0)} = \Lambda^{(1)} = \Lambda$, $\CCC^{(0)} = \CCC^{(1)} = \CCC$
	and $\PPP^{(1)} = \emptyset$ as there is already a ``spread event'' as discussed in Section \ref{SEC_NLOOPS_ITIN}.

So for the remainder we assume that no such $\tilde{L}$ exists.
Fix a vertex $w\in \Lambda$ in $L$ with corresponding $w^{(n)}\in \Gamma^{\mathrm{sp}}_n$, $n\in \WWW$,
	and associated path words $L^{(n)}$ representing $L$.
In other words, $L^{(n)}$ is a path in $\Gamma_n$ that starts at $w^{(n)}$,
	follows the edges defined by $L$ in order and ends at $w^{(n)}$.
So $w^{(n)}$ is visited exactly twice by $L^{(n)}$ and each other vertex in $\Gamma_n$
	is visited at most once.
	
Before continuing, we make the following claim.
This will allow us to exclude some problematic cases before
	addressing the general case.
\begin{lemm}\label{LEM_EDGE_PATH_EXTENSION}
For the fixed choices in this section and any $d\geq 0$,
		there exists a path word $P\in \LLL$ that begins with $w^{(n)}$,
		traverses at least $d$ edges in $\Gamma_n^{\mathrm{sp}}$ (possibly with repetition)
		and
			$$
				\DDD(P, x^{(\nu)}) \geq \left(\frac{1}{16 (K+1)}\right)^{d} \DDD(w^{(n)},x^{(\nu)}).
			$$
\end{lemm}

\begin{proof}
If $d =0$, then the proof is concluded using $P = w^{(n)}$.
	Assume we have constructed such a path $P'$ that satisfies the claim for a given $d\geq 0$;
		we will produce a path $P$ for $d+1$.
		
	Let
		$
			\BBB = \{P'a \in \LLL:~a\in \AAA\}
		$
	be the set of right extensions of $P'$ in the language.
	By \eqref{EQ_GROWTH_AND_EXTENSIONS}, $\BBB$ must contain at most $K+1$ elements.
	By Lemma \ref{LemLoopWords}, there exists $a$ such that
		$$
			\DDD(P'a,x^{(\nu)}) \geq \frac{1}{4 (K+1)} \DDD(P',x^{(\nu)}).
		$$
	We then extend $P'a$ further: there exists a unique right extension of $P'a$ of minimal length, which we call $P$,
		that has an $\mathfrak{s}$-special word as its suffix $s$
		of length $n$ for some $\mathfrak{s}\in \{\ell,r\}$.
	Because this extension is unique, we may again apply Lemma \ref{LemLoopWords} to get
		$$
			\DDD(P, x^{(\nu)}) \geq \frac{1}{4} \DDD(P'a,x^{(\nu)}).
		$$
	Furthermore, $P$ is the extension of path $P'$ in $\Gamma_n^{\mathrm{sp}}$
		either by two edges (if $s$ is bispecial) or by one edge (otherwise).
	Therefore the claim holds for $P$ as desired.
\end{proof}

First, we argue that for infinitely many $n\in \WWW$ the loop word $L^{(n)}$ both belongs to the language and satisfies
	$$
	 \liminf_{\WWW\ni n \to \infty} \DDD(L^{(n)}, x^{(\nu)}) >0.
	$$
If either of these conditions fail, then by Lemmas \ref{LemSubword} and \ref{LEM_EDGE_PATH_EXTENSION}
	there exists a path word $P$ in $\Lambda$ of length at least $d$ (the number of edges in $L$) that contains $w^{(n)}$
	with each edge colored by $\nu$.
It follows that there is at least one edge that starts at a vertex visited by $L$ but is not itself
	in $L$ (meaning an ``outside edge'' as discussed previously)
	that is colored by $\nu$.
This means that a spread event has occurred already in $\Lambda$ which has previously been addressed (above Lemma~\ref{LEM_EDGE_PATH_EXTENSION}).

Let $N = |L^{(n)}|$ and $q = N - n$.
Then for each $r\geq 1$, $[L^{(n)}]^{q\ast r}$ is a path in $\Gamma_n$ that
	travels around loop $L$ exactly $r+1$ times.
We would like to use our arguments now for exit words for $L^{(n)}$ with step $q$
	as discussed in Section \ref{SEC_EXIT_WORDS}.
However, to do so we would need $q$ to satisfy $N \geq 2 q$ or equivalently that $N \leq 2n$.
To avoid this issue, we will use the word $P^{(n)} = [L^{(n)}]^{q\ast 2}$
	instead.
The length $q$ will remain a valid minimal step for $P^{(n)}$ and $|P^{(n)}| = N+q \geq 2q$.
Before continuing with the general construction,
	we note that by a similar argument to the one in the previous paragraph we may assume that
	(after ignoring finitely many $n\in \WWW$)
	$P^{(n)} = [L^{(n)}]^{q\ast 2}$ is in the language and the $\liminf$ of the
	densities of these words in $x^{(\nu)}$ are positive.
Otherwise, a spread event will occur at $\Lambda$.

If we consider the vertex $L^{(n)}$ in $\Gamma_N$, then $[L^{(n)}]^{q\ast r}$
	now travels around $r$ times on the loop closely related to $L^{(n)}$ from $\Gamma_n$.
Because $N \leq (K+2)n$ for large enough $n\in \WWW$
	we may consider the densities of the vertices and edges in $\Gamma_N^{\mathrm{sp}}$
	as they will force a lower bound on densities for corresponding edges and vertices in $\Gamma_n^{\mathrm{sp}}$
	by Lemma \ref{LemSubword}.

For the word $P^{(n)}$ we may apply Lemmas \ref{LemExitDens} and \ref{LemExitCount}
	to find exit word $Z = p [P^{(n)}]^{q\ast r_0} s$
	such that
		\begin{equation}\label{EQ_APPLICATION_OF_EXIT_LEMMA}
			\DDD(Z, x^{(\nu)}) \geq \frac{\DDD(P^{(n)},x^{(\nu)})}{2 K^2 (2K+3)} .
		\end{equation}
Because $Z$ is also an exit word for $[P^{(n)}]^{q\ast r}$, $r\leq r_0$, with minimal step $q$, we may apply the first part of
	Lemma \ref{LemExitDens} to get
	\begin{equation}\label{EQ_MANY_LOOPS}
		\DDD([P^{(n)}]^{q\ast r}, x^{(\nu)}) \geq \frac{1}{3K+9} \DDD(Z,x^{(\nu)}),\mbox{ for all } 1\leq r \leq r_0.
	\end{equation}
Let $N' = N + (r_0-1) q$ be the length of $[P^{(n)}]^{q\ast (r_0-1)}$.
For each $n \leq m \leq N'$ we claim the following:
	\begin{itemize}
		\item $\Gamma_m$ has a path word $L^{(m)}$ corresponding to the loop $L$.
			$L^{(m)}$ traverses the closed circuit that $[P^{(n)}]^{q\ast r_0}$ does in $\Gamma_m$,
		\item For each vertex $v^{(m)}$ in $\Gamma_m$ visited by $L^{(m)}$:
			\begin{itemize}
				\item If $m \leq N$, then $v^{(m)}$ is a subword of $P^{(n)}$ and so
					$$
						\DDD(v^{(m)}, x^{(\nu)}) \geq \frac{1}{2(K+1)} \DDD(\PPP^{(n)}, x^{(\nu)})
					$$
				by Lemma \ref{LemSubword}, using that $N \leq (K + 2) n$ as previously discussed.
				
				\item If $N < m \leq N'$, let $r'$ be the unique integer satisfying
							$$N + (r'-1)q < m \leq N + r' q.$$
					In this case, $v^{(m)}$ is a subword of $P^{q \ast r'}$ and so
						$$
						\DDD(v^{(m)}, x^{(\nu)}) \geq \frac{1}{24K^2(K+3)(2K+3)} \DDD(\PPP^{(n)}, x^{(\nu)})
						$$
					by Lemma \ref{LemSubword}, the bounding inequalities \eqref{EQ_APPLICATION_OF_EXIT_LEMMA} and \eqref{EQ_MANY_LOOPS} and 
						using that $N + rq' \leq 2 m$.
			\end{itemize}
	In either case, $v^{(m)}$ is bounded from below by a constant (depending only on $K$) times $\DDD(P^{(n)},x^{(\nu)})$.
		\item If we consider the special graphs $\Gamma^{\mathrm{sp}}_n$, $\Gamma^{\mathrm{sp}}_{n+1}$,\ldots, $\Gamma^{\mathrm{sp}}_{m}$,
			the edges that correspond to subpaths of $L$ in $\Lambda$ can only be altered by RBS moves that either exchange vertices (twist moves)
				or eject vertices (shrink moves).
			Recall that in the previous section we justified that only RBS moves may occur
					from one graph to the next and the discussion of possible RBS moves on loops corresponding
						to $L$ were covered in Sections \ref{SEC_2LOOPS} and \ref{SEC_NLOOPS_MOVES}.
	\end{itemize}
	By the bounding arguments above, we witness a spread event at or before $N'$.
	Indeed, the loop corresponding to $L$
		in $\Gamma^{\mathrm{sp}}_{m}$, $n \leq m \leq N'$, has a density with positive lower bound that only depends on $K$ and the density of
		$P^{(n)}$.
	Furthermore, the edges in $\Gamma^{sp}_{N'}$ that correspond to the $N'$ length prefix and suffix of exit word $Z$ respectively
		also have density with positive lower bound and these edges must be outside edges, meaning edges that begin or end at 
			a vertex in the loop but are not themselves in the loop.
	
	Let $n = m^{(n)}_0 < m^{(n)}_1 < \dots m^{(n)}_{M^{(n)}} = N'$ be the times such that at least one shrink RBS move on $L^{(m^{(n)}_{i})}$
		occurs
		 from $\Gamma^{\mathrm{sp}}_{m^{(n)}_i}$ to $\Gamma^{\mathrm{sp}}_{m^{(n)}_i+1}$ for $1\leq i \leq M^{(n)}-1$.
	We note that $M^{(n)}$ cannot be more than the number of vertices in $L$ minus $2$ and it is possible that $M^{(n)} = 1$, meaning no shrink
			moves occur from $n$ to $N'$.
	
	For each $n\in \WWW$ we obtain the values $m_0^{(n)},\dots,m_{M^{(n)}}^{(n)}$ and may refine $\WWW$ so that:
		\begin{itemize}
		\item $M^{(n)}$ is a constant value $M$, independent of $n$.
		\item there exist $\Lambda^{(1)},\Lambda^{(2)},\dots, \Lambda^{(M)}$ so that for each $i\in \{1,\dots, M\}$ and $n\in \WWW$,
				we have $\Gamma^{\mathrm{sp}}_{m^{(n)}_i} = \Lambda^{(i)}$.
		\item there are colorings $\CCC^{(1)},\dots,\CCC^{(M)}$ so that each pair $\Lambda^{(i)},\CCC^{(i)}$ satisfies Notation
			\ref{NOTA1} and Rules List \ref{RULE1}.
		\end{itemize}
	We let $\Lambda^{(0)} = \Lambda$ and $\CCC^{(0)} = \CCC$.
	By our construction above, the loop $L$ persists, meaning a corresponding version of the loop exists by the $L^{(m^{(n)}_i)}$
		path words described above,
		and is colored by $\nu$ on each of these graphs.
	We also know that a spread event must occur for $L$ by $\Lambda^{(M)}$
		and we may (by removing graphs $\Lambda^{(i)}$ at the end) assume that the spread event occurs from $\Lambda^{(M-1)}$ to $\Lambda^{(M)}$.
	Combinatorially, the transition from $\Lambda^{(i)}$ to $\Lambda^{(i+1)}$ may be defined by choosing any
		$n\in \WWW$ and using the finite list of RBS moves from $\Gamma^{\mathrm{sp}}_{m^{(n)}_i}$ to $\Gamma^{\mathrm{sp}}_{m^{(n)}_{i+1}}$.

	To finish building this itinerary, we let $\PPP^{(0)}$ be our original partition $\PPP = \{L\}$
		and for $i < M$ we let $\PPP^{(i)}$ be the partition containing what remains in $L$ after ejections via shrink moves, meaning the loop corresponding to the $L^{(m_i^{(n)})}$ paths as discussed above.
	Because a spread event occurs on $L$ from $\Lambda^{(M-1)}$ to $\Lambda^{(M)}$, we let $\PPP = \emptyset$.
	
	If we instead assume an initial partition with multiple loops, we may proceed as we did in the simpler case
		with some modifications.
	For each $n\in \WWW$ and loop $L_\nu\in \PPP$ where $\nu\in \EEE(X)$ denotes the color of the loop, we may find the values
			$$n = m_0^{(n,\nu)} < \dots < m_{M^{(n,\nu)}}^{(n,\nu)}$$
	as described above for $L_\nu$.
	We then interleave all values over all loops to get $n = m_0^{(n)}< \dots < m_{M^{(n)}}^{(n)}$.
	
	We then again refine our set $\WWW$ and create the graphs $\Lambda^{(0)},\dots,\Lambda^{(M)}$
		and respective colorings $\CCC^{(0)},\dots,\CCC^{(M)}$ as before.
	When creating our partitions $\PPP^{(0)},\dots, \PPP^{(M)}$ we keep each loop from $\PPP$ (minus ejections due to
		shrink moves)
		until that loop's spread event.
	Once a spread event occurs for a loop, it is removed from the $\PPP^{(i)}$'s from that point on.
	
	The construction presented above builds an itinerary from Definition \ref{DEF_ITIN}.
	To realize Rules List \ref{RULE4},
		we note that we may restrict from a partition $\PPP$ to any subpartition $\PPP'$
		and build from our initial itinerary an itinerary for $\PPP'$ by ignoring intermediary graphs
		that only realize shrink or spread events for loops in $\PPP\setminus \PPP'$.
	Recalling that shrink moves are not possible for $2$-loops and a twist move actually fixes a $2$-loop,
		we see that Rules List \ref{RULE3} is a simplification of Rules List \ref{RULE4}
		when we restrict to the case of only $2$-loops.

\section{Future Work}\label{SEC_FUTURE_WORK}

		The main result in this paper provides a new combinatorial proof of
			the bound $d/2$ for ergodic measures for minimal $d$-IETs.
		However, if $\pi$ is the permutation on $d$ symbols used to define a minimal IET,
			the known upper bound is $g(\pi)$ as shown independently by \cite{cKatokMeasures} and \cite{cVeechMeas},
			where $g(\pi)$ is the genus of the translation surfaces naturally associated to the IET via
			the zippered rectangle construction as defined in \cite{cVeechZip}.
		This upper bound is sharp as shown in \cite{cKeane} for $d=4$, in \cite{cYoccoz} when $\pi$ is the order reversing permutation
			and in \cite{cFickIET} for the remaining cases.
		While $g(\pi) = \lfloor d/2\rfloor$ when $\pi$ belongs to the same Rauzy class as the order-reversing permutation,
			the relationship $g(\pi) < \lfloor d/2\rfloor$ does hold in many classes.
		If the remaining conditions from \cite{cRBCforIET} are applied, and so $X$ is indeed
			the encoding of a minimal IET, can the correct bound of $g(\pi)$ be
			shown using these combinatorial methods?

			In \cite{cDamFick2016}, the bound $K-2$ was shown under the eventually constant growth condition with constant $K\geq 4$.
			Is the RBC assumption necessary to improve the bound to $\frac{K+1}{2}$?
			In this work, RBC was necessary to bound the number of exit words in Lemma \ref{LemExitCount}.

			Generic measures are those such that \eqref{EQ_ERGODIC_THM} holds for some $x\in X$.
			By the pointwise ergodic theorem, every ergodic measure is
				generic.
			However not every generic measure is ergodic.
			In \cite{cCyrKra}, new methods were used to strengthen and extend the results from \cite{cBosh}
				to bound generic measures.
			Our current coloring definition no longer applies to generic measures
				as they are not extremal.
			However, does there exist an alternate coloring compatible
				with generic measures that may be used to get the $\frac{K+1}{2}$
				bound presented here?
			As of the authorship of this work, the correct bound for generic measures
				for a minimal $d$-IET is not known (by any method).

\appendix

		\section{Proof of Proposition \ref{PROP_CONNECT}}
		Here we prove Proposition~\ref{PROP_CONNECT}. We begin with our (limiting) special Rauzy graph $\Lambda$, which is a directed graph satisfying the following conditions (recall Notation~\ref{NOTA1} and the choice of loops from Section~\ref{SEC_NLOOPS_NEWG}):

		\medskip
		\noindent
		{\bf Conditions A.}
		\begin{enumerate}
		\item $\Lambda$ has $K_\ell+K_r$ vertices, $K_r$ of in-degree 1 and out-degree $\geq 2$, and $K_\ell$ with in-degree $\geq 2$ and out-degree 1.
		\item $\Lambda$ is strongly connected; that is, between any two vertices of $\Lambda$, there is a directed path.
		\item $\Lambda$ contains directed circuits $L_1, \ldots, L_E$ which are vertex-disjoint and vertex self-avoiding, each with at least two vertices.
		\end{enumerate}

		Given a graph $\Lambda$ and circuits $L_1, \ldots, L_E$ as specified in Conditions A, define $\overline{\Lambda}$ as the graph obtained from $\Lambda$ by removing all edges in the union of the $L_i$'s. We will use the terms $L_i$ for the same corresponding in $\overline{\Lambda}$, which are actually no longer circuits, but only collections of vertices.
		
		We first explain what happens to the (weakly connected) components of $\overline{\Lambda}$ after applying an RBS move to both $\Lambda$ and $\overline{\Lambda}$. Letting $C_1, \ldots, C_p$ be the components of $\overline{\Lambda}$ (as unordered vertex sets), we note that since $\Lambda$ was connected, each $C_i$ must contain at least one vertex from the union of all the circuits $L_j$ for $j=1, \ldots, E$. Let $X_i$ be the $E$-tuple $(S_1^i, \ldots, S_E^i)$ consisting of entries $S_j^i$, the set of vertices of $L_j$ that are in $C_i$. We will track both the components $C_i$ and the ``tags'' $X_i$ after applying an RBS move. The move will be one of the following types (recall these terms from Section~\ref{SEC_NLOOPS_MOVES}):
		\begin{enumerate}
		\item[A.] a ``twist'' move operating on two vertices of the same circuit $L_j$,
		\item[B.] a ``shrink'' move operating on two vertices of the same circuit $L_j$ containing at least 3 vertices, or
		\item[C.] some other RBS move operating on two vertices in the complement of $L_1 \cup \cdots \cup L_E$.
		\end{enumerate}
		These are the only three possible RBS moves that operate on $\Lambda$ because there are no bispecial edges (edges leaving an out-degree 1 vertex and entering an in-degree 1 vertex) entering or exiting the $L_i$'s. Note that after applying one of these moves to a graph $\Lambda$ satisfying Conditions A to result in a graph $\zeta$, we can identify the circuits $L_i$ in $\Lambda$ to circuits in $\zeta$ in the natural way. Move A will permute the order of two vertices on a circuit, and move B will remove one vertex from a circuit. We also call these associated circuits in $\zeta$ by the names $L_i$.
		
		The next proposition explains what happens to the $C_i$'s and $X_i$'s after an RBS move of the three types above. To describe the result, we use the following notation. If $X_{i_1}$ and $X_{i_2}$ are tags, then we write $X_{i_1} \cup X_{i_2}$ for the $E$-tuple $(S_1^{i_1} \cup S_1^{i_2}, \dots, S_E^{i_1}\cup S_E^{i_2})$. Furthermore, if $x$ is a vertex of a circuit $L_{j_0}$, then $(X_{i_1} \cup X_{i_2}) \setminus \{x\}$ refers to the $E$-tuple whose $j_0$-th entry is $(S_{j_0}^{i_1}\cup S_{j_0}^{i_2}) \setminus \{x\}$.

		\begin{prop}\label{prop: connected_moves}
		Let $\Lambda$ be a graph satisfying Conditions A and define $\overline{\Lambda}$ as above. Define $\zeta$ as the graph obtained from $\Lambda$ by applying an RBS move of type A, B, or C, and let $\overline{\zeta}$ be the graph obtained from $\zeta$ by removing the edges in the union of the $L_i$'s. If $C_1, \ldots, C_p$ and $X_1, \ldots, X_p$ are the components and tags for $\overline{\Lambda}$, write $D_1, \ldots, D_q$ and $Y_1, \ldots, Y_q$ for the components and tags for $\overline{\zeta}$. The RBS moves A, B, or C have the following effects on the $C_i$'s, $X_i$'s, $D_i$'s, and $Y_i$'s.
		\begin{enumerate}
		\item Moves A and C leave components and tags unchanged. That is, $q=p$, and for each $i$ there is a unique $j$ such that  $(C_i,X_i)=(D_j,Y_j)$.
		\item Move B can only merge components. That is, $q \in \{p-1,p\}$, and there are $I_1,I_2,J$ (where $I_1$ can equal $I_2$, and this occurs if and only if $q=p$) with the following properties. For all $i \neq I_1,I_2$, there is a unique $j \neq J$ such that $(C_i,X_i)=(D_j,Y_j)$. Furthermore, $D_J = C_{I_1}\cup C_{I_2}$, and $Y_J = (X_{I_1} \cup X_{I_2}) \setminus \{x\}$, where $x$ is the vertex contained in the circuit $L_{j_0}$ in $\Lambda$ which is removed from $L_{j_0}$ in $\zeta$ by the move.
		\end{enumerate}
		\end{prop}
		\begin{proof}
		A move of type A only permutes two vertices on a circuit in $\Lambda$, so their corresponding components in $\overline{\Lambda}$ also permute. Thus none of the components or tags change from $\overline{\Lambda}$ to $\overline{\zeta}$. Suppose we perform a move of type C in $\Lambda$ on a bispecial edge $e$, which is not in any of the circuits $L_1, \ldots, L_E$, to obtain $\zeta$. Since both endpoints of $e$ are in the same component we can verify that no components or tags change: let $u,v$ be the endpoints of $e$ and, following earlier notation of the paper, let $y_1, \dots, y_L$ (respectively $z_1, \dots, z_R$) be the beginning (respectively ending) vertices of edges that end (respectively begin) at $u$ (respectively $v$). Then in $\overline{\Lambda}$, all the vertices $u,v,y_i,z_i$ are in the same component. Because none of the edges between these vertices are in any of the circuits $L_i$, after performing the RBS move (see Figure~\ref{FIG_RBS}) these vertices are still all in the same component of $\overline{\zeta}$, and the tags for this component do not change.
		The last case is a move of type B. In this case, write $e$ (with endpoints $u$ and $v$) again for the edge on which the move occurs, so that $e$ is on a loop $L_{j_0}$. In $\overline{\Lambda}$, $u$ and $v$ are in tags for components $C_{I_1}$ and $C_{I_2}$ respectively (and these components could be equal). 
		There are two possible shrink moves that operate on $e$. In one, the edge $e$ is replaced by an edge from $v$ to $u$ and $u$ is ejected from $L_{j_0}$ (this move is shown in Figure~\ref{FIG_NLOOPS}(b)). The other is similar, but $v$ is ejected from the circuit. 

		Note that in both cases, one edge and one vertex are removed from $L_{j_0}$. Since this circuit had at least 3 vertices (and 3 edges) in $\Lambda$, in $\zeta$ it is still a circuit. In the first case, $C_{I_1}$ is the component of $u$ in $\overline{\Lambda}$, all vertices from $C_{I_1}$ are placed into $C_{I_2}$, and since $u$ is removed from $L_{j_0}$ and not placed in another circuit, it is removed from the set of tags. Therefore there is a component $Y_J$ of $\overline{\zeta}$ equal to $C_{I_1} \cup C_{I_2}$, and one has $Y_J = (X_{I_1} \cup X_{I_2}) \setminus \{u\}$. In the second case, all vertices of $C_{I_2}$ are placed into $C_{I_1}$, and $v$ is removed from the set of tags. Therefore there is a component $Y_J$ of $\zeta'$ equal to $C_{I_1} \cup C_{I_2}$, and one has $Y_J = (X_{I_1} \cup X_{I_2}) \setminus \{v\}$.
		\end{proof}
		
		As a direct corollary of the previous proposition, the components and tags of any two graphs on the same vertex set with the same directed circuits $L_i$ and the same components and tags behave the same under moves A-C.
		\begin{coro}\label{cor: connected_moves}
		Let $\Lambda, \Gamma$ be graphs satisfying Conditions A on the same vertex set, both containing the (vertex disjoint, vertex self-avoiding) directed circuits $L_1, \ldots, L_E$. Define $\overline{\Lambda}$ and $\overline{\Gamma}$ as above, by removing the edges in the circuits $L_i$ in both $\Lambda$ and $\Gamma$. Suppose that the components and tags for $\overline{\Lambda}$ and $\overline{\Gamma}$ are equal. 
		\begin{enumerate}
		\item Let $\zeta$ be the graph obtained by applying an RBS move of type C to $\Lambda$, with $\overline{\zeta}$ obtained from $\zeta$ by removing the edges of circuits $L_1, \ldots, L_E$. Then the tags and components of $\overline{\Lambda}$ and $\overline{\zeta}$ are the same.
		\item Similarly let $\chi$ and $\xi$ be the graphs obtained from $\Lambda$ and $\Gamma$ by applying the same move of type A or B to both (to the same vertices in the circuit), with $\overline{\chi}$ and $\overline{\xi}$ the corresponding graphs with edges of $L_1, \ldots, L_E$ removed. Then the tags and components of $\overline{\chi}$ and $\overline{\xi}$ are the same.
		\end{enumerate}
		\end{coro}

		Given these preliminaries, we now continue with the original connectedness claim for the graph $\Xi$ (the proof of Proposition~\ref{PROP_CONNECT}). First we recall that the graph $\Xi$ from Section~\ref{SEC_NLOOPS_NEWG} was constructed as follows. We first extract all the shrink and twist moves from the list $\mathsf{L}$, taken from the itinerary, as $m_1, \ldots, m_T$. These moves act on the circuits $L_1, \dots, L_E$. We first apply all these moves to $\Lambda$ to get a new graph $\Lambda^*$ with corresponding circuits $L_1^*, \dots, L_E^*$. We then remove all edges in the union of these circuits (and make all other edges undirected) to obtain a graph $\overline{\Lambda^*}$. 
		Then we identify all right special vertices in a loop $L_i^*$ in $\overline{\Lambda^*}$ to one vertex $\nu_{r}^i$ and all left special vertices in $L_i^*$ to one vertex $\nu_\ell^i$. The resulting graph is called $\Xi$.
		We will show here that $\Xi$ is weakly connected. As stated in Corollary~\ref{cor: main_corollary}, this implies that $E \leq (K+1)/2$.
		
		So suppose that $\Xi$ is not weakly connected. We choose a maximal set of pairs $P_i = (\nu_\ell^i, \nu_r^i)$ as follows. Note that if, for all $i$, we add an edge $e_i$ between $\nu_\ell^i$ and $\nu_r^i$, then $\Xi$ becomes weakly connected. So choose a maximal set $T$ (possibly empty) of pairs $P_i$ such that if we add all the associated $e_i$'s into $\Xi$, then $\Xi$ is still not connected. Now let $S$ be the set of pairs not in $T$, and note that $S$ is nonempty. For simplicity, we will write
		\[
		S = \{P_1, \ldots, P_q\} \text{ and } T = \{P_{r+1}, \ldots, P_E\} \text{ for } q \in [1,E].
		\]
		Furthermore, adding the edges $e_i$ into $\Xi$ for $i=q+1, \ldots, E$ produces a graph $\hat \Xi$ which has exactly two components, and
		\begin{equation}\label{eq: xi_hat_disconnect}
		\text{for }i=1, \ldots, q,~\nu_\ell^i \text{ and } \nu_r^i \text{ are in different components of }\hat \Xi.
		\end{equation}
		Now let $I$ be the value of $i$ such that the circuit $L_i$ is the first of $L_1, \ldots, L_q$ (corresponding to pairs $P_1, \ldots, P_q$) to undergo a spread event in the itinerary. (Such $i$ may not be unique.) We then let $\hat \Lambda$ be the first graph from the itinerary in which $L_i$ undergoes a spread event, According to the definition of a spread event and Rules List~\ref{RULE1}(4), in this graph $\hat \Lambda$, there must be a directed path $\pi$ from a right special vertex $x_I$ on $L_I$ to a left special vertex $y_I$ on $L_I$ such that this path lies entirely outside $L_I$, and each vertex of $\pi$ has the same color as that of the circuit $L_I$ in $\hat \Lambda$. We claim that 
		\begin{equation}\label{eq: to_show}
		\pi \text{ intersects } L_j \text{ for some }j \neq I \text{ and }j \in [1,q].
		\end{equation}
		As in the proof of Lemma~\ref{LEM2LOOPS}, this will give a contradiction to the coloring rules, since in $\hat \Lambda$, $L_j$ has a different color than does $L_I$.
		
		To show \eqref{eq: to_show}, we need some more constructions. In what follows, we will use the following notation: if $\chi$ is a graph with the loops $L_1, \dots, L_E$, then $\overline{\chi}$ is the graph obtained by removing the edges in these loops, and $\overline{\chi_0}$ is the graph obtained by removing the edges in $L_1, \dots, L_q$.
		In the graph $\Xi$, $\nu_\ell^I$ and $\nu_r^I$ are in different components, and we have seen that when we add the edges $e_i$ for $i = r+1, \ldots, E$ into $\Xi$ to produce $\hat \Xi$, then $\nu_\ell^I$ and $\nu_r^I$ are still in separate components (they are in the two different components of $\hat \Xi$). Now we produce a second graph $\overline{\Lambda_0^*}$ by expanding the pairs $P_i$ for $i = q+1, \ldots, E$ into circuits $L_{q+1}, \ldots, L_E$ containing all their edges, and re-expanding the pairs $P_i$ for $i=1, \ldots, q$ into circuits $L_1, \ldots, L_q$ with their edges removed. Equivalently, $\overline{\Lambda_0^*}$ is the graph $\Lambda^*$ with the edges in $L_1, \ldots, L_q$ removed. Then directly from \eqref{eq: xi_hat_disconnect}, we obtain
		\begin{equation}\label{eq: fact_1}
		\begin{array}{c}
		\text{ in }\overline{\Lambda^*_0}, \text{ for each } i = 1, \ldots, q,\text{ each right special vertex of }L_i \\
		\text{ is in a different component than each left special vertex of }L_i.
		\end{array}
		\end{equation}	 		 
		
		Given \eqref{eq: fact_1}, the next step in proving \eqref{eq: to_show} is to define two sequences of graphs. First list
		\[
		\Lambda = \zeta^{(0)} , \ldots, \zeta^{(T)} = \Lambda^*
		\]
		as the sequence of graphs obtained by applying the shrink and twist moves $m_1, \ldots, m_T$ in order to $\Lambda$. Write
		\[
		\overline{\Lambda_0} = \overline{\zeta_0^{(0)}} , \ldots, \overline{\zeta_0^{(T)}} = \overline{\Lambda_0^*}
		\]
		as the sequence of graphs obtained from $\zeta^{(0)}, \ldots, \zeta^{(T)}$ by removing the edges in the circuits $L_1, \ldots, L_q$. We next claim that
		\begin{equation}\label{eq: fact_2}
		\text{for }1\leq i \leq q, \text{ and }  0 \leq t \leq T, ~x_i \text{ and }y_i \text{ are in separate components in }\overline{\zeta_0^{(t)}}.
		\end{equation}
		Here, $x_i$ and $y_i$ are defined similarly to $x_I$ and $y_I$:
			they are left and right special vertices in $L_i$ through which $L_i$ spreads color during its first spread event.
		To show \eqref{eq: fact_2}, we argue inductively,
			backward on $t$.
		Suppose that for some $t \in [1, p]$,
			one has $x_i$ and $y_i$ in separate components for all $i = 1, \ldots, q$,
			in $\overline{\zeta_0^{(t)}}$ (it holds by \eqref{eq: fact_1} for $t=T$).
		Now apply Proposition~\ref{prop: connected_moves} with $\Lambda = \zeta_{t-1}$ and $\zeta = \zeta_t$ (removing loops $L_1, \dots, L_q$).
		If the move $m_{t-1}$ is of type A then the components and tags of
			$\overline{\zeta_0^{(t-1)}}$ and $\overline{\zeta_0^{(t)}}$ are the same.
		In particular, for any $i=1, \ldots, q$, the vertices $x_i$ and $y_i$ are in separate components in $\overline{\zeta_0^{(t-1)}}$.
		If the move is of type B (which is the only other possibility),
			then two components from $\overline{\zeta_0^{(t-1)}}$ may have merged to form a single component in $\overline{\zeta_0^{(t)}}$.
		But since for each $i=1, \ldots, q$, $x_i$ and $y_i$ were in separate components in $\overline{\zeta_0^{(t)}}$, they must also then be in $\overline{\zeta_0^{(t-1)}}$. This shows \eqref{eq: fact_2}.

		Next we list all shrink and twist moves $m_1, \ldots, m_\kappa$
			from the full list $m_1,\ldots,m_T$ that was taken from $\mathsf{L}$ in the itinerary
			from the original graph $\Lambda$ until we reach $\hat \Lambda$ (the graph obtained from $\Lambda$ by applying all moves --- not just shrink and twist moves --- from the itinerary, until the loop $L_I$ spreads its color).
		Now we must relate $\overline{\zeta_0^{(\kappa)}}$ to $\hat \Lambda$. To do so, we need two more sequences of graphs. List all moves $M_1, \ldots, M_s$  performed from $\Lambda$ to $\hat \Lambda$ from the full list $\mathsf{L}$, and apply these to the graph $\Lambda$. In other words, we define a sequence of graphs
		\[
		\Lambda = \Delta^{(0)} , \ldots, \Delta^{(s)} = \hat \Lambda
		\]
		obtained from $\Lambda$ so that $\Delta^{(j)}$ is the result of applying moves $M_1, \ldots, M_j$ to $\Lambda$. As before, we also set
		\[
		\overline{\Lambda_0} = \overline{\Delta_0^{(0)}}, \ldots, \overline{\Delta_0^{(s)}} = \overline{\hat \Lambda_0}
		\]
		as the sequence obtained from $\Delta^{(0)}, \ldots, \Delta^{(s)}$ by removing the edges from the circuits $L_1, \ldots, L_q$. To synchronize these graphs to the sequences $(\zeta^{(i)})$ and $(\overline{\zeta_0^{(i)}})$, we note that the shrink and twist moves $m_1, \ldots, m_T$ are in the list $\mathsf{L}$, whose moves are $M_1, \ldots, M_s$, so we write $k_0=0$, and set $k_j$ so that
		\[
		m_j = M_{k_j} \text{ for } j = 1, \ldots, \kappa,
		\]
		where $k_\kappa \leq s$.
		
		Given these definitions, we finally claim that 
		\begin{equation}\label{eq: last_claim}
		\text{for } 1 \leq i \leq q, x_i \text{ and } y_i \text{ are in separate components in } \overline{\Delta_0^{(s)}} = \overline{\hat \Lambda_0}.
		\end{equation}
		Applying this claim to $i=I$ shows that in particular $x_I$ and $y_I$ are in separate components in $\overline{\Delta_0^{(s)}} = \overline{\hat \Lambda_0}$, and so the path $\pi$ from \eqref{eq: to_show} must intersect a circuit $L_j$ for $1 \leq j \leq q$ and $j \neq I$, showing \eqref{eq: to_show}.
		
		To show \eqref{eq: last_claim}, we first prove that 
		\begin{equation}\label{eq: component_tag}
		 \text{for }j=0, \ldots \kappa \text{, the components and tags of } \overline{\Delta_0^{(k_j)}} \text{ are the same as those of }\overline{\zeta_0^{(j)}}.
		 \end{equation} 
		 This is true for $j=0$, since $\overline{\Delta_0^{(k_0)}} = \overline{\Lambda_0} = \overline{\zeta_0^{(0)}}$.
		 Assuming this holds for some $j=0, \ldots, \kappa-1$,
		 	we use Corollary~\ref{cor: connected_moves} (item 1) with $\Lambda = \Delta^{(k_j)}$,
		 	$\Gamma = \zeta^{(j)}$, and removing circuits $L_1, \dots, L_q$.
		 To obtain $\Delta_0^{(k_j+1)}$ from $\Delta^{(k_j)}$,
		 	we apply a move of type C, so we set $\zeta = \Delta^{(k_j+1)}$.
		 We deduce then that components and tags of $\overline{\Delta_0^{(k_j+1)}}$
		 	are the same as those of $\overline{\zeta_0^{(j)}}$.
		 We now repeat this argument, at the $a$-th time, applying the corollary with $\Lambda = \Delta^{(k_j+a-1)}$, $\Gamma = \zeta^{(j)}$,
		 	and $\zeta = \Delta^{(k_j+a)}$,
		 	and so on,
		 	to eventually find that the tags and components of $\overline{\Delta_0^{(k_{j+1}-1)}}$
		 	are the same as those of $\overline{\zeta_0^{(j)}}$.
		 To move from $\Delta^{(k_{j+1}-1)}$ to $\Delta^{(k_{j+1})}$
		 	(and $\zeta^{(j)}$ to $\zeta^{(j+1)}$),
		 		we must use the move $m_{j+1}$, which is of type A or B.
		 We therefore apply Corollary~\ref{cor: connected_moves} (item 2)
		 	with $\Lambda = \Delta^{(k_{j+1}-1)}$, $\Gamma = \zeta^{(j)}$, $\chi = \Delta^{(k_{j+1})}$, and $\xi = \zeta^{(j+1)}$.
		 We then deduce that the tags and components of $\overline{\Delta_0^{(k_{j+1})}}$ are the same as those
		 	in $\overline{\zeta_0^{(j+1)}}$.
		 By induction, this verifies \eqref{eq: component_tag}.
		 
		Since the components and tags of $\overline{\Delta_0^{(k_\kappa)}}$
			are the same as those of $\overline{\zeta_0^{(\kappa)}}$,
			we apply the previous paragraph's argument a few more times to compare the components and tags for the graphs
		\[
		\overline{\Delta_0^{(k_\kappa)}}, \overline{\Delta_0^{(k_\kappa+1)}}, \ldots, \overline{\Delta_0^{(s)}}
		\]
		to those of $\overline{\zeta_0^{(\kappa)}}$.
		Since the only moves involved in transitioning between these graphs are of type C,
			we can again use Corollary~\ref{cor: connected_moves} to say that all of these graphs have the same components and tags.
		In particular, $\overline{\Delta_0^{(s)}}$ and $\overline{\zeta_0^{(\kappa)}}$ have the same components and tags.
		Using \eqref{eq: fact_2} with $t=\kappa$,
			we find that for $1 \leq i \leq q$,
			$x_i$ and $y_i$ are in separate components in $\overline{\Lambda_0^{(s)}}$,
			which is \eqref{eq: last_claim}, and this completes the proof.
		
\bibliographystyle{alpha}
\bibliography{BIBTEX}

\end{document}